\documentclass[11pt]{amsart} 
\usepackage{amssymb,amscd,eucal}
\usepackage[all]{xy}
\usepackage{color}
\usepackage{cite,hyperref}
\usepackage{mathbbol}
\usepackage{latexsym}
\usepackage{appendix}
\usepackage{tikz}
\usetikzlibrary{decorations.markings}
\usetikzlibrary{decorations.pathreplacing}
\usetikzlibrary{arrows,shapes,positioning}
\tikzstyle directed=[postaction={decorate,decoration={markings,
    mark=at position #1 with {\arrow{>}}}}]
\tikzstyle rdirected=[postaction={decorate,decoration={markings,
    mark=at position #1 with {\arrow{<}}}}]
\usepackage{youngtab}
\usepackage{ytableau}
\usepackage[margin=1in]{geometry}
\usepackage{tensor,enumitem,verbatim}

\numberwithin{equation}{subsection}
\newtheorem{theorem}{Theorem}[subsection]
\newtheorem{proposition}[theorem]{Proposition}
\newtheorem{corollary}[theorem]{Corollary}
\newtheorem{lemma}[theorem]{Lemma}

\theoremstyle{definition}
\newtheorem{definition}[theorem]{Definition}
\newtheorem{example}[theorem]{Example}
\newtheorem{remarks}[theorem]{Remarks}
\newtheorem{remark}[theorem]{Remark}

\newcommand{\ec}{\mathcal{E}}
\newcommand{\fc}{\mathcal{F}}
\newcommand{\gc}{\mathcal{G}}

\newcommand{\Af}{\mathsf{A}}
\newcommand{\Cf}{\mathsf{C}}
\newcommand{\Df}{\mathsf{D}}
\newcommand{\Nf}{\mathsf{N}}
\newcommand{\Pf}{\mathsf{P}} 
\newcommand{\Qf}{\mathsf{Q}}
\newcommand{\Hf}{\mathsf{H}}
\newcommand{\Sf}{\mathsf{S}}

\newcommand{\Zr}{\mathrm{Z}}
\newcommand{\Dr}{\mathrm{D}}
\newcommand{\Ir}{\mathrm{I}}
\newcommand{\vb}{\mathbf{v}}
\newcommand{\xb}{\mathbf{x}}
\newcommand{\diag}{\mathsf{diag}}
\newcommand{\dimm}{\mathsf{dim}}
\newcommand{\modd}{\mathsf{mod\,}}
\newcommand{\UU}{\mathsf{U}}
\newcommand{\Wc}{\mathcal{W}}
\newcommand{\Uc}{\mathcal{U}}
\newcommand{\Vc}{\mathcal{V}}
\newcommand{\Lc}{\mathcal{L}}
\newcommand{\WW}{\mathsf{W}}
\newcommand{\VV}{\mathsf{V}}
\newcommand{\XX}{\mathsf{X}}
\newcommand{\YY}{\mathsf{Y}}
\newcommand{\CC}{\mathbb{C}}
\newcommand{\ZZ}{\mathbb{Z}}
\newcommand{\GG}{\mathsf{G}}
\newcommand{\Kf}{\mathsf{K}}
\newcommand{\McV}{\mathsf{M_{V}}}
\newcommand{\Mf}{\mathsf{M}}
\newcommand{\QV}{\mathsf{Q_{V}}}
\newcommand{\QVT}{\mathsf{Q_{V}^{\tt T}}}

\newcommand{\Tf}{\mathsf{T}}

\newcommand{\ot}{\otimes}
\newcommand{\half}{\frac{1}{2}}
\newcommand{\tr}{ \mathsf{tr}}

\newcommand{\Trs}{\operatorname{\mathsf{Tr}_\Sf}}
\newcommand{\Trp}{\operatorname{\mathsf{Tr}_\Pf}}
\newcommand{\cto}{\buildrel {\color{black} \mathsf{c}} \over  \longrightarrow} 
\newcommand{\isit}{\buildrel {\color{black} \mathbf{?}} \over  =} 
 
\title{McKay Matrices for Finite-dimensional Hopf Algebras}
 
\author[Benkart]{Georgia Benkart} 
\address[G.~Benkart]{Department of Mathematics, University of Wisconsin-Madison, Madison, WI 53706, U.S.A.}
\email{benkart@math.wisc.edu}

\author[Biswal]{Rekha Biswal}
\address[R.~Biswal]{Max Planck Institute for Mathematics, Bonn, Germany}
\email{rekhabiswal27@gmail.com} 

\author[Kirkman]{Ellen Kirkman}
\address[E.~Kirkman]{Department of Mathematics and Statistics, Wake Forest University, Winston-Salem, NC, 27109, U.S.A.} 
\email{kirkman@wfu.edu} 

\author[Nguyen]{Van C.~Nguyen}
\address[V.C.~Nguyen]{Department of Mathematics, United States Naval Academy, Annapolis, MD 21402, U.S.A.}
\email{vnguyen@usna.edu} 

\author[Zhu]{Jieru Zhu}
\address[J.~Zhu]{Department of Mathematics, University at Buffalo, 
Buffalo, NY 14260-2900, U.S.A.}
\email{jieruzhu699@gmail.com} 
\date{}

\allowdisplaybreaks

\makeatletter
\@namedef{subjclassname@2020}{%
  \textup{2020} Mathematics Subject Classification}
\makeatother

\begin{document}


\begin{abstract}
For a finite-dimensional Hopf algebra $\mathsf{A}$, the McKay matrix $\McV$ of an $\Af$-module
$\VV$ encodes the relations for tensoring the simple $\Af$-modules with $\VV$.    We prove results about  
the eigenvalues and the right and left (generalized) eigenvectors of $\McV$ by relating them to
characters. We show how the projective McKay matrix $\Qf_\VV$  obtained by tensoring the projective indecomposable modules
of $\Af$ with $\VV$ is related  to the McKay matrix of the dual module of $\VV$.   We illustrate these results for the Drinfeld double
$\Df_n$ of  the Taft algebra by deriving
expressions for the eigenvalues and eigenvectors  of $\McV$ and $\Qf_\VV$ in terms of several kinds of Chebyshev polynomials.
For the matrix $\Nf_\VV$ that encodes the fusion rules for tensoring $\VV$ with a basis of projective indecomposable 
$\Df_n$-modules for the image of the Cartan map, we show
that the eigenvalues and eigenvectors also have such Chebyshev expressions. 
\end{abstract}

\maketitle

\noindent \section{\bf Introduction}  

Assume $\Af$ is a finite-dimensional associative algebra over an algebraically closed field $\mathbb k$. Let $\Sf_1,\Sf_2, \dots, \Sf_m$
be the nonisomorphic simple (irreducible)
$\Af$-modules and   $\Pf_1,\Pf_2,\dots, \Pf_m$ be  their projective covers, that is, the nonisomorphic indecomposable projective modules  such that for each $j=1,\dots, m$, the module
$\Pf_j$ modulo its radical is isomorphic to $\Sf_j$.   The dimensions of these $\Af$-modules determine two column vectors
in $\ZZ^m$,
\begin{equation}\label{eq:sandp}  \mathbf{s} = [\dimm(\Sf_1)\,\;\dimm(\Sf_2)\,\; \dots\,\; \dimm(\Sf_m)]^{\tt T} \quad \text{and} \quad
\mathbf{p} = [\dimm(\Pf_1)\,\;\dimm(\Pf_2)\,\; \dots\,\; \dimm(\Pf_m)]^{\tt T},\end{equation}
where ${\tt T}$ denotes ``transpose.''   If $\Af$ is viewed as a left $\Af$-module under left multiplication,  then 
\begin{equation}\label{eq:Adecomp} \Af = \bigoplus_{j=1}^m \Pf_{j}^{\oplus \dimm(\Sf_{j})} \quad \text{and} 
\quad \mathbf{p}^{\tt T}  \, \mathbf{s} = \sum_{j=1}^{m}  \dimm(\Pf_j) \ \dimm(\Sf_j)  =  \dimm(\Af).
\end{equation}
When $\Af$ is semisimple,  then $\Pf_j = \Sf_j$ for all $j$, and  the second part of \eqref{eq:Adecomp} is the familiar 
result $\dimm (\Af) = \sum_{j=1}^m \left(\dimm(\Sf_j)\right)^2$.

There are two Grothendieck groups, $\GG_0(\Af)$ and $\Kf_0(\Af)$ associated to $\Af$:
\begin{itemize} 
\item  $\GG_0(\Af)$ is the quotient of the free abelian group on the set of all isomorphism
classes $[\VV]$ of finite-dimensional $\Af$-modules $\VV$ subject to the relations $[\UU]- [\VV] + [\WW] = 0$ for each short exact sequence
$0 \rightarrow \UU \rightarrow \VV \rightarrow \WW \rightarrow 0$ of $\Af$-modules. By the Jordan-H\"older theorem, this group has a $\ZZ$-module basis consisting of the classes
$[\Sf_1], [\Sf_2],. . . , [\Sf_m]$,  and  
$$[\Af] = \sum_{j=1}^m\dimm(\Pf_j) [\Sf_j]  \quad \text{in} \; \; \GG_0(\Af).$$
\smallskip

\item $\Kf_0(\Af)$ is the quotient of the free abelian group on the set of all
isomorphism classes $[\VV]$ of finite-dimensional projective $\Af$-modules $\VV$, subject to the relations $[\UU] - [\VV]+[\WW] = 0$ for each direct sum decomposition $\VV = \UU \oplus \WW$ of $\Af$-modules. This group has a $\ZZ$-module basis consisting of
the classes $[\Pf_1], [\Pf_2],\dots, [\Pf_m]$  due to the Krull-Remak-Schmidt theorem, and \eqref{eq:Adecomp} says
$$[\Af] = \sum_{j=1}^m \dimm(\Sf_j) [\Pf_j] \quad \text{in} \; \; \Kf_0(\Af).$$
\end{itemize}

When $\Af$ is a Hopf algebra,  the tensor product of two $\Af$-modules is an $\Af$-module with the $\Af$-action given by the coproduct, and the dual vector space of an $\Af$-module  is an $\Af$-module via the antipode.
Both $\GG_0(\Af)$ and $\Kf_0(\Af)$ have products using $\ot$, so that $[\UU] [\WW] = [\UU \ot \WW]$, where
$\GG_0(\Af)$ is a ring with a unit element, which is the one-dimensional $\Af$-module $\mathbb{k}$ with action given by 
the counit, and  $\Kf_0(\Af)$ is a ring without a unit element. 
 
 Let $\VV$ be an $\Af$-module, and set $d = \dimm(\VV)$.     The \emph{McKay matrix} $\McV$  for tensoring with $\VV$ has as its
 $(i,j)$ entry $\Mf_{ij} =[\Sf_i \ot \VV: \Sf_j]$, the multiplicity of $\Sf_j$ as a composition factor of the $\Af$-module $\Sf_i \ot \VV$, or
 equivalently, the coefficient of $[\Sf_j]$ when $[\Sf_i \ot \VV]$ is expressed as a $\ZZ$-linear combination of the basis elements in $\GG_0(\Af)$.
 Then 
\begin{equation*} d \, \mathbf{s}_i  = \dimm(\VV)  \dimm(\Sf_i) = \dimm(\Sf_i \ot \VV) 
= \sum_{j=1}^m
[\Sf_i \ot \VV : \Sf_j] \dimm(\Sf_j) = \sum_{j=1}^m \mathsf{M}_{ij} \dimm(\Sf_j)  =  (\McV\,\mathbf s)_i\end{equation*}
shows that the $i$th entry of $\McV\,\mathbf s$ is $d$ times the $i$th entry of $\mathbf s$ for $1\le i \le m$,  which implies that $\mathbf s$ is a right eigenvector of $\McV$ for the eigenvalue $d = \dimm(\VV)$.

One can argue as in the paper \cite{GHR} by Grinberg, Huang, and Reiner that the dimension vector  $\mathbf{p}^{\tt T} = [\dimm(\Pf_1)\,\; \dimm(\Pf_2)\,\; \ldots\,\; \dimm(\Pf_m)]$ is a left eigenvector for $\McV$ with eigenvalue $d$.   (But note that their McKay matrix and ours are transposes of one another, so for them $\mathbf{p}^{\tt T}$ is a right eigenvector, and $\mathbf{s}$ is a left one.)

The matrix $\McV$ can be viewed as the adjacency matrix of a quiver (the so-called \emph{McKay quiver determined by $\VV$}) having nodes  labeled by $1 \le i \le m$  that correspond to
the simple modules $\Sf_i$.  There are $[\Sf_i \ot \VV : \Sf_j]$ arrows from $i$ to $j$.   If there is an arrow from $i$ to
$j$ and one from $j$ to $i$, they  are  replaced by a single undirected edge.    The idea to consider such a quiver and matrix goes back to 
McKay's insight \cite{M} that the quivers determined
by tensoring with the $\GG$-module  $\VV = \mathbb{C}^2$ for a
 finite subgroup $\GG$  of $\mathsf{SU}_2$ exactly correspond to the affine Dynkin diagrams of types $\Af, \Df, \mathsf{E}$.  This result, 
subsequently referred to as the \emph{McKay correspondence},  has been the inspiration for much work on a host of topics in singularity theory,
group theory, orbifolds, and many other subjects.  

In expanding on McKay's result,  Steinberg \cite{St} showed that for the group algebra
$\mathbb{C} \GG$ of any finite group $\GG$, the columns of the character table of $\GG$  give a complete set of right eigenvectors for the McKay matrix determined by tensoring with \emph{any} finite-dimensional $\GG$-module $\VV$,
and the associated eigenvalues are given by the character values $\chi_{{}_\VV}(g_j)$, as $g_j$ ranges over a set of conjugacy class representatives of  $\GG$.   Group algebras over algebraically closed fields of characteristic $p > 0$ were considered in
  \cite{GHR}, where it was shown that the columns of the Brauer character table of $\GG$ are right eigenvectors of $\McV$ with corresponding
eigenvalues $\chi_{{}_\VV}(g_j)$, as $g_j$ ranges over a set of representatives for the $p'$ conjugacy
classes of  $\GG$ (that is, elements whose order is relatively prime to $p$).  Each  such $g_j$ determines a left eigenvector  
whose coordinates are  the character values of the projective indecomposable modules $\Pf_j$ evaluated at $g_j$, and whose
eigenvalue is $\chi_{{}_\VV}(g_j)$.

Aside from the result of Grinberg, Huang, and Reiner in \cite{GHR} mentioned earlier,
very little is known about the eigenvalues and eigenvectors of the McKay matrices for \emph{arbitrary} finite-dimensional Hopf algebras,
and the purpose of this paper is to remedy that situation.
The specific case of the quantum group $\mathfrak{u}_q(\mathfrak{sl}_2)$ at $q$ an odd root of unity  was examined
in depth in \cite{BDLT}.   In that work,  it was shown that the McKay matrix for tensoring with $\VV = \mathbb{C}^2$ has two-dimensional
generalized eigenspaces for every eigenvalue different from  $2 = \dimm(\VV)$, and  it is necessary to work with Jordan blocks of size
$2 \times 2$ in that example.    The main point of  \cite{BDLT} is that McKay matrices and quivers determine interesting
Markov chains, and in the particular case of $\mathfrak{u}_q(\mathfrak{sl}_2)$, the chain exhibits new phenomena due  to the
existence of such Jordan blocks.   The examples in \cite{BDLT} show that the analysis of the rates of convergence of the Markov chains determined by McKay matrices is quite delicate.

Most of the work on characters of Hopf algebras has been in the case that the algebra is semisimple. For example, Witherspoon defined a notion of a character table for finite-dimensional semisimple, almost cocommutative Hopf algebras and showed
 that the characters provide eigenvectors for the McKay matrix for tensoring with the simple modules (see \cite[proof of Thm.~3.2]{W}). 
In\cite{CW3}, Cohen and Westreich determined Verlinde formulas for semisimple, almost cocommutative Hopf algebras, and
in \cite{CW1, CW2},  for (nonsemisimple) factorizable ribbon Hopf algebras such as the Drinfeld double of the Taft algebra considered here.  Such formulas
were introduced by Verlinde \cite{V}  for diagonalizing fusion relations in 2D rational conformal field theory
and have played an important role in physics. They have been considered subsequently in many different contexts. 
 
For a Hopf algebra $\Af$ with nonisomorphic simple modules $\Sf_i$ and  
corresponding  projective covers $\Pf_i$,  the Cartan matrix $\Cf = (\Cf_{ij})$ records the multiplicity $\Cf_{ij}
= [\Pf_i:\Sf_j]$ of $[\Sf_j]$ when $[\Pf_i]$ is expressed in the $\ZZ$-basis of $\GG_0(\Af)$.
 Let $r$ be the rank of
$\Cf$ and select $r$ of the $\Pf_i$ so that the corresponding rows of $\Cf$ are linearly independent.  For each simple module $\Sf_j$
there is an $r \times r$ matrix $\Nf^j$, that contains the fusion rules for tensoring those $r$ projective modules with $\Sf_j$  and
writing the answer $[\Pf_i \ot \Sf_j]$ as a linear combination of the $r$  chosen projectives.  
It was shown in \cite{CW1} that the matrices $\Nf^j$ are diagonalizable and have eigenvectors that can be expressed
using the primitive idempotents $e_i$ corresponding to the projective modules $\Pf_i$.  

In Section 2 of this work,  we consider \emph{arbitrary} finite-dimensional Hopf algebras $\Af$
and prove general results about McKay matrices, their eigenvalues,  and their (left and right) eigenvectors by using the coproduct  and the characters of simple and projective modules (Theorem
\ref{Thm:tracesimp} and Proposition \ref{prop:reigen}).   The tensor product $\Pf \ot \VV$  of a projective module $\Pf$ with a finite-dimensional module $\VV$  is projective, and so $[\Pf \ot \VV]$ can be written as an integral combination of the classes $[\Pf_j]$ of the
projective indecomposable modules.  Letting $\Qf_\VV = (\Qf_{ij})$,   where $\Qf_{ij}$ is the multiplicity
$[\Pf_i \ot \VV: \Pf_j]$ of $[\Pf_j]$ in $[\Pf_i \ot \VV]$, we obtain what we term a \emph{projective McKay matrix}.  Theorems
\ref{thm:QV} and \ref{thm:projevec} and Corollary \ref{cor:V=V**} show how $\Qf_\VV$ and its
eigenvectors are related to the McKay  matrix of the dual module $\VV^*.$   
In the special case that the Hopf algebra is semisimple, the McKay matrix $\McV$ and the projective
McKay matrix $\Qf_\VV$ are the same,  and if the module 
$\VV$ is self-dual,  then the McKay matrix is orthogonally diagonalizable (Corollary \ref{cor:Q}). 

In Section 3, we  illustrate the general 
results of Section 2 by applying them to a family of nonsemisimple Hopf algebras, namely,  the Drinfeld double $\Df_n$ of the Taft algebra 
$\mathsf{A}_n$ for $n$ odd, $n \ge 3$.  When $n$ is even, the eigenvalues exhibit different
patterns, and that case will not be considered here.
The algebras $\Df_n$ provide a convenient testing ground,  as their representation theory 
has been developed in great detail by Chen and coauthors (see \cite{Chen99, Chen2, Chen2.5, Chen2.75, Chen, Chen3}). 
Unlike the situation for semisimple, almost cocommutative Hopf algebras, the
McKay matrices for $\Df_n$ fail to be diagonalizable.     More specifically, we
\begin{itemize}
\item determine the eigenvalues, right and left eigenvectors and generalized eigenvectors of the McKay matrix $\McV$  obtained  by tensoring the simple $\Df_n$-modules with one of the two-dimensional simple $\Df_n$-modules,  $\VV(2,0)$ (Secs. \ref{S3.3}-\ref{S3.9}); 
\item  express the coordinates of  these vectors using Chebyshev polynomials (Secs. \ref{S3.4}, \ref{S3.5}, \ref{S3.8});
\item  relate the eigenvectors  to character values of the grouplike elements and other special elements of $\Df_n$ (Secs. \ref{S3.6}, \ref{S3.7}, \ref{S3.9}) and show that the character value of a grouplike element on any simple $\Df_n$-module can be
computed using Chebyshev polynomials (Thm. \ref{thm:rightevs} of  Sec.~\ref{S3.6};
\item prove that the (generalized) eigenvectors for $\McV$, $\VV= \VV(2,0)$, are (generalized) eigenvectors for the McKay matrix 
of \emph{any} simple $\Df_n$-module
(Sec. \ref{S3.10});
\item  show that the eigenvalues of the McKay matrix of any simple $\Df_n$-module can be expressed in terms of Chebyshev polynomials of the second kind (Thm. \ref{thm:Ml0evals} of Sec. \ref{S3.11});
\item find the eigenvectors and eigenvalues of the projective McKay matrix $\Qf_\VV$ by relating them
to the McKay matrix of the dual module $\VV^*$ (Prop. \ref{prop:Ql0evals} of Sec. \ref{S3.12});
\item determine the structure of the complex Grothendieck algebra $\GG_0^{\CC}(\Df_n) = \CC \ot_{\ZZ} \GG_0(\Df_n)$ and prove that  its Jacobson radical squares to 0; (Thm. \ref{thm:Groth} of Sec. \ref{S3.13});
\item construct certain idempotents  in $\GG_0^{\CC}(\Df_n)$,  and 
show how they provide an alternate approach to producing the eigenvectors and generalized eigenvectors of $\McV$ (Sec. \ref{S3.13});
\item compute the eigenvectors and eigenvalues of the matrix $\Nf_\VV$ that encodes the fusion rules for tensoring
a maximal set  of independent projective covers in $\GG_0(\Df_n)$ with 
$\VV = \VV(2,0)$  (Sec. \ref{S3.14}). 
\end{itemize}

It has been said that ``Chebyshev polynomials are everywhere dense in
numerical analysis'' (see \cite[Sec.~1.1]{MH} for a discussion of this quotation).   In this paper,  Chebyshev polynomials (of the
second, third, and fourth kind)
are everywhere dense in
expressing eigenvalues, eigenvectors, and generalized  eigenvectors of McKay matrices and fusion rules for $\Df_n$.  
The characters of the simple $\Df_n$-modules evaluated on the grouplike elements of $\Df_n$ also have
Chebyshev polynomial expressions. 

When $n$ is odd,  $\Df_n$ is a ribbon Hopf algebra  \cite{KR},   and $\Df_n$ provides an invariant of 3-manifolds \cite{H}.  
In  \cite{BBKNZ},  we determine the unique ribbon element of
$\Df_n$ explicitly.  We use the $\mathsf{R}$-matrix and ribbon element of the quasitriangular Hopf algebra $\Df_n$ to obtain an algebra homomorphism  from the Temperley-Lieb algebra $\mathsf{TL}_k(-(q^{\half}+q^{-\half}))$ to the centralizer algebra
$\mathsf{End}_{\Df_n}(\VV^{\ot k})$, 
when $\VV$ is any two-dimensional simple $\Df_n$-module.   In the special case that $\VV$ is
the unique two-dimensional simple module that is self-dual,  we show that the homomorphism is
injective for all $k \ge 1$ and an isomorphism for $k \le 2(n-1)$. This leads to
a realization of $\mathsf{End}_{\Df_n}(\VV^{\ot k})$ for $\VV$ self-dual as a diagram algebra.

\subsection*{Acknowledgments}  Our joint work began in connection with  the workshop WINART2 at  the University of Leeds in May 2019.  The authors would like to extend thanks 
to the organizers of WINART2 and  to the University of Leeds for its hospitality, and to acknowledge support
provided by a University of Leeds conference grant,  the London Mathematical Society Workshop Grant WS-1718-03,  the U.S. National Science Foundation grant DMS-1900575, the Association for Women in Mathematics (NSF Grant DMS-1500481), and by a research fellowship from the Alfred P. Sloan Foundation.  Van C. Nguyen was also supported by the Naval Academy Research Council in summer 2020.  Rekha Biswal gratefully acknowledges the Max Planck Institute for the fellowship she received as a postdoctoral researcher there
and for providing an excellent atmosphere for research. The authors thank the referees for their helpful comments.

\noindent \section{\bf McKay Matrices for Arbitrary Hopf Algebras}

In the classical representation theory of finite groups, the columns of the character table are obtained by evaluating the characters (traces) of the simple modules on one element from each conjugacy class \cite[Sec.~2.1]{FH}. It is known that these columns are right eigenvectors for any McKay matrix determined by tensoring with a finite-dimensional module of the group (see \cite{St}).  In this section, we develop an analog of this result by showing how grouplike elements and generalizations of skew primitive elements in an arbitrary Hopf algebra can be used to construct eigenvectors for McKay matrices.  In the special case that the algebra is semisimple, more detailed results are possible. 

\emph{Throughout Section 2, $\Af$ denotes a finite-dimensional Hopf algebra over an algebraically closed field $\mathbb{k}$.
All $\Af$-modules are assumed to be finite-dimensional, and all tensor products are over $\mathbb{k}$. We adopt Sweedler's
notation for the coproduct $\Delta$ applied to an element $x \in \Af$,}
$$\Delta(x) = \sum_{x} x_{(1)} \ot x_{(2)}.$$

\noindent \subsection{Right eigenvectors from traces of simple modules}
\label{S2.1}
We will use the coproduct of $\Af$ to obtain right (generalized) eigenvectors of McKay matrices by
taking the trace of certain elements of $\Af$ on simple modules.  To accomplish this, we apply
the following well-known results on traces.  Parts (a) and (b) hold for
any finite-dimensional algebra $\Af$.
\begin{lemma}\label{lem:hopftrace} \begin{itemize}
\item[\rm{(a)}] Assume $\Sf$, $\Tf$, $\UU$ are finite-dimensional modules over the algebra $\Af$ such that $\UU \simeq \Sf/\Tf$.  Then for any $x\in \Af$,
the trace $\tr_\UU(x)$ of $x$ on $\UU$ satisfies $\tr_\UU(x)=\operatorname{tr}_{\Sf}(x)+\tr_{\Tf}(x)$.

\item[\rm{(b)}]
If $\UU_1,\dots,\UU_s$ are the composition factors of a finite-dimensional $\Af$-module $\UU$, where $\UU_i$ occurs with multiplicity $c_i$, then for any $x \in \Af$,  $\tr_\UU(x)=c_1\tr_{\UU_1}(x)+\dots+c_s\tr_{\UU_s}(x).$

\item[ \rm{(c)}] {\rm (See \cite[Proposition 10.21\,(b)]{LM}.)
For any finite-dimensional $\Af$-modules $\UU$ and $\WW$, and any $x \in \Af$,
\begin{equation}\label{eq:tenstrace}
\operatorname{tr}_{\UU \otimes \WW}(x)= \sum_x  \operatorname{tr}_{\UU}\big(x_{(1)}\big)\operatorname{tr}_{\WW}\big(x_{(2)}\big). 
\end{equation}}
\end{itemize}
\end{lemma}
Let $\Sf_1,\Sf_2, \dots,\Sf_m$ be a $\ZZ$-basis of simple modules for the Grothendieck ring $\GG_0(\Af)$ of the Hopf algebra $\Af$, and for any $x \in \Af$   set 
 \begin{equation}\label{eq:tracevec} \Trs(x) :=  [ \tr_{\Sf_1}(x) \,\; \tr_{\Sf_2}(x) \;\,
\dots \; \,\tr_{\Sf_m}(x)]^{\tt T}.  \end{equation}
Observe that $\Trs(1) = [\dimm(\Sf_1) \,\; \dimm(\Sf_2)\, \; \ldots \; \,\dimm(\Sf_m)]^{\tt T} = \mathbf{s},$
the vector of dimensions of the simple $\Af$-modules.      We know that $ \mathbf{s}$  gives a right eigenvector for any McKay matrix determined
by tensoring with a finite-dimensional $\Af$-module.
Next we explore some other elements of the Hopf algebra $\Af$ that give such right eigenvectors.     
 
\begin{theorem} 
\label{Thm:tracesimp}   Assume  $\McV = (\Mf_{ij})$, where $\Mf_{ij} = [\Sf_i \ot \VV: \Sf_j]$,  is the McKay matrix associated to tensoring with an  $\Af$-module $\VV$.    Then  the trace of $x \in \Af$ on $\Sf_i \ot \VV$ for any $1 \le i \le m$ is given by
\begin{equation}
\label{eq:coprodtrace} 
\sum_x\tr_{\Sf_i}(x_{(1)}) \tr_{\VV}(x_{(2)}) =
\sum_{j=1}^m \Mf_{ij}\tr_{\Sf_j}(x), \ \text{and} \end{equation} 
\begin{equation}\label{eq:matprodtrace}  \McV\Trs(x) =\sum_x\tr_\VV(x_{(2)})\Trs(x_{(1)}).
\end{equation}
\end{theorem}

\begin{proof} 
This follows from Lemma \ref{lem:hopftrace}\,(c) and  the fact that $\Mf_{ij}$ is the multiplicity of $\Sf_j$ as a composition factor of $\Sf_i \ot \VV$.
Equation \eqref{eq:coprodtrace} gives the $i$th coordinate of the matrix equation in \eqref{eq:matprodtrace}. 
\end{proof}

 \begin{corollary}\label{rem:coprod} For the McKay matrix $\Mf_\VV$ associated to tensoring the simple $\Af$-modules with $\VV$, the following hold:
\begin{itemize}\item[\rm{(a)}]  When $g \in \Af$ is grouplike, then $\Delta(g) = g \ot g$,  and \eqref{eq:matprodtrace} says
\begin{equation}\label{eq:grouplike} \McV\Trs(g) =  \tr_{\VV}(g) \Trs(g). \end{equation}
Consequently,  for every grouplike element of $\Af$,  $\Trs(g)$ is a right eigenvector for $\McV$ with eigenvalue
$\tr_{\VV}(g)$.  When $g = 1$, this reverts to the result that the dimension vector $\mathbf{s}$  is a right eigenvector with
eigenvalue $\tr_\VV(1) = \dimm(\VV)$. 
\item[\rm{(b)}]  When $x \in \Af$ has the property that $\Delta(x) = x \ot y + z \ot x$  for some nonzero $y,z \in \Af$,  
then  \eqref{eq:matprodtrace} says
 \begin{equation}\label{eq:skewprim} \Mf_\VV \Trs(x) = \tr_{\VV}(y)\Trs(x) + \tr_{\VV}(x)
\Trs(z).
 \end{equation} 
\end{itemize} 
\end{corollary}
 
The next result draws some useful conclusions from this equation.  In part (2)(b) of Proposition
\ref{prop:reigen}, the phrase {\it generalized right
eigenvector of $\Mf_\VV$ with eigenvalue $\lambda$} refers to a vector $u$ such that $(\Mf_\VV-\lambda \mathrm{I}) u$
is a nonzero right eigenvector of $\Mf_\VV$ for $\lambda$.  Similarly, in Proposition 
 \ref{prop:leigen}\,(2)(b) below,  a generalized left eigenvector for the matrix $\QVT$ with eigenvalue $\lambda$ is a vector
 $w$ such that $w(\QVT-\lambda{\mathrm I)}$ is a nonzero left eigenvector for $\QVT$ with eigenvalue $\lambda$.

 \begin{proposition} \label{prop:reigen}  Assume  \eqref{eq:skewprim} holds for some $x \in \Af$,  and $\Trs(x) \ne \mathbf{0}$.    
 \begin{itemize}
 \item [{\rm (1)}]  
 If  $\tr_{\VV}(x) = 0$,  then  $\Trs(x)$ is a right eigenvector for $\Mf_\VV$ with eigenvalue  $\tr_{\VV}(y)$. 
\item [{\rm (2)}]   If  $\tr_{\VV}(x) \ne  0$, and  $\Trs(z)$ is a right eigenvector for $\Mf_\VV$ with eigenvalue $\lambda$, then  
\begin{itemize} 
\item [{\rm (a)}] 
 $\tr_{\VV}(y) \ne \lambda$ implies that
 $(\tr_{\VV}(y)-\lambda)\Trs(x) + \tr_{\VV}(x)\Trs(z)$ is a right eigenvector for $\Mf_\VV$
 with eigenvalue $\tr_{\VV}(y)$;
 \item [{\rm (b)}] $\tr_{\VV}(y) = \lambda$ and $\Trs(x) \not \in \mathbb{k}\Trs(z)$ imply that $\Trs(x)$ is a generalized right eigenvector for $\Mf_\VV$ with eigenvalue $\lambda$: \,  $\McV\Trs(x)= \lambda \Trs(x) + \tr_{\VV}(x)\Trs(z)$.
 \end{itemize}
\end{itemize}
  \end{proposition}    
 
 \begin{remark} \label{rem:Trs}  The case
$\tr_{\VV}(y) = \lambda$ and  $\Trs(x) = \delta\Trs(z)$ for some $0 \ne \delta \in \mathbb{k}$ cannot happen
when  $\Trs(z)$ is assumed to be a right eigenvector for $\Mf_\VV$ with eigenvalue $\lambda$ in (2),  as this
would imply that  $\Trs(x)$ is a right eigenvector for $\McV$ with both
eigenvalue $\lambda$ and eigenvalue $\lambda+\delta^{-1}\tr_{\VV}(x) \neq \lambda$. \end{remark}
  
\subsection{Left eigenvectors from traces of projective modules}\label{S2.2}  This section discusses using projective modules to
produce left eigenvectors.  Suppose $\Pf_1, \Pf_2, \dots, \Pf_m$ are the nonisomorphic indecomposable projective modules for the Hopf algebra $\Af$.  The tensor product of any  $\Af$-module $\VV$ with a projective $\Af$-module
is projective, hence the corresponding isomorphism class is a $\ZZ$-linear combination of the $[\Pf_j]$ in the Grothendieck group $\Kf_0(\Af)$
 (see \cite[Sec. 3.1]{L}).   We use $[\Pf_i \ot \VV:\Pf_j]$ to denote the multiplicity of $[\Pf_j]$ in $[\Pf_i \ot \VV]$ and define
\begin{equation}\label{eq:QV} \QV=\left(\Qf_{ij}\right), \;  \; \text{where}\; \;
\Qf_{ij} = [\Pf_i \ot \VV:  \Pf_j].\end{equation}    We refer to the matrix $\QV$ as the \emph{projective McKay matrix} and relate
$\QV$ to a McKay matrix in Theorem \ref{thm:QV} below.  If $x \in \Af$ and $\Delta(x) = \sum_x  x_{(1)} \ot x_{(2)}$,  then  
\begin{equation}
\label{eq:proj} 
\sum_x  \tr_{\Pf_i}(x_{(1)}) \tr_{\VV}(x_{(2)}) =  \tr_{\Pf_i \ot \VV}(x)  = \sum_{j=1}^m \Qf_{ij} \tr_{\Pf_j}(x) = \sum_{j=1}^m \tr_{\Pf_j}(x) \left(\QVT\right)_{ji}.
\end{equation}
This is the $i$th component of the matrix equation
\begin{equation}
\label{eq:Mproj} 
\Trp(x) \QVT = \sum_x  \tr_\VV(x_{(2)})  \Trp(x_{(1)}), \quad \text{where}
\end{equation}
\begin{align}
\label{eq:Trproj}
\Trp(x):=\begin{bmatrix}
\tr_{\Pf_1}(x)\; \,\tr_{\Pf_2}(x) \; \,\dots \; \,\tr_{\Pf_{m}}(x)
\end{bmatrix}.
\end{align}
We have the following analogs of \eqref{eq:grouplike} and \eqref{eq:skewprim}:
\begin{corollary}\label{cor:left}   Let $\QVT$ be the McKay matrix associated to tensoring the projective indecomposable $\Af$-modules with $\VV$. Then
\begin{itemize}\item[\rm{(a)}]  When $g \in \Af$ is grouplike, 
\begin{equation}\label{eq:prgrouplike} \Trp(g)\QVT  =  \tr_{\VV}(g)  \Trp(g). \end{equation}
Consequently,  for every grouplike element of $\Af$,  $\Trp(g)$ is a left eigenvector for $\QVT$ with eigenvalue
$\tr_{\VV}(g)$.    When $g = 1$,  the eigenvalue is $\tr_\VV(1) = \dimm(\VV)$, and the eigenvector $\Trp(1)$
is just the dimension vector $\mathbf{p} = 
\begin{bmatrix}
\dimm(\Pf_1)\; \,\dimm(\Pf_2)\,\; \dots \;\,\dimm(\Pf_{m})
\end{bmatrix}$.
\item[{\rm(b)}] When $x \in \Af$ has the property that $\Delta(x) = x \ot y + z \ot x$  for some nonzero $y,z \in \Af$,  
then \eqref{eq:Mproj} says  
 \begin{equation}\label{eq:skewpprim} \Trp(x) \QVT = \tr_{\VV}(y)\Trp(x) + \tr_{\VV}(x)
 \Trp(z).
 \end{equation} 
 \end{itemize} 
 \end{corollary}  
 \smallskip

 The next result is the projective version of Proposition \ref{prop:reigen}.  
 \smallskip
 
 \begin{proposition} 
 \label{prop:leigen}  
 Assume \eqref{eq:skewpprim} holds for $x \in \Af$ and $\Trp(x) \ne \mathbf{0}$.    
 \begin{itemize}
 \item [{\rm (1)}]  
 If  $\tr_{\VV}(x) = 0$,  then  $\Trp(x)$ is a left eigenvector for $\QVT$ with eigenvalue  $\tr_{\VV}(y)$. 
\item [{\rm (2)}]   If  $\tr_{\VV}(x) \ne  0$, and  $\Trp(z)$ is a left eigenvector for $\QVT$ with eigenvalue $\lambda$, then  
\begin{itemize} 
\item [{\rm (a)}] 
 $\tr_{\VV}(y) \ne \lambda$ implies that
 $(\tr_{\VV}(y)-\lambda)\Trp(x) + \tr_{\VV}(x) \Trp(z)$ is a left eigenvector for $\QVT$
 with eigenvalue $\tr_{\VV}(y)$;
 \item [{\rm (b)}]  $\tr_{\VV}(y) = \lambda$ and  $\Trp(x) \not \in \mathbb{k}\Trp(z)$ imply that $\Trp(x)$ is a generalized left eigenvector for $\QVT$ with eigenvalue $\lambda$:  $\Trp(x)\QVT= \lambda \Trp(x) + \tr_{\VV}(x)\Trp(z)$.
\end{itemize}
\end{itemize}
  \end{proposition}   
 
\begin{remark} As in Remark \ref{rem:Trs}, the case  $\tr_{\VV}(y) = \lambda$ and  $\Trp(x) = \delta \Trp(z)$ for some $0\ne \delta \in
 \mathbb{k}$ cannot happen when $\Trp(z)$ is assumed to be a left eigenvector for $\QVT$ with eigenvalue $\lambda$.
 \end{remark}  
  
\subsection{The Cartan map and Cartan matrix}\label{S2.3}  We will use the Cartan map to relate McKay matrices and 
projective McKay matrices for any finite-dimensional Hopf algebra $\Af$.  

The two Grothendieck groups mentioned in the Introduction are related by the \emph{Cartan map} $\mathsf{c}:  \Kf_0(\Af) \rightarrow  \GG_0(\Af)$, $[\Pf] \mapsto [\Pf]$.
The \emph{Cartan matrix} is the integral matrix $\Cf = \big(\Cf_{ij})$ representing $\mathsf{c}$ in the bases $\{[\Pf_j] \mid 1\le j \le m\}$
and $\{[\Sf_j] \mid 1\le j \le m\}$ of $ \Kf_0(\Af)$ and $\GG_0(\Af)$ respectively. It has  as its $(i,j)$ entry $\Cf_{ij} = [\Pf_i:\Sf_j]$, the multiplicity of
$[\Sf_j]$, when $[\Pf_i]$ is written as a $\ZZ$-combination of the classes $[\Sf_j]$, which is also the multiplicity of
$\Sf_j$   in a composition series for $\Pf_i$.   
When $\Af$ is semisimple, $\Cf$ is the identity matrix, as
$\Pf_j = \Sf_j$ for all $j$.    In general, the Cartan matrix  is not  invertible.   

Here is a tiny example exhibiting a non-invertible $\Cf$: The Sweedler Hopf algebra has a basis
$1,a,b,ab$, where $a^2 = 0$, $b^2 = 1$,  $ba = -ab$ (it is the baby Taft algebra $\Af_2$ - compare Section \ref{S3.1} with $q=-1$).    It has two one-dimensional simple modules 
$\Sf_\alpha$ and $\Sf_\varepsilon$,
where  $\alpha(a) = 0$, $\alpha(b) = -1$,  and $\varepsilon$ is the counit with $\varepsilon(a) = 0$ and $\varepsilon(b) = 1$.   The corresponding projective covers $\mathsf{P}_\alpha$
and $\Pf_\varepsilon$ are two-dimensional and have both simple modules as composition factors.
Thus, $\Cf= \left( \begin{matrix} 1 & 1 \\ 1 & 1 \end{matrix} \right)$.    In fact, the Cartan matrix of any Taft algebra 
$\Af_n$ defined using
a primitive root of unity of order $n\ge2$ is the $n \times n$ matrix with all entries equal to 1 (see \cite[Exer. 10.2.4]{LM}). 

Now suppose that $\QV = \big(\Qf_{ij}\big)$ is the projective McKay matrix for tensoring with $\VV$ as in \eqref{eq:QV}.
Then the relation between $\QV$ and  the McKay matrix $\McV$ for tensoring simple modules
with $\VV$ can be described as follows:

\begin{proposition} For any $\Af$-module $\VV$,  \hspace{-.3cm}  
\begin{itemize} \item[{\rm(a)}]  $\QV\Cf= \Cf \McV$, where $\Cf = \big(\Cf_{ij}\big)$ is the Cartan matrix.
 \item[{\rm(b)}] If $v$ is a right eigenvector for $\McV$ with eigenvalue $\lambda$, then
 $\Cf v$  is a right eigenvector for $\QV$ with eigenvalue $\lambda$.
 \item[{\rm(c)}] If $w$ is a left eigenvector for $\QV$ with eigenvalue $\mu$,  then $w\Cf$ is a left 
eigenvector for $\McV$ with eigenvalue $\mu$.    
\end{itemize}
\end{proposition}
  
 \begin{proof} (a) With notation as in \eqref{eq:QV}, we have on one hand
$$[\mathsf{P}_i \ot \VV]= \sum_{t=1}^m  \Qf_{it} [\mathsf{P}_t] \;
\cto \;  \sum_{t=1}^m  \Qf_{it}[\mathsf{P}_t]= \sum_{t,\ell = 1}^m  \Qf_{it}\Cf_{t\ell} 
 [\Sf_\ell],$$
  and on the other,   
$$[\mathsf{P}_i \ot \VV] \; \cto \; [\mathsf{P}_i \ot \VV] = 
[\mathsf{P}_i]  [ \VV]  = \sum_{t=1}^m \Cf_{it}[\Sf_t]
 [\VV] =  \sum_{t,\ell = 1}^m \Cf_{it} \Mf_{t\ell} [\Sf_\ell].$$ 
Thus,  $\QV\mathsf{C = CM_V}$.  

(b) If $v$ is a right eigenvector for $\McV$ with eigenvalue $\lambda$, then $\QV\mathsf{C}v = \mathsf{CM_V}v = \lambda {\mathsf C} v$ so that 
$\mathsf{C}v$  is a right eigenvector for $\Qf_\VV$ with eigenvalue $\lambda$.   (c) Similarly,
if $w$ is a left eigenvector for $\QV$ with eigenvalue $\mu$,  then $w\Cf$ is a left 
eigenvector for $\McV$ with eigenvalue $\mu$.  \end{proof}

For simplicity, we usually omit the brackets on the isomorphism class representatives of $\Kf_0(\Af)$ and
$\GG_0(\Af)$  in what follows unless 
they are needed for clarity.   We will use the next result,  which can be found for example in \cite[Prop. 9.2.3]{Et}. 

\begin{proposition}\label{prop:N} Let $\Pf_i$ be the
projective cover of the simple module $\Sf_i$. Then for any $\Af$-module $\Nf$,  
$$\mathsf{dim}_{\mathbb k} \mathsf{Hom}_{\Af}(\Pf_i, \mathsf{N}) = [\mathsf{N}: \Sf_i],$$ the
multiplicity of $\Sf_i$ in a Jordan-H\"older series of $\mathsf{N}$. \end{proposition}

For every $\Pf \in \Kf_0(\Af)$ and  $\VV,\WW  \in \GG_0(\Af)$, the following holds (see \cite[Sec.~3.1]{L})
\begin{equation}\label{eq:dimHom}\mathsf{dim}_{\mathbb k} \mathsf{Hom}_{\Af}(\Pf \ot \VV,\WW) = 
 \mathsf{dim}_{\mathbb k} \mathsf{Hom}_{\Af}(\Pf,\WW \ot \VV^*).\end{equation} 
 This will enable us to relate the projective McKay matrix to the McKay matrix  of the dual module.

Assume  $\McV
= \big(\Mf_{ij}\big)$ is the McKay matrix for tensoring with $\VV \in \GG_0(\Af)$,
and $\Mf_{\VV^*}
= \big(\Mf^*_{ij}\big)$ is the McKay matrix for tensoring with the dual module
$\VV^*$.     Let   $\Mf_{\VV^*}^{\tt T} = \big(\Mf^*_{ji}\big)$ be the transpose of  $\Mf_{\VV^*}$, and let 
$\Qf_\VV =\big(\Qf_{ij}\big)$ be the projective McKay matrix for $\VV$.   Then we have the following consequence of  Proposition \ref{prop:N} and \eqref{eq:dimHom}.  Variations of this result have appeared in several different contexts such as 
\cite[Lem.~8]{AM} and \cite[Lem.~10]{MM}, and it can be regarded as a special case of 
the tensor category result \cite[Prop. 6.1.2]{EGNO}. 

\begin{theorem} \label{thm:QV} Assume $\VV$ is a module for a finite-dimensional Hopf algebra $\Af$.     Then 
$$\Qf_{ij} = \mathsf{dim}_{\mathbb k} \mathsf{Hom}_{\Af}(\Pf_i \ot \VV,\Sf_j) = 
 \mathsf{dim}_{\mathbb k} \mathsf{Hom}_{\Af}(\Pf_i,\Sf_j \ot \VV^*) = [\Sf_j \ot \VV^*: \Sf_i] = \Mf_{ji}^*.$$ 
 Therefore,  $\Qf_\VV = \Mf_{\VV^*}^{\tt T}$. 
\end{theorem} 

\begin{corollary}\label{cor:Q} For any module $\VV$ over a finite-dimensional Hopf algebra $\Af$,
 \begin{itemize}  \item[{\rm(a)}] $\Mf_{\VV^*}^{\tt T}\, \Cf = \Cf\, \Mf_{\VV}$, where $\Cf$ is the Cartan matrix.
 \item[{\rm(b)}] If $v$ is a right eigenvector
of $\McV$ with eigenvalue $\lambda$, then $ \Cf v$ is a right eigenvector for 
$\Mf_{\VV^*}^{\tt T}$ with eigenvalue $\lambda$.  Similarly,  if $w$ is a left eigenvector of $ \Mf_{\VV^*}^{\tt T}$ with eigenvalue $\mu$, then 
$w \Cf$ is a left eigenvector for $\McV$ with eigenvalue $\mu$. 
\item[{\rm(c)}] If $\mathsf{C}$ is invertible, then $ \Mf_{\VV^*}^{\tt T} = \Cf \Mf_{\VV} \Cf^{-1}$.  Therefore, $\McV$ and $\Mf_{\VV^*}^{\tt T}$ (and also $\Mf_{\VV^*}$) have the same eigenvalues when $\Cf$ is invertible. 
 \item[{\rm(d)}] If $\Af$ is semisimple, the Cartan matrix is the identity matrix, and consequently,   $\Mf_{\VV^*}^{\tt T} =\McV.$  
Moreover, if $\VV$ is self-dual, then $\McV$ is symmetric, hence orthogonally diagonalizable. 
 \end{itemize}
\end{corollary}

\subsection{Eigenvectors from projective modules}  \label{S2.4}  We combine results from the previous sections 
to obtain left eigenvectors from traces of projective modules. Theorem \ref{thm:QV} implies that
$\QVT = \Mf_{\VV^*}$ holds for any finite-dimensional Hopf algebra $\Af$, where $\VV^*$ is the dual module to $\VV$, and $\Qf_\VV$
is the projective McKay matrix.
As a consequence of  \eqref{eq:Mproj},  we have 

\begin{theorem}\label{thm:projevec} For any $\Af$-module $\VV$ and all $x \in \Af$, \smallskip 
\begin{itemize} \item[{\rm(a)}]  
$\displaystyle{\Trp(x)\Mf_{\VV^*}  = \sum_x  \tr_\VV(x_{(2)})  \Trp(x_{(1)}),}$\,
where \, $\Trp(x) = \left[
\tr_{\Pf_1}(x)\, \; \tr_{\Pf_2}(x)\, \; \dots \; \,\tr_{\Pf_{m}}(x)
\right].$
 \smallskip 
 \item[{\rm (b)}]   If $(\VV^*)^* \cong \VV$, then 
$\displaystyle{
\Trp(x)\McV = \sum_x  \tr_{\VV^*}(x_{(2)})  \Trp(x_{(1)}).}$
\end{itemize}  \end{theorem}

\begin{corollary}\label{cor:V=V**}  Under the assumption that $(\VV^*)^*\cong \VV$ for  the $\Af$-module $\VV$,  the following 
results hold (compare these to the corresponding results for simple modules  \eqref{eq:grouplike} and \eqref{eq:skewprim}):
\begin{itemize}
\item[{\rm(a)}] When $g \in \Af$ is grouplike, 
\begin{equation}\label{eq:pgrouplike} \Trp(g)\McV  =  \tr_{\VV^*}(g)  \Trp(g). \end{equation}
Hence,  for every grouplike element of $\Af$,  $\Trp(g)$ is a left eigenvector of $\McV$ of eigenvalue
$\tr_{\VV^*}(g)$.  The vector $\Trp(1)$
is just the dimension vector $\mathbf{p} = 
\begin{bmatrix}
\dimm(\Pf_1)\;\, \dimm(\Pf_2)\; \,\dots \;\,\dimm(\Pf_{m})
\end{bmatrix}$,
and the eigenvalue is $\tr_{\VV^*}(1) = \dimm(\VV^*) = \dimm(\VV) =\tr_{\VV}(1)$.  \smallskip

\item[{\rm(b)}] When $x \in \Af$ has the property that $\Delta(x) = x \ot y + z \ot x$  for some nonzero $y,z \in \Af$,  
then \eqref{eq:Mproj} says  
 \begin{equation}\label{eq:skewprim2} \Trp(x)\McV = \tr_{\VV^*}(y) \Trp(x) + \tr_{\VV^*}(x)
 \Trp(z).
 \end{equation} 
 \end{itemize} 
 \end{corollary}

\subsection{Eigenvectors for McKay matrices from the Grothendieck algebra $\GG_0^{\CC}(\Af)$}\label{S2.5}
Next we describe a way to produce left eigenvectors and generalized left eigenvectors for McKay matrices
using the Grothendieck algebra $\GG_0^{\CC}(\Af) = \CC \ot_{\ZZ} \GG_0(\Af)$ of any Hopf algebra $\Af$.
The classes $[\Sf_1],[\Sf_2], \dots, [\Sf_m]$
of the nonisomorphic simple modules give a $\CC$-basis for
$\GG_0^{\CC}(\Af)$.  

\begin{proposition}\label{prop:rightmult}\begin{itemize} \item[{\rm(a)}] Let $\VV$ be an $\Af$-module, and assume 
$[\XX] = c_1 [\Sf_1]+ \cdots + c_m [\Sf_m] \in \GG_0^{\CC}(\Af)$,
$c_j \in \CC$ for all $j$,   is an eigenvector for the right multiplication operator $\mathsf{R}_{[\VV]}$ of $\GG_0^{\CC}(\Af)$ with eigenvalue $\lambda$.   Then  $[c_1\; c_2\;\, \dots\;\, c_m]$ is a left eigenvector for $\McV$
with eigenvalue $\lambda$.
\item[{\rm(b)}]  Suppose that $[\YY] = d_1 [\Sf_1] + \cdots + d_k [\Sf_m] \in \GG_0^{\CC}(\Af)$ has the property that 
$(\mathsf{R}_{[\VV]}- \lambda \mathrm{I})^\ell(\YY) = 0$.   Then for the McKay matrix $\McV$, we have
 $[d_1\,\;d_2\,\; \dots\;\; d_m] (\McV - \lambda \mathrm{I})^\ell = 0.$
 \end{itemize} \end{proposition}

\begin{proof} {\rm (a)} We are assuming that $\mathsf{R}_{[\VV]}([\XX])= \lambda [\XX]$, or more specifically,
$$(c_1 [\Sf_1] + \cdots + c_m [\Sf_m])[\VV] =  c_1 [\Sf_1][\VV] +  \cdots + c_m [\Sf_m] [\VV] = \lambda(c_1 [\Sf_1] + \cdots + c_m [\Sf_m]).$$
Since multiplication in $\GG_0^{\CC}(\Af)$ is given by tensoring, this implies that
$$\sum_{i=1}^m c_i [\Sf_i \ot \VV: \Sf_j] =\lambda c_j$$
holds for each $j$, where  $[\Sf_i \ot \VV: \Sf_j]$ is the multiplicity of $[\Sf_j]$ in $[\Sf_i \ot \VV]$.   However,
$[\Sf_i \ot \VV: \Sf_j] = \Mf_{ij}$, $(i,j)$ entry of $\McV$.      Therefore  $\sum_{i=1}^m c_i \Mf_{ij} =\lambda c_j$ for all $j$, which says
$$[c_1\,\;c_2 \,\;\dots\;\; c_m] \McV = \lambda[c_1\,\;c_2 \,\;\dots\;\; c_m].$$ 
Part (b) follows by induction, with part (a) providing the $\ell = 1$ case. 
  \end{proof}

\begin{remark}{\rm  Grothendieck algebras in general do not have to be commutative, 
so we do need to specify the ``right" multiplication above.}\end{remark}

\subsection{Eigenvectors for McKay matrices from $\Kf_0^{\CC}(\Af)$}\label{S2.6}
To determine additional information about the McKay matrix $\McV$, we consider the operator $\mathsf{R}_{[\VV^*]}$
of right multiplication by $[\VV^*]$ on the finite-dimensional complex vector space $\Kf_0^\CC(\Af) =
\CC \ot_{\ZZ} \Kf_0(\Af)$.       Suppose $[\mathsf{Y}]$ is an eigenvector for $\mathsf{R}_{[\VV^*]}$ with eigenvalue $\xi$.    Then $[\mathsf{Y}] [\VV^*] = \xi [\mathsf{Y}]$, where
  $[\mathsf{Y}] = \sum_{i=1}^m y_i [\Pf_i]$, a $\CC$-linear combination of the projective covers.
The matrix $\Qf_{\VV^*}$ has $(i,j)$ entry equal to the multiplicity of $[\Pf_j]$ in $[\Pf_i \ot \VV^*] = [\Pf_i][\VV^*]$, which implies
\begin{equation}\label{eq:Qeigen} [y_1\,\;y_2\,\; \dots\;\; y_m] \Qf_{\VV^*} = \xi  [y_1\,\;y_2\,\; \dots\;\; y_m].\end{equation}
 
Using the fact that
$\Qf_{\VV^*}^{\tt T} = \Mf_{(\VV^*)^*}$ from Theorem \ref{thm:QV} and \eqref{eq:Qeigen}, we have

\begin{proposition} Assume that the $\Af$-module $\VV$ satisfies  $(\VV^*)^* \cong \VV$.    Let
$[\mathsf{Y}]$ be a left eigenvector for $\mathsf{R}_{[\VV^*]}$ with eigenvalue $\xi$ in  $\Kf_0^\CC(\Af)$, and let 
$ [y_1\;y_2\,\; \dots\;\; y_m]$ 
be its coordinate vector relative to the basis $\{[\Pf_i]\}_{i=1}^m$ of $\Kf_0^\CC(\Af)$ of projective covers.     Then 
$$\Mf_{\VV} [y_1\,\;y_2\,\; \dots\;\; y_m]^{\tt T} = \xi [y_1\;y_2\,\; \dots\;\; y_m]^{\tt T},$$
so that $ [y_1\,\;y_2\,\; \dots\;\; y_m]^{\tt T}$ is a right eigenvector for $\Mf_{\VV}$ with eigenvalue $\xi.$ \end{proposition}

\begin{remark} {\rm When the antipode $S$ of $\Af$ has the property that $S^2$ is an inner automorphism of $\Af$, then
$(\VV^*)^* \cong \VV$ holds for all $\Af$-modules $\VV$ (see \cite[Lem.~10.2(a)]{LM}).     Any semisimple Hopf algebra will have $(\VV^*)^* \cong \VV$  for all $\VV$, as will any symmetric algebra \cite{OS}.  Drinfeld doubles are always symmetric (i.e. have
a nondegenerate, symmetric, associative bilinear form).
In particular, the Drinfeld double $\Df_n$,  which we investigate in detail in Section 3, has the property $(\VV^*)^* \cong \VV$  for all $\VV$, and $\Df_n$ is not semisimple.    Also, since the quantum 
group $\mathfrak{u}_q(\mathfrak{sl}_2)$ for $q$ a root of unity has a unique simple module of each dimension,  
$(\VV^*)^* \cong \VV$ holds  for all  $\mathfrak{u}_q(\mathfrak{sl}_2)$-modules $\VV$, and $\mathfrak{u}_q(\mathfrak{sl}_2)$ is
not semisimple.} \end{remark}

\subsection{Eigenvectors for McKay matrices of semisimple Hopf algebras}\label{S2.7}
In this section, we assume that the Hopf algebra $\Af$ is semisimple.
The assumption of semisimplicity enables us to say more about the eigenvectors and characters of $\Af$.  

Following \cite{L} and \cite[p. 886]{W},   we define    
$$\langle \XX, \YY \rangle = \mathsf{dim}_{\mathbb k} \mathsf{Hom}_\Af(\XX^*, \YY),$$
for any two $\Af$-modules $\XX$ and $\YY$ in $\GG_0^\CC(\Af)$, where $\XX^*$ denotes the dual module of $\XX$. 
This generates a nondegenerate, symmetric associative $\CC$-bilinear form on
$\GG_0^\CC(\Af)$.  Therefore, $\GG_0^\CC(\Af)$ is a symmetric algebra with dual
bases $\{\Sf_1,\Sf_2, \dots, \Sf_m\}$ and $\{\Sf_1^*,\Sf_2^*, \dots, \Sf_m^*\}$, where the $\Sf_i$ are representatives of the isomorphism
classes of simple $\Af$-modules.    

Let $\zeta_1,\zeta_2,\dots, \zeta_s$ be the $\GG_0^\CC(\Af)$-characters afforded by the simple $\GG_0^\CC(\Af)$-modules.   
Since $\Af$ is semisimple, so is the Grothendieck algebra $\GG_0^\CC(\Af)$ (see, for example, \cite{Z} and \cite{W}).
Therefore, by  \cite[Sec.~9B]{CR}  we know that the characters of $\GG_0^\CC(\Af)$ are linearly independent
over $\CC$.   Moreover, if $\ec_1,\ec_2, \dots, \ec_s$ are the primitive central idempotents of $\GG_0^\CC(\Af)$,  then $\zeta_i(\ec_i) = \zeta_i(1)$,
and $\zeta_i(\ec_j) = 0$ for $j \ne i$.     The idempotents $\ec_1,\ec_2,\dots, \ec_s$ form a $\CC$-basis for the center of $\GG_0^\CC(\Af)$.  

By Proposition 9.17 (i) of \cite{CR}, the duality of the bases $\{\Sf_j^*\}$, $\{\Sf_j\}$ relative to the symmetric bilinear form above can be used to define 
$$\mathcal{D}_i = \sum_{j=1}^m \zeta_i(\Sf_j^*) \Sf_j = \sum_{j=1}^m \zeta_i(\Sf_j) \Sf_j^*, \qquad \text{for} \; 1 \le i \le s,$$
and to prove that $\zeta_i(\mathcal{D}_i) \ne 0$,    Then applying\cite [Prop.~9.17 (ii)]{CR}, we have
\begin{proposition} \label{prop:cidems} Assume that the Hopf algebra $\Af$ is semisimple, and let $\zeta_1,\zeta_2,\dots, \zeta_s$ be the characters of the simple $\GG_0^\CC(\Af)$-modules.  Then for all $1 \le i \le s$,  the primitive central idempotent $\ec_i$ of $\GG_0^\CC(\Af)$ corresponding to the character $\zeta_i$  is given by
$$\ec_i = \zeta_i(1) \zeta_i(\mathcal{D}_i)^{-1} \mathcal{D}_i, \quad \text{where} \quad
\zeta_i(\mathcal{D}_i) = \sum_{j=1}^m \zeta_i(\Sf_j^*) \zeta_i(\Sf_j) = \zeta_i \left(\bigoplus_{j=1}^m \Sf_j^* \ot \Sf_j\right),$$ 
and $\Sf_j$, $1 \le j \le m$,  are the simple $\Af$-modules. 
\end{proposition}
 
 \begin{remarks} $\bullet$  {\rm In \cite{W}, Witherspoon investigated $\GG_0^\CC(\Af)$ when $\Af$ is a semisimple, almost cocommutative Hopf 
algebra.   Under these assumptions, $\GG_0^\CC(\Af)$ is both semisimple and commutative, so the simple modules
for $\GG_0^\CC(\Af)$ are one-dimensional.   The expression for the primitive central idempotents coming from  
\cite[(3.4)]{W} is basically the same as the one given in Proposition \ref{prop:cidems}; however,  since $\zeta_i(1) = 1$ for all $i$ when
$\Af$ is semisimple and almost cocommutative, 
that factor does not appear in \cite{W}.   

 \noindent $\bullet$ The primitive central idempotents  $\ec_i$ form a basis for $\GG_0^\CC(\Af)$ when $\Af$ is semisimple and almost cocommutative, and in fact,  the simple $\GG_0^\CC(\Af)$-modules are exactly the $\CC \ec_i$.  Therefore, $s = m$ in the situation considered in \cite{W}.}     This is not true when
$\Af$ is an \emph{arbitrary} semisimple Hopf algebra.  Each $\ec_i$ is the identity element of a matrix block of the semisimple
algebra $\GG_0^\CC(\Af)$, but these blocks do not have to be one-dimensional.  The Hopf algebra 
 that is (14) in Kashina's classification  \cite{K} of 16-dimensional semisimple Hopf algebras  is an example of this phenomenon.  The Grothendieck algebra $\GG_0^\CC(\Af)$ has a $2 \times 2$ matrix block in that case.  
\end{remarks}  

Assume $\Af$ is semisimple and 
$\VV$ is an $\Af$-module that is central in $\GG_0^\CC(\Af)$.    Then $\VV = \sum_{i=1}^s \lambda_i \ec_i$, which implies that
$ \ec_i \VV = \lambda_i \ec_i$, that is,  $\ec_i$ is an eigenvector for right multiplication by $\VV$ in $\GG_0^\CC(\Af)$.  Therefore, by
Propositions \ref{prop:rightmult} and \ref{prop:cidems}, we have for the McKay matrix $\McV = (\Mf_{ij})$, 
\begin{align*}\lambda_i \ec_i &= \lambda_i  \zeta_i(1)  \zeta_i(\mathcal{D}_i)^{-1}  \sum_{j=1}^m  \zeta_i(\Sf_j^*) \Sf_j  = \ec_i \VV =
 \zeta_i(1)  \zeta_i(\mathcal{D}_i)^{-1}  \sum_{j=1}^m \zeta_i(\Sf_j^*) \Sf_j \VV \\ 
& =  \zeta_i(1)  \zeta_i(\mathcal{D}_i)^{-1}  \sum_{j=1}^m \zeta_i(\Sf_j^*)\left( \sum_{\ell=1}^m [\Sf_j \ot \VV: \Sf_\ell] \Sf_\ell\right)  
 =  \zeta_i(1)  \zeta_i(\mathcal{D}_i)^{-1} \sum_{\ell=1}^m  \left( \sum_{j=1}^m  \zeta_i(\Sf_j^*) \Mf_{j\ell}\right)\Sf_\ell.
\end{align*}
In other words,  $\lambda_i  \zeta_i (\Sf_\ell^*) =  \sum_{j=1}^m  \zeta_i(\Sf_j^*) \Mf_{j\ell}$.
This says that $[\zeta_i(\Sf_1^*)\;\, \zeta_i(\Sf_2^*)\;\, \dots\;\, \zeta_i(\Sf_m^*)]$ for $1\le i \le s$,
is a left eigenvector for $\McV$ with eigenvalue $\lambda_i$ when 
$\VV = \sum_{i=1}^s \lambda_i \ec_i$ is a central element of $\GG_0^\CC(\Af)$.  Hence, we have shown the following

\begin{proposition} Assume that the Hopf algebra $\Af$ is semisimple, and $\VV$ is an $\Af$-module that is central in 
the Grothendieck algebra $\GG_0^\CC(\Af).$   Let $\zeta_1,\zeta_2, \dots, \zeta_s$  be the simple characters
of  $\GG_0^\CC(\Af).$ Then $[\zeta_i(\Sf_1^*)\;\, \zeta_i(\Sf_2^*)\;\, \dots\;\, \zeta_i(\Sf_m^*)]$ is a left eigenvector of $\McV$ for $1 \le i \le s$,
where the characters $\zeta_i$ are evaluated on the dual modules $\Sf_1^*,\Sf_2^*,
\dots, \Sf_m^*$  of  the nonisomorphic simple $\Af$-modules $\Sf_1,\Sf_2,\dots, \Sf_m$. \end{proposition}

\begin{remark} {\rm When $\Af$ is semisimple and almost cocommutative (as in \cite{W}), then  all the left eigenvectors for $\McV$
for any choice of $\VV$  are obtained in this fashion, 
since the $\ec_i$ form a $\CC$-basis of $\GG_0^\CC(\Af)$  in this case, and every $\Af$-module $\VV$ can be
expressed as $\VV = \sum_{i=1}^{s=m} \lambda_i \ec_i$ for some $\lambda_i \in \CC$.  Consequently,
the following result holds.} \end{remark}

\begin{corollary} When  $\Af$ is a semisimple, almost cocommutative Hopf algebra, the left eigenvectors of $\McV$ are the same for the McKay matrix of \emph{any} finite-dimensional $\Af$-module $\VV$,
and they can be gotten by evaluating the simple characters $\zeta_i$, $1 \le i \le m$,  of  $\GG_0^\CC(\Af)$, 
\begin{equation}[\zeta_i(\Sf_1^*)\;\, \zeta_i(\Sf_2^*)\;\, \dots\;\, \zeta_i(\Sf_m^*)], \end{equation}
on the dual modules $\Sf_j^*$ of the nonisomorphic simple $\Af$-modules $\Sf_j$, $1 \le j \le m$. 
\end{corollary} 

\noindent \section{\bf The Drinfeld Double of the Taft Algebra and its Modules}

\emph{Throughout Section 3, we assume  $n$ is an odd integer $\ge 3$,  $\mathbb{k}$ is an algebraically closed field of characteristic $0$,
and $q$ is a primitive $n$th root of unity in $\mathbb{k}$. When $n$ is even, the Drinfeld double and its 
modules are defined similarly,   but different behavior is exhibited, and so this case will not be considered in this work.}   
\noindent \subsection{Preliminaries} \label{S3.1}   
The Drinfeld double $\Df_n$ of the Taft algebra $\Af_n$  has
a presentation as  the Hopf algebra over $\mathbb{k}$ with generators
$a,b,c,d$ that satisfy the following relations:
\begin{align} 
\begin{split} ba = q ab, & \qquad  db = qbd, \\
ca = q ac, & \qquad  dc = qcd, \\
bc = cb, & \qquad  da-qad = 1-bc, \\
a^n = 0 = d^n, & \qquad b^n = 1 = c^n. \end{split}\end{align}
The coproduct, counit, and antipode of $\Df_n$ are given by
\begin{align}
\begin{split}\label{eq:hopfstructure}  \Delta(a) = a \otimes b + 1 \otimes a, & \quad  \Delta(d) = d \otimes c + 1 \otimes d,  \\
\Delta(b) = b \otimes b, & \quad  \Delta(c) = c \otimes c, \\
\varepsilon(a) = 0 = \varepsilon(d), & \quad  \varepsilon(b) = 1 = \varepsilon(c),\\
S(a) = -ab^{-1}, \  \ S(b) = b^{-1}, & \quad
S(c) = c^{-1}, \  \ S(d) = -dc^{-1}.  
\end{split}\end{align} 
The Taft algebra $\Af_n$ is the Hopf subalgebra generated by $a$ and $b$.   It follows from \eqref{eq:hopfstructure} and the
fact that $\Delta$ is an algebra homomorphism 
that the elements $b^ic^k$ for $0 \le i,k \le n-1$ are grouplike.

 \subsubsection{The simple and projective $\Df_n$-modules}\label{S:3.11} \, 
 The simple $\Df_n$-modules $\VV(\ell,r)$  are indexed by a pair $(\ell,r)$ where
$1 \le \ell \le n$ and $r \in \mathbb{Z}_n =  \mathbb{Z}/n\mathbb{Z}$  (the integers modulo $n$).   Then  $\VV(\ell,r)$ is a $\mathbb k$-vector space of dimension $\ell$ with basis
$v_1,v_2, \dots, v_\ell$  and with $\Df_n$-action given  by
\begin{align}
\label{eq:action} 
\begin{split} a. v_j &= v_{j+1}, \  1 \leq j < \ell, \hspace{2.2cm} a.v_\ell = 0, \\
b.v_j &= q^{r+j-1}v_j,  \hspace{3.3cm}  c.v_j = q^{j-(r+\ell)}v_j, \ \ 1 \le j \le \ell, \\
d.v_j &= \alpha_{j-1}(\ell) v_{j-1},  \ 1 < j \le \ell,  \hspace{0.9cm}  d.v_1 = 0,\end{split}  \end{align}
where 
\begin{equation}\alpha_i(\ell) = \frac{ \left(q^i - 1\right)\left(1-q^{i-\ell}\right)}{q-1} \qquad \text{for}\; 1 \le i \le n-1.
\end{equation}

From Chen et al. \cite{Chen2.5, Chen},  the following hold:  
\begin{enumerate} 
\item $\VV(1,0)$ is the trivial $\Df_n$-module with action given by the counit $\varepsilon$.
\item $\VV(\ell, r) \otimes \VV(1,s)  \cong \VV(\ell, r+s)$.    
\item $\VV(\ell,r) \otimes \VV(\ell',s)$ is completely reducible if and only if $\ell+\ell' \le n+1$. In this case,  if $m = \mathsf{min}(\ell,\ell')$, 
then 
\begin{equation}\label{eq:tensdecomp}\VV(\ell,r) \otimes \VV(\ell',s) \cong \bigoplus_{j=1}^m \VV(\ell+\ell' + 1-2j, r+s + j-1).\end{equation}
\end{enumerate}

 Let $\Pf(\ell,r)$ be the projective cover of the simple $\Df_n$-module $\VV(\ell,r)$. 
Chen \cite{Chen2.75}  has shown that any indecomposable projective left $\Df_n$-module is isomorphic to one of the 
modules $\Pf(\ell,r)$ for $1\le \ell < n$ or to $\VV(n,r)$ for some $r \in \mathbb{Z}_n$, and the module $\Pf(\ell,r)$  
for $1 \le \ell < n$ has the following structure.  There is a chain of submodules $\Pf(\ell,r)  \supset  \mathsf{soc}^2(\Pf(\ell,r))
\supset  \mathsf{soc}(\Pf(\ell,r)) \supset (0)$ such that 
\begin{enumerate}
\item  $\mathsf{soc}(\Pf(\ell,r))$ is the socle of $\Pf(\ell,r)$ (the sum of all the simple submodules), and  $\mathsf{soc}(\Pf(\ell,r))\cong \VV(\ell,r)$;
\item $\mathsf{soc}^2(\Pf(\ell,r))/\mathsf{soc}(\Pf(\ell,r)) \cong  \VV(n-\ell, r+\ell) \oplus \VV(n-\ell, r+\ell)$;
\item $\Pf(\ell,r) /\mathsf{soc}^2(\Pf(\ell,r)) \cong \VV(\ell,r)$.
\end{enumerate} 
Therefore,  $[\Pf(\ell,r)] = 2 [\VV(\ell,r)] + 2 [\VV(n-\ell, r+\ell)]$ in the Grothendieck group $\GG_0(\Df_n)$,
and the dimension of the indecomposable module $\Pf(\ell,r)$ is $2n$ for $1 \le \ell <n$.  Hence, it follows that
$[\Pf(\ell,r)]=[\Pf(n-\ell, r+\ell)]$ holds in $\GG_0(\Df_n)$ for all $1\le \ell < n$ and all $r \in \ZZ_n$.  
The modules $\VV(n,r)$  for all $r \in \mathbb{Z}_n$  are the only $\Df_n$-modules  that are both simple and projective.    
 \subsubsection{The Cartan map for $\Df_n$}\label{S:3.12} \, We consider an extension of the Cartan map of $\Df_n$  to a 
 $\CC$-linear map (also denoted  $\mathsf{c}$),   $\mathsf{c}: \Kf_0^{\CC}(\Df_n) \rightarrow    
 \GG_0^{\CC}(\Df_n)$, $[\Pf] \mapsto [\Pf]$.   Then $\mathsf{c}\big([\Pf(\ell,r]\big) = [\Pf(\ell,r)] = 2 [\VV(\ell,r)] + 2 [\VV(n-\ell,\ell+r)] = \mathsf{c}\big([\Pf(n-\ell,\ell+r)]\big)$
 for all $1 \le \ell \le \frac{n-1}{2}$, and $r \in \ZZ_n$,
 and  $\mathsf{c}\big([\VV(n,r)]\big) = [\VV(n,r)]$ for all $r \in \ZZ_n$.   Therefore,  the images under $\mathsf{c}$ of the basis elements of $\Kf_0^{\CC}(\Df_n)$  are $\CC$-linearly independent elements
 of  $\GG_0^{\CC}(\Df_n)$, and it follows that
 $\dimm\big( \mathsf{im}(\mathsf{c})\big) = \frac{n(n+1)}{2}$.  Since the elements $[\Pf(\ell,r)] - [\Pf(n-\ell,\ell+r)]$, $1 \le \ell \le \frac{n-1}{2}$, $r \in \ZZ_n$, are
 linearly independent elements of the kernel of $\mathsf{c}$ and $\dimm\big(\mathsf{ker}(\mathsf{c})\big) = n^2-\frac{n(n+1)}{2} = \frac{n(n-1)}{2}$, they form a basis for the kernel.   To summarize, we have the following result.
 \begin{proposition} \label{prop:CartanDn} \begin{itemize} \item[{\rm(a)}]  The elements $2 [\VV(\ell,r)] + 2 [\VV(n-\ell,\ell+r)]$
 for all $1 \le \ell \le \frac{n-1}{2}$, $r \in \ZZ_n$, and $[\VV(n,r)], r \in \ZZ_n$,  form a basis for the image 
 $\mathsf{im}(\mathsf{c})$  of the Cartan map $\mathsf{c}: \Kf_0^{\CC}(\Df_n) \rightarrow    
 \GG_0^{\CC}(\Df_n)$.  Therefore, $\dimm\big(\mathsf{im}(\mathsf{c})\big) = \frac{n(n+1)}{2}$. 
  \item[{\rm(b)}]  The rank of the Cartan matrix $\mathsf{C}$ of $\Df_n$  is $\frac{n(n+1)}{2}$.
   \item[{\rm(c)}] The elements $[\Pf(\ell,r)] - [\Pf(n-\ell,\ell+r)]$, $1 \le \ell \le \frac{n-1}{2}$, $r\in \ZZ_n$,  form a basis for the kernel of $\mathsf{c}$,
and  $\dimm\big(\mathsf{ker}(\mathsf{c})\big) = \frac{n(n-1)}{2}$.
   \end{itemize}
   \end{proposition}

\subsection{The McKay matrix for tensoring with the $\Df_n$-module $\VV(2,0)$}\label{S3.2}
Throughout this section, we assume that $\VV$ is the two-dimensional simple $\Df_n$-module $\VV(2,0)$ with basis $\{v_1, v_2\}$.   Now it follows from \eqref{eq:tensdecomp} and  
\cite[Prop.~3.1, Thms.~3.3 and 3.5]{Chen}  that  for $\VV = \VV(2,0)$, 
\begin{enumerate}\label{eq:tensrules} 
\item $\VV(1,r) \ot \VV = \VV(2,r)$;
\item  $\VV(\ell,r) \ot \VV =  \VV(\ell+1,r) \oplus \VV(\ell-1,r+1)$ for $2 \le \ell < n$;   
\item $ \VV(n,r) \ot \VV   \cong \Pf(n-1, r+1)$;   
\item $ \Pf(1,r) \ot \VV  \cong  \Pf(2,r) \oplus 2 \VV(n,r+1)$;
\item $\Pf(\ell,r ) \ot \VV  \cong  \Pf(\ell+1,r) \oplus \Pf(\ell-1,r+1)$ for $2 \le \ell < n-1$;   
\item $\Pf(n-1,r) \ot \VV  \cong  \Pf(n-2,r+1) \oplus 2 \VV(n,r)$.   
\end{enumerate}

The McKay matrix for tensoring  the simple $\Df_n$-modules $\VV(\ell,r)$ with $\VV:=\VV(2,0)$ is the $n^2 \times n^2$ matrix $\McV = \left( \Mf_{(\ell,r),(\ell',s)} \right)$, whose entry $\Mf_{(\ell,r),(\ell',s)}$ is given by the composition series multiplicity   
$$\Mf_{(\ell,r),(\ell',s)}= [\VV(\ell,r) \otimes \VV : \VV(\ell',s)].$$ 
We assume that the numbering of the rows and columns of $\McV$ is  first by $\ell = 1$, then by $\ell =2,$ etc.,  and for each $\ell$ 
the numbering is $r = 0,1,\dots, n-1$;  that is,  we are numbering by lexicographic order, and we will often simply write $\Mf$
for $\McV$ when the choice of $\VV$ is unambiguous.
Using the decomposition formulas above,  we observe that
 the McKay matrix $\Mf$ then can be displayed using $n\times n$ blocks, 
 where $\mathrm{I} = \Ir_n$ is the $n\times n$ identity matrix, and $\mathrm{Z}$ is the $n \times n$ cyclic permutation matrix
 as presented below   
 
  \begin{equation}
\label{eq:McZ}  
\Mf =\left( \begin{matrix} 0 & \mathrm{I} & 0  & & \cdots & 0 & 0 \\ 
 \mathrm{Z} & 0 & \mathrm{I}  & & \cdots & 0 & 0 \\
 0 &  \mathrm{Z} & 0 & \mathrm{I} & \cdots & 0 & 0 \\
\vdots  & \vdots  & \mathrm{Z}   & \ddots  & \ddots & 0 & 0 \\ 
 \vdots & \vdots & \vdots & &  \ddots &  \mathrm{I} & 0 \\
0 & 0 & 0 & 0  \cdots &  \mathrm{Z} & 0 &  \mathrm{I} \\ 
2\mathrm{I}  & 0 & 0 & 0 & \cdots &  2\mathrm{Z} & 0 
  \end{matrix} \right) \qquad \qquad   \mathrm{Z} =\left( \begin{matrix} 0 & 1 & 0  & & \cdots & 0 & 0 \\ 
 0 & 0 & 1  & & \cdots & 0 & 0 \\
 0 & 0 & 0 & 1 & \cdots & 0 & 0 \\
\vdots  & \vdots  &   & \ddots  & \ddots & 0 & 0 \\ 
 \vdots & \vdots & \vdots & &  \ddots & 1 & 0 \\
0 & 0 & 0 & 0 & \cdots & 0 & 1 \\ 
1 & 0 & 0 & 0 & \cdots & 0 & 0 
  \end{matrix} \right).
\end{equation} 
 There are identity matrices on the superdiagonal of $\Mf$,  and the matrix $\mathrm{Z}$ is on the subdiagonal
 of $\Mf$ except for the last row (corresponding to the modules $\VV(n,r)$ for $r\in \ZZ_n$), 
where the nonzero entries are $2 \mathrm{I}$ and $2\mathrm{Z}$, due to the fact that
$\VV(n,r) \ot \VV(2,0) = \Pf(n-1,r+1)$, which has composition factors $\VV(n-1,r+1)$ (twice) and $\VV(1,r)$ (twice).  
 
Assume  $\YY = \diag\{\XX,\XX, \dots, \XX\}$,  where the $n \times n$ matrix $\XX$ diagonalizes $\mathrm{Z}$,
$$\XX\Zr \XX^{-1} = \Dr:= \diag\{1,q,\ldots,q^{n-1}\},$$  and  $q$ is as before, a primitive $n$th root of unity in $\mathbb{k}$ for $n$ odd and
$\ge 3$.   Then
\begin{equation}\label{eq:Mprime}\Mf' = \YY \Mf \YY^{-1}
 =\left( \begin{matrix} 0 & \mathrm{I} & 0  & & \cdots & 0 & 0 \\ 
 \Dr& 0 & \mathrm{I}  & & \cdots & 0 & 0 \\
 0 &  \Dr & 0 & \mathrm{I} & \cdots & 0 & 0 \\, \vdots & \vdots & \vdots & &  \ddots &  \mathrm{I} & 0 \\
0 & 0 & 0 &  \cdots &  \Dr & 0 &  \mathrm{I} \\ 
2\Ir  & 0 & 0 & \cdots & 0 &  2\Dr & 0 
  \end{matrix} \right).\end{equation}

\subsection{Characteristic polynomial and characteristic roots of $\McV$, $\VV= \VV(2,0)$}\label{S3.3}

We want to determine the characteristic roots of $\Mf = \McV$, which we can do by computing the characteristic
roots of $\Mf'$, as they are the same.     The advantage to working with $\Mf'$ is that its entries 
lie in the commutative ring of diagonal matrices, and so usual matrix operations apply.

Consider the characteristic polynomial of $\Mf'$, which can be found by computing  
$$\mathsf{det}\left(t \Ir_{n^2} - \Mf'\right) =  \mathsf{det}\left( \begin{matrix} t\Ir & -\Ir & 0  & & \cdots & 0 & 0 \\ 
- \Dr& t\Ir & -\mathrm{I}  & & \cdots & 0 & 0 \\
 0 & - \Dr & t\Ir &- \mathrm{I} & \cdots & 0 & 0 \\
\vdots  & \vdots  & -\Dr & \ddots  & \ddots & 0 & 0 \\ 
 \vdots & \vdots & \vdots & &  \ddots & - \mathrm{I} & 0 \\
0 & 0 & 0 & 0  \cdots &  -\Dr & t\Ir & - \mathrm{I} \\ 
-2\Ir  & 0 & 0 & 0 & \cdots &  -2\Dr & 0 
  \end{matrix} \right).$$

Define polynomials $\Uc_k(t,\Dr)$ recursively by
\begin{equation}\label{eq:QktD1} \Uc_0(t,\Dr) = \Ir,\; \; \Uc_1(t,\Dr) = t\Ir, \;\; \Uc_k(t,\Dr) = t \Uc_{k-1}(t,\Dr) -\Dr 
\Uc_{k-2}(t,\Dr), \,\; k \ge 2.\end{equation}
These polynomials are related to Chebyshev polynomials of the second kind, as we explain 
in the next section.
In computing the determinant of $t\Ir_{n^2} - \Mf'$, we will abbreviate $\Uc_k(t,\Dr)$ as $\Uc_k$.     
We perform row operations on $t\Ir_{n^2}- \Mf'$ using the matrices $-\Ir$ on the superdiagonal to eliminate the
entries beneath them.  Therefore, after using the $-\Ir$ in the first row and then the $-\Ir$ in the second row,    we have
$$ \left( \begin{matrix} \Uc_1 & -\Ir & 0  & & \cdots & 0 & 0 \\ 
t\Uc_1-\Dr \Uc_0 & 0 & -\mathrm{I}  & & \cdots & 0 & 0 \\
-\Dr \Uc_1 & 0 & t\Ir &- \mathrm{I} & \cdots & 0 & 0 \\
\vdots  & \vdots  & -\Dr & \ddots  & \ddots & 0 & 0 \\ 
 \vdots & \vdots & \vdots & &  \ddots & - \mathrm{I} & 0 \\
0 & 0 & 0 & 0  \cdots &  -\Dr & t\Ir & - \mathrm{I} \\ 
-2\Ir  & 0 & 0 & 0 & \cdots &  -2\Dr & t\Ir 
  \end{matrix} \right), \;\;  \left( \begin{matrix} \Uc_1 & -\Ir & 0  & & \cdots & 0 & 0 \\ 
\Uc_2 & 0 & -\mathrm{I}  & & \cdots & 0 & 0 \\
t\Uc_2-\Dr \Uc_1 & 0 &0 &- \mathrm{I} & \cdots & 0 & 0 \\
\vdots  & \vdots  &  & \ddots  & \ddots & 0 & 0 \\ 
 \vdots & \vdots & \vdots & &  \ddots & - \mathrm{I} & 0 \\
0 & 0 & 0 & 0  \cdots &  -\Dr & t\Ir & - \mathrm{I} \\ 
-2\Ir  & 0 & 0 & 0 & \cdots &  -2\Dr &t\Ir 
  \end{matrix} \right),$$  
  respectively.    Continuing, we obtain
$$\left( \begin{matrix} \Uc_1 & -\Ir & 0  & & \cdots & 0 & 0 \\ 
\Uc_2 & 0 & -\mathrm{I}  & & \cdots & 0 & 0 \\
\Uc_3 & 0 &0 &- \mathrm{I} & \cdots & 0 & 0 \\
\vdots  & \vdots  &  & \ddots  & \ddots & 0 & 0 \\ 
\Uc_{n-2} & \vdots & \vdots & &  \ddots & - \mathrm{I} & 0 \\
\Uc_{n-1} & 0 & 0 & 0  \cdots & 0 & 0 & - \mathrm{I} \\ 
-2\Dr \Uc_{n-2}-2\Ir  & 0 & 0 & 0 & \cdots & 0 & t\Ir 
  \end{matrix} \right),$$
  so that after the final step, the result of using the bottommost $-\Ir$ on the superdiagonal  is
$$\left( \begin{matrix} \Uc_1 & -\Ir & 0  & & \cdots & 0 & 0 \\ 
\Uc_2 & 0 & -\mathrm{I}  & & \cdots & 0 & 0 \\
\Uc_3 & 0 &0 &- \mathrm{I} & \cdots & 0 & 0 \\
\vdots  & \vdots  &  & \ddots  & \ddots & 0 & 0 \\ 
 \vdots & \vdots & \vdots & &  \ddots & - \mathrm{I} & 0 \\
\Uc_{n-1} & 0 & 0 & 0 &  \cdots & 0 &  - \mathrm{I} \\ 
t\Uc_{n-1}-2\Dr \Uc_{n-2} - 2\Ir  & 0 &0 &0  & \cdots & 0 & 0
  \end{matrix} \right).$$
Therefore, setting 
\begin{equation}\label{eq:pndef} \mathsf{p}_n(t,\Dr) := t\Uc_{n-1}(t,\Dr)-2\Dr \Uc_{n-2}(t,\Dr) -2\Ir  = \Uc_n(t,\Dr) - \Dr \Uc_{n-2}(t,\Dr) - 2\Ir, \end{equation} we have 
$\mathsf{det}(t\Ir_{n^2} - \Mf') = \mathsf{p}_n(t,\Dr)$, where $\Mf'$ is as in \eqref{eq:Mprime},  and the characteristic
roots of $\Mf'$, hence also of $\Mf$,  are the roots of $\mathsf{p}_n(t,\Dr)$.    
   
Here are the first few polynomials $\mathsf{p}_n(t,\Dr)$:
\begin{align} \begin{split}\label{eq:pntD}
  n=3: \quad & t^3-3\Dr t -2\Ir  \\
  n=5: \quad &t^5-5\Dr t^3 +5\Dr^2t -2\Ir  \\ 
  n=7:  \quad &t^7 - 7\Dr t^5+14\Dr^2 t^3 -7\Dr^3 t - 2\Ir  \\  
  n=9:  \quad &t^9-9\Dr t^7 + 27\Dr^2 t^5-30\Dr^3 t^3+9\Dr^4 t - 2\Ir   \\  
  n=11: \quad &t^{11}-11\Dr t^9+44\Dr^2 t^7 - 77\Dr^3 t^5+55\Dr^4 t^3 -11\Dr^5 t - 2\Ir  \\  
  n=13:  \quad &t^{13} - 13\Dr t^{11}+65\Dr^2 t^9 -156\Dr^3 t^7+182\Dr^4 t^5 -91\Dr^5 t^3
  +13\Dr^6t - 2\Ir.  \end{split}
 \end{align}
 These are polynomials with coefficients that are $n \times n$ diagonal matrices.  Each diagonal entry
 $q^k$ of $\Dr$ for $k\in \ZZ_n$ 
determines a polynomial $\mathsf{p}_n(t,q^k)$ in $q$ and $t$ with coefficients in $\ZZ$.  The characteristic roots of $\Mf$ are obtained by setting those $n$ polynomials equal to 0.  For example, when 
 $n = 7$, the polynomials are $\mathsf{p}_n(t,q^k) =t^7 - 7q^k t^5+14 q^{2k} t^3 -7q^{3k} t - 2$ for $k \in \ZZ_7$, and the 
characteristic roots of $\Mf$ are the roots of those 7 polynomials.

\subsection{Right eigenvectors for $\McV$, $\VV = \VV(2,0),$  and Chebyshev polynomials}\label{S3.4}
The polynomials $\mathsf{p}_n(t,\Dr)$ are related to Chebyshev polynomials $U_k(t)$ of the second
kind, which are defined recursively by the formulas
\begin{equation}\label{eq:Qt1}
U_0(t)=1,  \; \; \;
U_1(t) =2t, \; \; \;
U_k(t) =2tU_{k-1}(t) - U_{k-2}(t)\; \text{ for } k \geq 2.
\end{equation}
Setting $\Uc_k(t) = U_k(\frac{t}{2})$,  we have
\begin{equation}\label{eq:Qt}
\Uc_0(t)=1,  \; \; \;
\Uc_1(t) =t, \; \; \;
\Uc_k(t) =t \Uc_{k-1}(t) - \Uc_{k-2}(t)\; \text{ for } k \geq 2.
\end{equation}
There are a number of closed-form formulas for Chebyshev polynomials of the second kind.
Replacing $t$ by $\frac{t}{2}$ in one such formula (see for example, \cite[18.5.10 with $\lambda = 1$]{NIST} or \cite[(23), p.~185]{EMOT}) gives the following expression for $\Uc_k(t)$:
\begin{equation}\label{eqQkt}\Uc_k(t)  = \sum_{j=0}^{\lfloor \frac{k}{2}\rfloor} (-1)^j {k-j \choose j} t^{k-2j}.
\end{equation}

The polynomials $\Uc_k(t,\Dr)$, which were defined in the previous section, satisfy a similar recursion \eqref{eq:QktD1}, 
and as a result, 
\begin{equation}
\label{eqQkDt}
\Uc_k(t,\Dr)  = \sum_{j=0}^{\lfloor \frac{k}{2}\rfloor} (-1)^j {k-j \choose j} t^{k-2j}\Dr^j.\end{equation}
The polynomial $\Uc_k(t)$ is $\Uc_k(t,\Dr)$ with $\Dr$ and the $n \times n$ identity matrix $\Ir$ replaced by 1. 
Thus, when $\Dr$ and $\mathrm{I}$ are replaced by $1$ in $\mathsf{p}_n(t,\Dr)$, we obtain
\begin{equation}
\label{eq:pn} 
\mathsf{p}_n(t) := \mathsf{p}_n(t,1) =  t\Uc_{n-1}(t) - 2\Uc_{n-2}(t) - 2 = \Uc_n(t) - \Uc_{n-2}(t) - 2.
\end{equation}
 
Assume now that  $n = 2h+1$ for $h \ge 1$.  We aim to show 
\begin{equation}\label{eq:pn2}\mathsf{p}_n(t) = \mathsf{p}_{2h+1}(t) = (t-2)\Wc_h^2(t),\end{equation}
where $\Wc_h(t) = \Uc_h(t) + \Uc_{h-1}(t)$ for all $h \ge 1$, by appealing to results on Chebyshev polynomials of the fourth kind.

The sum of two consecutive Chebyshev polynomials of the second kind is a \emph{Chebyshev polynomial of the fourth kind}.
These polynomials are defined recursively by  the following formulas (see \cite[Secs. 1.2.3, 1.2.4]{MH}):
\begin{equation}\label{eq:Wt}
W_0(t)=1,  \; \; \;
W_1(t) =2t+1, \; \; \;
W_k(t) =2tW_{k-1}(t)-W_{k-2}(t)\; \text{ for } k \geq 2.
\end{equation}
Thus, they satisfy the same recursion as the polynomials $U_k(t)$, except $W_1(t) = 2t+ 1$,  while
$U_1(t) = 2t$.   
In particular,    $W_k(t) = U_k(t)+U_{k-1}(t)$ for all $k \ge 1$ by \cite[(1.54)]{MH}.

Chebyshev polynomials have integer coefficients and complex roots.   
Suppose $x = e^{i \theta}\in \CC$.   Then $z := \frac{x+x^{-1}}{2} = \cos(\theta) \in \CC$,  and it follows from \cite[(1.54)]{MH} that  the relation
\begin{equation}\label{eq:Uhz} U_{k}(z) = x^{k} + x^{k-2} + \cdots + x^{-(k-2)} + x^{-k} =\frac{x^{k+1}-x^{-(k+1)}}{x-x^{-1}} \end{equation}
holds.   Moreover, by \cite[(1.57)]{MH}, 
\begin{equation}\label{eq:Whz} W_k(z) = \frac{x^{\frac{2k+1}{2}}-x^{-\frac{2k+1}{2}}}{x^{\half}-x^{-\half}}
= x^k + x^{k-1} + \cdots + 
x^{-(k-1)} + x^{-k} = x^{-k}\,\frac{x^{2k+1}-1}{x-1}. 
 \end{equation}  
 Hence, for $\Uc_k(x+x^{-1}) = U_k(\frac{x+x^{-1}}{2})$ and  $\Wc_k(x+x^{-1}) := W_k(\frac{x+x^{-1}}{2})$, we can conclude the following: 
 
\begin{proposition}\label{prop:qrels}  Assume $\Uc_k(t)$ is defined as in \eqref{eq:Qt}.   Set $\Wc_0(t) = 1 = \Uc_0(t)$,  and let $\Wc_k(t) =
\Uc_k(t) + \Uc_{k-1}(t)$ for $k \ge 1$. Let $x = e^{i\theta} \in \CC$ be chosen so that  $t = x + x^{-1}$.  Then 
for all $k \ge 1$, the following hold:
\begin{itemize}
\item[{\rm (a)}]  $\Uc_{k}(t) = x^{k} + x^{k-2}+ \cdots  + x^{-(k-2)}+ x^{-k} = \frac{x^{k+1}-x^{-(k+1)}}{x-x^{-1}}$;
\item[{\rm (b)}]   $\Wc_{k}(t) = x^{k} + x^{k-1} + x^{k-2}+\cdots + x^{k-2}+x^{-(k-1)}+x^{-k} = x^{-k}\,\frac{x^{2k+1}-1}{x-1}$;
\item[{\rm (c)}]  $\mathsf{p}_n(t) =  \Uc_n(t) - \Uc_{n-2}(t) - 2 =  (t-2) \Wc_h^2(t)$  for $n = 2h+1$, $h \ge 1$. 
\end{itemize}
\end{proposition}   
\begin{proof} Only the last equality in (c) needs to be verified, and we proceed to show that the two sides of (c)
are equal by computing both by induction on $h$ and comparing them.
When $h = 1$, 
 $$\Wc_1^2(t) = x^{2} + 2x + 3 + 2x^{-1} + x^{-2} = (x^{2} + 1 + x^{-2}) + 2(x + x^{-1}) + 2 = \Uc_{2}(t) + 2 \Uc_1(t) + 2\Uc_0(t).$$
 Assuming the statement 
 \begin{equation}\label{eq:qexpresiona} \Uc_{2h}(t) +2\Uc_{2h-1}(t)+ \cdots + 2 \Uc_1(t) + 2\Uc_0(t) =  \Wc_h^2(t)\end{equation}
for $h \ge 1$,  we have 
\begin{align*}
& \Uc_{2h+2}(t) +2\Uc_{2h+1}(t) + 2\Uc_{2h}(t) + \cdots + 2 \Uc_1(t) + 2\Uc_0(t)  \\
& \quad = \left(x^{2h+2} + x^{2h}+ \cdots + x^{-2h}+x^{-2h-2}\right) + \left(2x^{2h+1} +2x^{2h-1} + 
\cdots + 2x^{-2h+1}+2 x^{-2h-1} \right)\\
& \qquad + \left(x^{2h} + x^{2h-2} + \cdots + x^{-2h+2} + x^{-2h}\right)+
\Uc_{2h}(t) + 2 \Uc_{2h-1}(t) + \cdots + 2\Uc_1(t) + 2\Uc_0(t)  \\
& \quad =  \left(x^{2h+2} + 2x^{2h+1} + 2x^{2h}+ \cdots + 2 + \cdots +2 x^{-2h}+2x^{-2h-1}+x^{-2h-2}\right)\\
& \qquad \qquad \qquad  \quad + \left(x^{2h} + 2x^{2h-1} +3x^{2h-2} + \cdots + (2h+1) + \cdots + 3x^{-2h+2} +2x^{-2h+1}+ x^{-2h}\right) \\
& \quad = x^{2h+2} + 2x^{2h+1} + 3x^{2h}+ \cdots + (2h+3) + \cdots + 3x^{-2h} +  2x^{-2h-1}+
x^{-2h-2} \;  = \; \Wc_{h+1}^2(t).
\end{align*} 
This completes the induction step and proves \eqref{eq:qexpresiona}. 

Now on the other hand, we claim that when $n = 2h+1 \geq 3$,
\begin{align}
\mathsf{p}_n(t)=(t-2)\Big( \Uc_{n-1}(t)+2\Uc_{n-2}(t)+2\Uc_{n-3}(t)+\cdots+2\Uc_0(t) \Big). \label{eq:pexpresionb}
\end{align}
We argue this by induction on $n$. The base case when $n=3$ follows from a direct calculation. Suppose the statement is true for $n$.  Then by \eqref{eq:pn}, for $n+1$ we have
\begin{align*}
\mathsf{p}_{n+1}(t)=&t\Uc_n(t)-2\Uc_{n-1}(t)-2\\
=&(t-2)\Uc_n(t)+2\Uc_n(t)-2\Uc_{n-1}(t)-2\\
=&(t-2)\Uc_n(t)+2\big[t\Uc_{n-1}(t)-\Uc_{n-2}(t)\big]-2\Uc_{n-1}(t)-2\\
=&(t-2)\big[\Uc_n(t)+\Uc_{n-1}(t)\big] + \mathsf{p}_n(t)\\
=&(t-2)\big[\Uc_n(t)+\Uc_{n-1}(t)+\Uc_{n-1}(t)+2\Uc_{n-2}(t)+\cdots+2\Uc_0(t)\big]\\
=&(t-2)\big[\Uc_{n}(t)+2\Uc_{n-1}(t)+2\Uc_{n-2}(t)+\cdots+2\Uc_0(t)\big].
\end{align*} 
Therefore, the assertion $\mathsf{p}_n(t) = (t-2) \Wc_h^2(t)$ follows by comparing (\ref{eq:qexpresiona}) and (\ref{eq:pexpresionb}).  
 \end{proof}

Applying \eqref{eq:Whz} with $k = h$, we deduce  
\begin{corollary}
\label{cor:qroots} 
For $h \geq 1$, the roots of $\Wc_h(t)$ as a polynomial in $x$ are all $x \ne 1$ which are roots of unity in $\CC$ of order $2h+1$.  As a polynomial in $t$, the roots of 
 $\Wc_h(t)$ are all $t = x+x^{-1}$, where $x$ is a root of unity of order $2h+1$ in $\CC$
 and $x \neq 1$.   
   \end{corollary}  

\begin{remark} We are assuming that $\Df_n$ is defined over an algebraically closed field $\mathbb k$ of characteristic 0, and $q$ is a primitive $n$th root of unity in $\mathbb{k}$ for $n=2h+1$, $h \ge 1$.   Since the subfield of $\mathbb k$ generated by 1 and $q$ is isomorphic to the subfield of $\CC$ generated by $1$ and $x$, where $x \ne 1$ is a root of
unity of order $2h+1$ as in Corollary \ref{cor:qroots}, we will identify $x$ with $q$ in what follows. \end{remark}

\begin{corollary}
\label{cor:evals} 
Assume $n \geq 3$, $n$ odd. The characteristic roots of the McKay matrix $\Mf$ in \eqref{eq:McZ} are $\lambda_{j,r}=q^r(q^j+q^{-j})
= q^r\Uc_1(q^j+q^{-j})$, where $r\in \ZZ_n$, $0 \leq j \leq \frac{n-1}{2}$, and $q$ is a primitive $n$th root of unity.  Each root  $\lambda_{0,r}=2q^r$ has multiplicity 1, and each root $\lambda_{j,r}$ for $j \ne 0$  has multiplicity 2.  
\end{corollary}  

The next result gives an expression for right eigenvectors of $\Mf = \McV$, $\VV = \VV(2,0)$, in terms of the Chebyshev polynomials.
\begin{proposition}
\label{prop:rteigenvector}
Assume $r \in \ZZ_n$ and $0 \leq j \leq \frac{n-1}{2}$,  and let $\mathbf{v}_0$ be the right eigenvector of the
matrix $\mathrm{Z}$ in \eqref{eq:McZ} corresponding to the eigenvalue $q^{2r}$ given by
$\mathbf{v}_0 =
\left [1\;\; q^{2r} \;\;  \cdots  \;\; q^{(n-1)2r} \right]^{\tt T}$.  For $1 \leq \ell \leq n-1$, set $\mathbf{v}_\ell:= q^{\ell r}\Uc_\ell(q^j+q^{-j}) \mathbf{v}_0$, where $\Uc_\ell$ is as in \eqref{eq:Qt}. Then $\mathbf{v}_{j,r}=[\mathbf{v}_0 \; \mathbf{v}_1 \; \dots \; \mathbf{v}_{n-1}]^{\tt T}$ is a right eigenvector of $\Mf$ corresponding to the eigenvalue $\lambda_{j,r}=q^r(q^j+q^{-j})$. 
\end{proposition}

\begin{proof}We check directly that the given vector $\mathbf{v}_{j,r}$ is a right eigenvector of $\Mf$, that is, we verify 
\begin{equation}\label{eq:Mv}\Mf\mathbf{v}_{j,r} = \Mf\begin{bmatrix} \mathbf{v}_0 \\ \mathbf{v}_1 \\ \vdots \\ \mathbf{v}_{n-1} \end{bmatrix} = \lambda_{j,r} \begin{bmatrix} \mathbf{v}_0 \\ \mathbf{v}_1 \\ \vdots \\ \mathbf{v}_{n-1} \end{bmatrix} \end{equation}
holds by comparing both sides of \eqref{eq:Mv}.  We assume that $\mathbf{v}_0$ is as in the statement of the proposition,  and argue this forces $\mathbf{v}_\ell := q^{\ell r}\Uc_\ell(q^j+q^{-j}) \mathbf{v}_0$ to hold for $1 \le \ell \le n-1$, where $\Uc_\ell$ is as in \eqref{eq:Qt}.  
The comparison involves checking
\begin{align*}
\mathbf{v}_1 &\isit \lambda_{j,r} \mathbf{v}_0 = q^{r}(q^j+q^{-j}) \mathbf{v}_0 =  q^{r}\Uc_1(q^j+q^{-j}) \mathbf{v}_0, & \hspace{-.65cm} \text{(Row 0)}\\
\mathbf{v}_\ell&\isit \lambda_{j,r} \mathbf{v}_{\ell-1} - \Zr \mathbf{v}_{\ell-2} &\hspace{-.65cm} \text{(Row $1\le \ell < n-1$)} \\ &= \lambda_{j,r} q^{(\ell-1)r}\Uc_{\ell-1}(q^j+q^{-j})\mathbf{v}_0 - \Zr q^{(\ell-2)r}\Uc_{\ell-2}(u^j+u^{-j})\mathbf{v}_0, & \\
&= q^{\ell r}[(q^j+q^{-j})\Uc_{\ell-1}(q^j+q^{-j}) - \Uc_{\ell-2}(q^j+q^{-j})] \mathbf{v}_0 = q^{\ell r} \Uc_\ell(q^j+q^{-j}) \mathbf{v}_0.  &
\end{align*}
For the final row, we compare $2\mathbf{v}_0+2\Zr\mathbf{v}_{n-2}$ with $\lambda_{j,r} \mathbf{v}_{n-1}$, by showing
 $2\mathbf{v}_0+2\Zr\mathbf{v}_{n-2} -  \lambda_{j,r} \mathbf{v}_{n-1} = \mathbf{0}$:
\begin{align*}
2\mathbf{v}_0+2\Zr\mathbf{v}_{n-2} -  \lambda_{j,r} \mathbf{v}_{n-1} &= 2\mathbf{v}_0 + 2q^{2r} q^{(n-2)r} \Uc_{n-2}(q^j+q^{-j}) \mathbf{v}_0 - \lambda_{j,r} \, q^{(n-1)r} \Uc_{n-1}(q^j+q^{-j})\mathbf{v}_0 \\
&= \left( 2 + 2q^{nr} \Uc_{n-2}(q^j+q^{-j}) - q^{nr} (q^j+q^{-j})\Uc_{n-1}(q^j+q^{-j}) \right) \mathbf{v}_0 \\
&= -\mathsf{p}_n(\lambda_{j,0})\mathbf{v}_0 = \mathbf{0},
\end{align*}
because $\lambda_{j,0}=q^j+q^{-j}$ is a root of $\mathsf{p}_n(t)$  by Proposition 3.4.6 (b). \end{proof}

\begin{remark}\label{rem:lastcomp}  In the expression for the right eigenvector $\mathbf{v}_{j,r}$ of $\Mf$ in \eqref{eq:Mv} corresponding to the eigenvalue $\lambda_{j,r} =q^r(q^j+q^{-j})$, the last vector component is
\begin{equation}\label{eq:lastcoord} \mathbf{v}_{n-1}= q^{(n-1)r}\Uc_{n-1}(q^j+q^{-j})\mathbf{v}_0=q^{(n-1)r} \frac{q^{jn}-q^{-jn}}{q^j-q^{-j}}
\mathbf{v}_0 = \mathbf{0}\end{equation}
when $j \ne 0$  by Proposition \ref{prop:qrels} (a).
\end{remark}

\subsection{Generalized right eigenvectors for $\McV$} \label{S3.5}  
Using the Chebyshev polynomials $\Uc_k(q^j + q^{-j})$, we now describe generalized right eigenvectors 
for $\Mf$.
\begin{theorem}\label{thm:genright}  Assume $r \in \ZZ_n$ and fix a choice of $j \in \{1,2,\ldots, \frac{n-1}{2}\}$. 
Let $\mathbf{v} = \left [1\;\; q^{2r} \;\;  \cdots  \;\; q^{(n-1)2r} \right]^{\tt T}$ be the right eigenvector of the matrix $\Zr$ corresponding to the eigenvalue $q^{2r}$ as in Proposition \ref{prop:rteigenvector}, and set $\xb_0 = \mathbf{v}$. For any $1 \leq k \leq n-1$, assume
\begin{equation}\label{eq:xkrecur} \mathbf{x}_{k} := q^{kr} \Uc_k  \mathbf{v} + q^{(k-1)r} \sum_{s=0}^{\lfloor \frac{k-1}{2}\rfloor} (k-2s) \Uc_{k-1-2s} \mathbf{v},
\end{equation}
where $\Uc_k$ is shorthand for $\Uc_k(q^j + q^{-j})$. Then $\xb_{j,r} =\left[ \xb_0 \; \, \xb_1 \; \, \ldots \; \, \xb_{n-2}\; \, \xb_{n-1}\right ]^{\tt T}$  is a generalized right eigenvector of $\Mf_\VV$ corresponding to the eigenvalue $\lambda_{j,r} = q^r(q^j + q^{-j})$, 
and $\McV \xb_{j,r} = \lambda_{j,r} \xb_{j,r} + \vb_{j,r}$, where $\vb_{j,r}$ is the right eigenvector for $\McV$ in Proposition
\ref{prop:rteigenvector}.
\end{theorem} 

\begin{proof} The proof amounts to showing that the matrix equation below holds 
 \begin{equation}
\label{eq:McK}  
\Mf_\VV \mathbf{x}_{j,r}  =\left( \begin{matrix} 0 & \mathrm{I} & 0  & & \cdots & 0 & 0 \\ 
 \mathrm{Z} & 0 & \mathrm{I}  & & \cdots & 0 & 0 \\
 0 &  \mathrm{Z} & 0 & \mathrm{I} & \cdots & 0 & 0 \\
\vdots  & \vdots  & \mathrm{Z}   & \ddots  & \ddots & 0 & 0 \\ 
 \vdots & \vdots & \vdots & &  \ddots &  \mathrm{I} & 0 \\
0 & 0 & 0 & 0  \cdots &  \mathrm{Z} & 0 &  \mathrm{I} \\ 
2\mathrm{I}  & 0 & 0 & 0 & \cdots &  2\mathrm{Z} & 0 
  \end{matrix} \right) \left [ \begin{matrix}  \mathbf{x}_0 \\ \mathbf{x}_1 \\  \mathbf{x}_2  \\ \vdots \\
 \mathbf{x}_{n-3} \\ \mathbf{x}_{n-2}\\ \mathbf{x}_{n-1}\end{matrix} \right]  = \lambda_{j,r} \left [ \begin{matrix}  \mathbf{x}_0 
  \\ \mathbf{x}_1 \\  \mathbf{x}_2  \\ \vdots \\
 \mathbf{x}_{n-3} \\  \mathbf{x}_{n-2}\\ \mathbf{x}_{n-1}\end{matrix} \right] + \left [ \begin{matrix}  \mathbf{v} 
  \\ q^r\Uc_1\mathbf{v}  \\  q^{2r}\Uc_2\mathbf{v}  \\ \vdots \\
q^{(n-3)r} \Uc_{n-3}\mathbf{v} \\  q^{(n-2)r}\Uc_{n-2}\mathbf{v}\\ q^{(n-1)r}\Uc_{n-1}\mathbf{v}\end{matrix} \right], 
\end{equation}
when $\xb_{j,r}$ has vector components given by \eqref{eq:xkrecur}.   

Row 0 of $\McV \xb_{j,r}$ says that $\xb_1 = \lambda_{j,r}\xb_0 + \vb
=  q^r\Uc_1 \xb_0 + \vb$, which is true for $\xb_1$ in \eqref{eq:xkrecur}.   Next we compute rows $k=2,\dots, n-1$ of  $\McV \xb_{j,r}$
proceeding  by induction to show  that $\xb_k = \lambda_{j,r}\xb_{k-1}
- \Zr x_{k-2} + q^{(k-1)r} \Uc_{k-1} \vb$ must hold for $2 \le k \le  n-1$. 
To facilitate this, we write \begin{equation}\label{eq:xb}\xb_k = q^{kr} \Uc_k \vb + q^{(k-1)r} \sum_{s=0}^{\lfloor \frac{k-1}{2}\rfloor} (k-2s) \Uc_{k-1-2s} \vb =
 q^{kr} \Uc_k \vb + q^{(k-1)r}\Sigma(k), \quad \text{where} \end{equation}
\begin{equation}\label{eq:Sig}\Sigma(k) = \sum_{s=0}^{\lfloor \frac{k-1}{2}\rfloor} (k-2s) \Uc_{k-1-2s} \vb.\end{equation}
Now for row $k$ of  $\McV \xb_{j,r}$ verifying that \eqref{eq:xb} holds involves using the Chebyshev recursion 
$\Uc_{k+1} = (q^j + q^{-j})\Uc_k - \Uc_{k-1}$ for $1\le k \le n-1$ (which we write 
$\Uc_{k+1} = \Uc_1\Uc_k - \Uc_{k-1}$ here for the sake of brevity) and showing that
\begin{align}\begin{split}\label{eq:xk1}  \xb_{k+1} & = \lambda_{j,r} \xb_k - \Zr \xb_{k-1} + q^{kr} \Uc_k \vb \\
  & = q^{(k+1)r}\Uc_1 \Uc_k \vb + q^{kr}\Uc_1\Sigma(k) - q^{2r} \xb_{k-1} + q^{kr} \Uc_k \vb,\\
& =  q^{(k+1)r}\left( \Uc_{k+1}\vb + \Uc_{k-1}\vb\right) + q^{kr}\left( \Uc_1\Sigma(k) + \Uc_k\vb\right)  
 -q^{2r} \left(q^{(k-1)r} \Uc_{k-1}\vb + q^{(k-2)r} \Sigma(k-1)\right) \\
 & = q^{(k+1)r} \Uc_{k+1}\vb + q^{kr}\left( \Uc_1\Sigma(k) -\Sigma(k-1) + \Uc_k\vb\right).\end{split} \end{align}
We see that \eqref{eq:xb} will hold for $k+1$,  if we can show that    
\begin{equation}\label{eq:sigma} \Sigma(k+1) =  \Uc_1\Sigma(k) - \Sigma(k-1) + \Uc_k \vb, \qquad (\Uc_1 = q^{j}+q^{-j}).\end{equation}
\subsubsection*{\bf Case $k$ odd, $k = 2t+1$ for $t \ge 1$}  We start from the right-hand side and use the Chebyshev
recursion relation.  When we encounter a term $\Uc_\ell\vb$ with $\ell < 0$, we assume it is 0 and drop it from the equation.
We use the fact that $\lfloor \frac{k-1}{2}\rfloor = t = \lfloor \frac{k}{2}\rfloor$ and $\lfloor \frac{k-2}{2}\rfloor = t-1$.  Then
\begin{align*}   \Uc_1\Sigma(k) - \Sigma(k-1) + \Uc_k \vb\\
 & \hspace{-2.7cm} = \Uc_1\sum_{s=0}^t (k-2s)\Uc_{k-1-2s}\vb - \sum_{s=0}^{t-1} (k-1-2s)
\Uc_{k-2-2s}\vb + \Uc_k \vb \\
& \hspace{-2.7cm} = \sum_{s=0}^t (k-2s)\Uc_{k-2s}\vb + \sum_{s=0}^{t-1}(k-2s)\Uc_{k-2-2s}\vb - \sum_{s=0}^{t-1}(k-1-2s)\Uc_{k-2-2s}\vb + \Uc_k \vb \\
& \hspace{-2.7cm}  = \sum_{s=0}^t (k-2s)\Uc_{k-2s}\vb + \sum_{s=0}^{t-1} \Uc_{k-2-2s} + \Uc_k \vb \\
& \hspace{-2.7cm}  = \sum_{s=0}^t (k+1-2s)\Uc_{k-2s}\vb - \sum_{s=0}^t \Uc_{k-2s}\vb +  \sum_{s=0}^{t-1} \Uc_{k-2-2s}\vb + \Uc_k \vb\\
& \hspace{-2.7cm}  = \sum_{s=0}^t (k+1-2s)\Uc_{k-2s}\vb = \sum_{s=0}^{\lfloor \frac{k}{2}\rfloor} (k+1-2s)\Uc_{k-2s}\vb =\Sigma(k+1).\end{align*}
\subsubsection*{\bf Case $k$ even.} Since the argument just requires minor
adjustments  to the one above when $k$ is even, we omit the proof.

What remains to be done is to compute the last row of $\McV \xb_{j,r}$, and to show that
$2 \xb_0 + 2 \Zr \xb_{n-2} = \lambda_{j,r} \xb_{n-1} + q^{(n-1)r}\Uc_{n-1}\vb$.
We will use  the relations $\Zr \mathbf{x}_k = q^{2r}\mathbf{x}_k$ for all $0 \le k \le n-1$  and 
$\Uc_{n-1}\vb = \mathbf{0}$ and $\Uc_n \vb = \vb$
which come from  $\Uc_k \mathbf{v} = \frac{q^{(k+1)j} - q^{-(k+1)j}}{q^j-q^{-j}}\mathbf{v}$ when $k=n-1, n$.   Since 
$\Uc_{n-1}\vb =  \mathbf{0}$, what we end up showing is that $2 \xb_0 + 2 \Zr \xb_{n-2} -\lambda_{j,r} \xb_{n-1} =  \mathbf{0}$. 
Now the computation in \eqref{eq:xk1} for $k = n-1$ with a little rearranging and   
with   $\Uc_1\Sigma(n-1) -\Sigma(n-2) + \Uc_{n-1}\vb$  replaced with $\Sigma(n) = \sum_{s=0}^{\frac{n-1}{2}}
(n-2s)\Uc_{n-1-2s} \mathbf{v}$   implies
\begin{align*} \Zr \mathbf{x}_{n-2} - \lambda_{j,r} \mathbf{x}_{n-1} &= q^{(n-1)r} \Uc_{n-1}\mathbf{v} - q^{nr}\Uc_n\mathbf{v} - q^{(n-1)r}\sum_{s=0}^{\frac{n-1}{2}}
(n-2s)\Uc_{n-1-2s} \mathbf{v} \\
&= -\mathbf{v} - q^{(n-1)r}\sum_{s=1}^{\frac{n-1}{2}}
(n-2s)\Uc_{n-1-2s} \mathbf{v}  \qquad \text{(using $\Uc_{n-1} \mathbf{v} = \mathbf{0}$ and $\Uc_n \mathbf{v} = \mathbf{v}$)}.
\end{align*}
Therefore, for the last row we have
\begin{align*} 2 \mathbf{x}_0 + 2 \Zr  \mathbf{x}_{n-2} - \lambda_{j,r}  \mathbf{x}_{n-1} & = 
2 \mathbf{v} + q^{2r} \mathbf{x}_{n-2} -\mathbf{v} - q^{(n-1)r}\sum_{s=1}^{\frac{n-1}{2}}
(n-2s)\Uc_{n-1-2s} \mathbf{v}  \\
&\hspace{-3.7cm}= \mathbf{v} + q^{2r}\Bigg(q^{(n-2)r}\Uc_{n-2}\mathbf{v} + q^{(n-3)r} 
\sum_{s=0}^{\frac{n-3}{2}}
(n-2-2s)\Uc_{n-3-2s} \mathbf{v} \Bigg)    - q^{(n-1)r}\sum_{s=1}^{\frac{n-1}{2}}
(n-2s)\Uc_{n-1-2s} \mathbf{v}  \\
&\hspace{-3.7cm}= \mathbf{v} + \Uc_{n-2} \mathbf{v} 
= \mathbf{v} + \Uc_{n-2} \mathbf{v}  - \frac{1}{2}\Uc_1\Uc_{n-1} \mathbf{v} 
= -\frac{1}{2}\mathsf{p}_n(q^j + q^{-j})\mathbf{v} =  \mathbf{0}, 
\end{align*}
since $\lambda_{j,0} = q^j + q^{-j}$ is a root of $\mathsf{p}_n(t)$.
\end{proof}

\subsection{Right eigenvectors from grouplike elements of $\Df_n$}\label{S3.6}
In this section, we focus on the grouplike elements $b^ic^k$, $0 \le i,k \le n-1$,  of $\Df_n$ and compute the trace vectors
$\Trs(b^ic^k)$ explicitly.  Relation
\eqref{eq:grouplike} with $g = b^ic^k$ says  $\Mf_\VV \Trs(b^ic^k) =  \tr_{\VV}(b^ic^k) \Trs(b^ic^k)$,
so that $\Trs(b^ic^k)$ is a right eigenvector of eigenvalue $ \tr_{\VV}(b^ic^k)$ 
for the McKay matrix $\McV$, $\VV = \VV(2,0)$.  We identify the eigenvalue $\tr_{\VV}(b^ic^k)$ with $\lambda_{j,r} = q^r(q^j + q^{-j})$
for certain values of $j$ and $r$ that depend on $i$ and $k$.  Since there is a unique right eigenvector of $\McV$ corresponding to the
eigenvalue $\lambda_{j,r}$  up to scalar multiples by Theorem \ref{thm:genright}, by comparing the vectors $\Trs(b^ic^k)$ with the vectors $\mathbf{v}_{j,r}$
in Section \ref{S3.4}, we obtain an expression for the characters $\eta_{\ell,s}$ of the simple $\Df_n$-modules $\VV(\ell,s)$ evaluated on the grouplike elements $b^ic^k$ of $\Df_n$  in terms of
Chebyshev polynomials of the second kind.   The trace vectors  $\Trs(b^ic^k) = \Trs(b^{-k}c^{-i})$ for
$i,k \in \ZZ_n$  are shown to give  a complete set of right eigenvectors for $\McV$.  

The simple $\Df_n$-module $\VV(\ell,s)$, $1 \leq \ell \leq n$ and $s \in \ZZ_n$,    
has a basis $\{v_1, v_2, \ldots, v_\ell\}$ with $\Df_n$-action  prescribed by  \eqref{eq:action}, which implies
that  $b^ic^k$ has the following character value on $\VV(\ell,s)$:
\begin{equation}\label{eq:trace}\eta_{\ell,s}(b^ic^k) := \tr_{\VV(\ell,s)}(b^ic^k) = \sum_{t=1}^\ell q^{(s+t-1)i+(t-(s+\ell))k} = q^{(s-1)i-(s+\ell)k} \sum_{t=1}^\ell q^{t(i+k)}.\end{equation} 
We assume that the simple modules of $\Df_n$ are ordered so that for a fixed grouplike element $g$, 
\begin{align}\begin{split}
\label{eq:etag}
&\Trs(g) =\left[\mathbf{u}_1\;\, \mathbf{u}_2 \;\, \ldots \;\, \mathbf{u}_{n-1}  \;\, \mathbf{u}_n \right]^{\tt T}, \quad 
\text{where} \\
&\mathbf{u}_\ell = \left[\eta_{\ell,0}(g)\;\, \eta_{\ell,1}(g) \;\, \ldots \;\, \eta_{\ell,n-1}(g) \right] \quad
\text{for $1 \le \ell \le n$}.
\end{split}\end{align}

\begin{theorem}\label{thm:rightevs}
Assume $i,k \in \ZZ_n$, and $\VV = \VV(2,0)$. Then 
\begin{itemize}
 \item[{\rm (a)}] $\Trs(b^ic^k)$ is a right eigenvector for $\Mf_\VV$  of
eigenvalue $\lambda_{j,r}=\eta_{2,0}(b^ic^k) = q^{i}+ q^{-k}$ if
\begin{align}\label{eq:rj}
 j=\pm \frac{i+k}{2} \hspace{.1 in} (\modd n) \qquad \quad  r = \frac{i-k}{2} \hspace{.1 in} (\modd n),
\end{align}
or equivalently,  if  for $0 \le j \le \frac{n-1}{2}$ and $r \in \ZZ_n$, 
\begin{equation}\label{eq:ik} \begin{array}{c} i = j+r \\
k = j-r \end{array} \hspace{.1 in} (\modd n),\qquad \text{or} \qquad  \begin{array}{c} i =-j+r \\
k = -j-r \end{array} \hspace{.1 in} (\modd n).
\end{equation}
\item[{\rm (b)}] $\Trs(b^ic^k) = \mathbf{v}_{j,r}$,  and  for all
$1\le \ell \le n$  and $s \in \ZZ_n$
\begin{equation}\label{eq:charCheby} \eta_{\ell,s}(b^ic^k) = q^{(\ell +s -1)r}\Uc_{\ell-1}(q^j + q^{-j}),\end{equation}
where $j$ and $r$ are as in \eqref{eq:rj}, and $\mathbf{v}_{j,r}$ is as in Proposition 
\ref{prop:rteigenvector}.
\item[{\rm (c)}]  $\Trs(b^ic^k)= \Trs(b^{-k}c^{-i})$ for all $i,k \in \ZZ_n$.  Therefore, the character $\eta_{\ell,s}$ of $\VV(\ell,s)$ satisfies
\begin{equation}\eta_{\ell,s}(b^{i}c^{k}) = \eta_{\ell,s}(b^{-k}c^{-i})\end{equation}
for 
all $1\le \ell \le n$ and $s \in \ZZ_n$.
\item[{\rm (d)}] The vectors $\Trs(b^ic^k)= \Trs(b^{-k}c^{-i})$ for $i,k \in \ZZ_n$ give a complete set of right eigenvectors for the
McKay matrix $\McV$ for tensoring with the $\Df_n$-module $\VV=\VV(2,0)$.
\end{itemize}
\end{theorem} 

\begin{proof}  We already know that  $\Trs(b^ic^k)$ is a right eigenvector of $\Mf_\VV$  of
eigenvalue $\eta_{2,0}(b^ic^k)$ for $i,k \in \ZZ_n$,  (see (1) of Corollary \ref{rem:coprod}).  

(a)  By \eqref{eq:trace} with $\ell = 2$ and $s=0$,  
\begin{align}\label{eq:eta20}
\eta_{2,0}(b^ic^k) =q^{i}+q^{-k}=q^{\frac{i-k}{2}}(q^{\frac{i+k}{2}}+q^{-\frac{i+k}{2}}) = q^r(q^j + q^{-j}) = \lambda_{j,r},
\end{align}
when $j$ and $r$ are as in \eqref{eq:rj}.  Conversely, given $r \in \ZZ_n$ and $0\le j \le \frac{n-1}{2}$,    
it is easy to verify that $\eta_{2,0}(b^i c^k) = \lambda_{j,r}$
 for the specified values of
$i$ and $k$ in \eqref{eq:ik}.

(b) Let  $\Trs(g)$ and  $\mathbf{u}_\ell$ be as in   \eqref{eq:etag}  for $g = b^ic^k$, and assume $2r = i-k$ and $2j = i+k\, (\modd n)$.
We  argue first  that $\mathbf{u}_1$ is the right eigenvector  $\left[1\;\,q^{2r}\;\, q^{4r}\;\, \dots\;\, q^{2(n-1)r}\right]^{\tt T}$
for the cyclic $n \times n$ matrix $\Zr$ with corresponding eigenvalue $q^{2r}$.
For this, observe that by \eqref{eq:trace},  
\begin{equation} \label{eq:eta1}  
\eta_{1,s}(b^ic^k)=q^{(s-1)i-(s+1)k}q^{i+k} = q^{si-sk}.  
\end{equation}
Thus, $\displaystyle \frac{\eta_{1,s}(b^ic^k)}{\eta_{1,s-1}(b^ic^k)} =q^{i-k}=q^{2r}$ for all $s\in \ZZ_n$. Since $\eta_{1,0}(b^ic^k) = 1$,  we
have  
\begin{equation*}\label{eq:u}\mathbf{u}_1 =\left[1\;\,q^{2r}\;\, q^{4r}\;\, \dots\;\, q^{2(n-1)r}\right]^{\tt T},\end{equation*}
which is an eigenvector for $\Zr$ of eigenvalue $q^{2r}$, and  $\eta_{1,s}(b^ic^k) = q^{2sr}$, for all $s\in \ZZ_n$. 

Now the first vector components $\mathbf{u}_1$ and $\mathbf{v}_0$ of $\Trs(g)$ and  $\mathbf{v}_{j,r}$
 are identical, and the subsequent vector components $\mathbf{u}_\ell$ of $\Trs(g)$ must satisfy the same
 relations as the vector components $\mathbf{v}_{\ell-1}$ in the proof of Proposition \ref{prop:rteigenvector} for $\ell \ge 1$. Thus,   
 $\Trs(b^ic^k)=\mathbf{v}_{j,r}$, where $j$ and $r$ are as in
\eqref{eq:rj} and   
\begin{equation}\mathbf{u}_\ell = \mathbf{v}_{\ell-1} = q^{(\ell-1)r}\Uc_{\ell-1}(q^j + q^{-j})\mathbf{v}_0 = 
q^{(\ell-1)r}\Uc_{\ell-1}(q^j + q^{-j})\mathbf{u}_1.\end{equation}
Equating component $s$ on both sides for $0 \le s \le n-1$ gives the assertion in \eqref{eq:charCheby} --  the character
value $\eta_{\ell,s}(b^ic^k)$ of $b^ic^k$ on $\VV(\ell,s)$  is given by
$\eta_{\ell,s}(b^ic^k) = q^{(\ell +s -1)r}\Uc_{\ell-1}(q^j + q^{-j})$,
where $j$ and $r$ are as in \eqref{eq:rj}.

(c)  We have seen in the proof of (a) that 
$\eta_{2,0}(b^i c^k) = q^i +q^{-k}$ for any $i,k \in \ZZ_n$.   Therefore,  $\eta_{2,0}(b^{-k}c^{-i}) = q^{-k} + q^i = \eta_{2,0}(b^i c^k)$. 
Since $\Trs(b^ic^k)$ and $\Trs(b^{-k}c^{-i})$ are two right eigenvectors of $\Mf_\VV$ with the same eigenvalue,  and since  
 $\eta_{1,s}(b^ic^k) = q^{si-sk} = \eta_{1,s}(b^{-k}c^{-i})$ for all $s\in \ZZ_n$  follows from \eqref{eq:eta1},  we obtain as in the proof of (b)
that  $\Trs(b^ic^k)= \Trs(b^{-k}c^{-i})$.    As a result, 
 $b^{i}c^{k}$ and $b^{-k}c^{-i}$ have the same
character value on any simple $\Df_n$-module $\VV(\ell,s)$,  that is $\eta_{\ell,s}(b^ic^k)= \eta_{\ell,s}(b^{-k}c^{-i})$ for all $1\le \ell 
\le n$, and $s \in \ZZ_n$. 

 (d) We know that the McKay matrix $\McV$ has $\frac{n(n+1)}{2}$ distinct eigenvalues $\lambda_{j,r}$, and there are
 $\frac{n(n+1)}{2}$ right eigenvectors $\Trs(b^ic^k), i,k \in \ZZ_n$, giving all these distinct eigenvalues, since for a given
 pair $j,r$ we can take $i = j+r\, (\modd n)$ and $k = j-r \, (\modd n)$ as in (a).    Consequently, the vectors $\Trs(b^ic^k)$
 give a complete set of right eigenvectors for $\McV$.  
 \end{proof}

\subsection{Generalized right eigenvectors as trace vectors } \label{S3.7} The goal of this section is to show that generalized right eigenvectors for the McKay  matrix $\mathsf{M} =\McV$ for $\VV = \VV(2,0)$ can also be realized  as trace vectors on simple modules,
but for traces of non-grouplike elements. Generalized eigenvectors occur only for
the eigenvalues $\lambda_{j,r}$ with $j \ne 0$, and the corresponding (generalized) eigenspace is
two-dimensional in that case.  We will use the coproduct expression for the trace in \eqref{eq:matprodtrace} and will 
require quantum integers $[\ell]=1+q+\cdots+q^{\ell-1}$,
the quantum factorial $[\ell]!=[\ell][\ell-1]\cdots[1]$,  and the quantum binomial coefficient
$${\ell \brack i}=\frac{[\ell]!}{[i]!\,[\ell-i]!}$$ 
for $\ell,i \in \ZZ_{\ge 0}$, $\ell \ge i$, where $[0] = [0]! = 1$ is understood.  

Chen has studied a family of Hopf algebras $\Hf(p,q)$, where $p$ and $q$ are arbitrary scalars,
and Lemma 2.7 of \cite{Chen99} gives an expression for  coproduct for the algebra $\Hf(p,1)$.  
When $q$ is an $n$th root of unity, $\Hf(p,q)$ modulo a certain Hopf ideal is 
a finite-dimensional quasi-triangular Hopf algebra $\Hf_n(p,q)$, and the algebra $\Hf_n(1,q)$ is isomorphic to $\Df_n$.
Chen's coproduct formula can be modified by replacing binomial
coefficients with their $q$ analogues to give a coproduct formula for specific
powers of the generators of $\Df_n$.  We will use these specializations in the next result.

\begin{lemma} For $\ell \in \ZZ_{\ge 0}$,
\begin{align*}
\Delta(d^\ell a^\ell)=&\sum_{t=0}^\ell {\ell\brack t}^2d^ta^t \otimes b^tc^td^{\ell-t}a^{\ell-t} +\text{nilpotent terms}\\
=& \; d^\ell a^\ell\otimes b^\ell c^\ell + [\ell]^2d^{\ell-1}a^{\ell-1} \otimes b^{\ell-1}c^{\ell-1}da+\text{nilpotent terms}.
\end{align*}
By nilpotent terms, we mean terms $x\otimes y$ such that $y$ acts as a nilpotent endomorphism on $\VV$.
\end{lemma}
\noindent {\it Proof.}   From \cite[Lemma 2.7]{Chen99}, 
we deduce \begin{align*}
\Delta(a^\ell)=\sum_{i=0}^\ell  {\ell \brack i} a^i \otimes a^{\ell-i}b^i, \hspace{.5 in}
\Delta(d^\ell)=\sum_{i=0}^{\ell}{\ell\brack i}d^i\otimes c^id^{\ell-i}.
\end{align*}
Recall that $ba=qab$, $db=qbd$ in $\Df_n$. Therefore,
\begin{align}\begin{split}\label{eq:dellaell}
\Delta(d^\ell a^\ell)=&  \left(\sum_{j=0}^{\ell} {\ell\brack j}d^j\otimes c^jd^{\ell-j}\right)\left(\sum_{i=0}^\ell {\ell \brack i} a^i \otimes a^{\ell-i}b^i  \right)\\
=& \sum_{t=0}^\ell {\ell \brack t }^2 d^ta^t \otimes c^td^{\ell-t} a^{\ell-t}b^t  +\text{nilpotent terms}\\
=&\sum_{t=0}^\ell {\ell \brack t}^2 d^ta^t \otimes b^tc^td^{\ell-t}a^{\ell-t} +\text{nilpotent terms}.
\end{split}
\end{align}
The term $b^tc^td^{\ell-t}a^{\ell-t}$ is nilpotent on $\VV$ only if $\ell-t=0$ or $\ell-t=1$, i.e. $t=\ell$ or $t=\ell-1$, hence
\begin{equation*}
 \hspace{.95truein} \Delta(d^\ell a^\ell) = d^\ell a^\ell\otimes b^\ell c^\ell + [\ell]^2 d^{\ell-1}a^{\ell-1} \otimes b^{\ell-1}c^{\ell-1}da+\text{nilpotent terms}.  \qquad \qquad \qquad  \square
\end{equation*} 

The following result gives a formulation of the generalized right eigenvectors for $\Mf_\VV$, $\VV = \VV(2,0)$,  as trace vectors on the simple modules.
Recall that $\{\Trs(b^ic^k)\}$ gives a list of right eigenvectors for $\Mf_\VV$, with repetitions described by Theorem~\ref{thm:rightevs} (c).
 
\begin{theorem}\label{thm:genrt-traces}
Given $0 \leq i,k \leq n-1$ with $k \ne -i\, (\modd n)$,  choose $1\leq s \leq n-1$ such that $s=-(i+k) \, (\modd n)$. Let $\gamma_1,\dots,\gamma_s$ be defined recursively by $\gamma_s=1$ and 
\begin{align*}
\gamma_\ell=\frac{[\ell+1]^2q^{-1-\ell+s}}{[\ell][s-\ell](q-1)}\gamma_{\ell+1}
\end{align*}
for $\ell=s-1,s-2, \ldots,1$.  
Then $\sum_{\ell=1}^s \gamma_\ell \Trs(b^ic^kd^\ell a^\ell)$ is in the generalized eigenspace of $\Mf = \McV$, $\VV=\VV(2,0)$,  containing $\Trs(b^ic^k)$.
\end{theorem}  
\noindent {\it Proof.}  From the coproduct expression in \eqref{eq:dellaell} and Lemma \ref{lem:hopftrace}\,(c), we have 
\begin{align*}
\Mf\Trs(d^\ell a^\ell)=& \tr_\VV(b^\ell c^\ell)\Trs(d^\ell a^\ell)+[\ell]^2\tr_\VV(b^{\ell-1}c^{\ell-1}da)\Trs(d^{\ell-1}a^{\ell-1}).
\end{align*}
And similarly
\begin{align*}
\Mf \Trs(b^ic^k d^\ell a^\ell)=& \tr_\VV(b^{\ell+i}c^{\ell+k})\Trs(b^ic^kd^\ell a^\ell)+\tr_\VV(b^{\ell+i-1}c^{\ell+k-1}da)[\ell]^2\Trs(b^ic^kd^{\ell-1}a^{\ell-1}).
\end{align*} 
Now from \eqref{eq:eta20} we know that $\tr_\VV(b^i c^k) = q^{i}+q^{-k}$ for any $i,k \in \ZZ_n$.   Using that fact and the action of
the generators of $\Df_n$ on the basis $\{v_1,v_2\}$ of $\VV$ in \eqref{eq:action}, we have 
\begin{align*}
&\tr_\VV(b^{\ell+i}c^{\ell+k})=q^{\ell+i}+q^{-\ell-k},\\
&da.v_1=\alpha_1(2)v_1=(1-q^{-1})v_1, \hspace{.3 in} b^{\ell+i-1}c^{\ell+k-1}.v_1=q^{-\ell-k+1}v_1,\\
&\tr_\VV(b^{\ell+i-1}c^{\ell+k-1}da)=(1-q^{-1})q^{-\ell-k+1}.
\end{align*}
Therefore,
\begin{equation}\label{eq:bcda}
\mathsf{M}\Trs(b^ic^kd^\ell a^\ell)=(q^{\ell+i}+q^{-\ell-k})\Trs(b^ic^k d^\ell a^\ell)+[\ell^2](1-q^{-1})q^{-\ell-k+1}\Trs(b^ic^kd^{\ell-1}a^{\ell-1}).
\end{equation}
By a change of the index of summation on the second summand, 
\begin{align*}
\mathsf{M}\left(\sum_{\ell=1}^s\gamma_\ell\Trs(b^ic^kd^\ell a^\ell)\right)& \\
& \hspace{-1.7cm} = \sum_{\ell=1}^s (q^{\ell+i}+q^{-\ell-k})\Trs(b^ic^kd^\ell a^\ell)\gamma_\ell+\sum_{\ell=0}^{s-1}[\ell+1]^2(1-q^{-1})q^{-\ell-k}\Trs(b^ic^kd^\ell a^\ell)\gamma_{\ell+1}\\
& \hspace{-1.7cm} = (1-q^{-1})q^{-k}\Trs(b^ic^k)\gamma_1+(q^{s+i}+q^{-s-k})\gamma_s \Trs(b^ic^kd^sa^s) \\
& \hspace{-1.1cm} +\sum_{\ell=1}^{s-1}\left( (q^{\ell+i}+q^{-\ell-k})\gamma_\ell+[\ell+1]^2(1-q^{-1})q^{-\ell-k}\gamma_{\ell+1} \right)\Trs(b^ic^kd^\ell a^\ell).
\end{align*}
We claim this is equal to 
\begin{align*}
\lambda\left(\sum_{\ell=1}^s \gamma_\ell \Trs(b^ic^kd^\ell a^\ell)\right)+\text{constant}\cdot \Trs(b^ic^k),
\end{align*}
for $\lambda=q^{i}+q^{-k}=\tr_\VV(b^ic^k) = \lambda_{j,r}$,  where $j$  and $r$ 
are as in \eqref{eq:rj}. 

First, the terms agree at $\ell=s$, where we recall that $s = -(i+k)\, (\modd n)$ so that
\begin{align*}
\lambda=q^{s+i}+q^{-s-k}=q^{-k}+q^{i}. 
\end{align*}
When $1\leq \ell\leq s-1$, the following is true
\begin{align*}
&(q^{\ell+i}+q^{-\ell-k})\gamma_\ell+[\ell+1]^2(1-q^{-1})q^{-\ell-k}\gamma_{\ell+1}=\lambda\gamma_\ell=\gamma_\ell(q^i+q^{-k}),
\end{align*}
if and only if the following is true: 
\begin{align*}\hspace{.8truein}
\gamma_\ell=&\frac{[\ell+1]^2(1-q^{-1})q^{-\ell-k}\gamma_{\ell+1}}{q^i+q^{-k}-q^{\ell+i}-q^{-\ell-k}}=\frac{[\ell+1]^2(q-1)q^{-1-\ell-k}\gamma_{\ell+1}}{q^i(1-q^\ell)-q^{-\ell-k}(1-q^\ell)}\\
=&\frac{[\ell+1]^2q^{-1-\ell-k}}{[\ell](q^{-\ell-k}-q^{-s-k})}\gamma_{\ell+1}=\frac{[\ell+1]^2q^{-1-\ell-k}}{[\ell]q^{-s-k}(q^{s-\ell}-1)}\gamma_{\ell+1} =\frac{[\ell+1]^2q^{-1-\ell+s}}{[\ell][s-\ell](q-1)}\gamma_{\ell+1}. \qquad \quad \; \square \end{align*} 
  
\subsection{Left eigenvectors and generalized left eigenvectors of $\McV$, $\VV= \VV(2,0)$}\label{S3.8}
In this section, we compute left (generalized) eigenvectors of the McKay matrix $\Mf = \McV$ for $\VV = \VV(2,0)$ using modified  Chebyshev polynomials $\Lc_j(t)$ that are defined recursively by 
\begin{equation}\label{eq:Ldef}
\Lc_0(t)=2,  \quad 
\Lc_1(t) =t, \quad  
\Lc_k(t) =t\Lc_{k-1}(t) - \Lc_{k-2}(t) \;  \text{ for } k \geq 2.
\end{equation}

\begin{proposition}
\label{prop:Li}
Assume $x \neq 1$,  and let $t=x+x^{-1}$ as in Proposition \ref{prop:qrels}. Then 
\begin{itemize}
\item[{\rm (a)}]   For $k \ge 2$, \; $\Lc_k(t) = \Uc_{k}(t) - \Uc_{k-2}(t)$.
\item[{\rm (b)}]   For all $k \geq 0$, \; $\Lc_k(t) = x^k + x^{-k}$.
\end{itemize}
\end{proposition} 

\begin{proof} (a)  We have $\Lc_2(t) = t\Lc_1(t) - \Lc_0(t) = t^2-2 = t^2-1-1 = \Uc_2(t)-\Uc_0(t)$, and for $k >2$ the proof of (a) is an
easy inductive argument starting from this base case.
(b) The relation $\Lc_k(t) = x^k + x^{-k}$ clearly holds for $k=0,1$.      
Proposition \ref{prop:qrels} (a)  says that  $\Uc_k(t) = x^k + x^{k-2} + \cdots + x^{-(k-2)} + x^{-k}$ for all $k \ge 2$. Part (b) follows readily from that and part (a).
\end{proof}
  
\begin{proposition}
\label{left eigenvector} For $r \in \ZZ_n$, let
$\mathbf{w}_0 \ne \mathbf{0}$ be a left eigenvector of $\Zr$ corresponding to the eigenvalue $q^{2r}$. 
Assume $0 \le j \le \frac{n-1}{2}$,  and set $\mathbf{w}_k = q^{kr}\Lc_k(q^j+q^{-j}) \mathbf{w}_0 = q^{kr}(q^{kj}+q^{-kj}) \mathbf{w}_0$ for  
$1\le k \le n-1$, where $\Lc_k$ is the modified Chebyshev polynomial defined in
\eqref{eq:Ldef}.  
Then $\mathbf{w}_{j,r}=\left[\mathbf{w}_{n-1}\;\, \mathbf{w}_{n-2} \;\,\ldots \;\,\mathbf{w}_1\, \; \mathbf{w}_0\right]$ is a left eigenvector of $\Mf$ corresponding to eigenvalue $\lambda_{j,r}=q^r(q^j+q^{-j})$. 
\end{proposition} 

\noindent{\it Proof.} The proof amounts to comparing the left and right sides of $\mathbf{w}_{j,r} \Mf = \lambda_{j,r} \mathbf{w}_{j,r}$
and showing that equality holds when 
$\mathbf{w}_k = q^{kr}\Lc_k(q^j+q^{-j}) \mathbf{w}_0 = q^{kr}(q^{jk}+q^{-jk}) \mathbf{w}_0$, where the
last ``=''  results from  
$\Lc_k(q^j+q^{-j}) =q^{jk}+q^{-jk}$, which is a direct consequence of Proposition \ref{prop:Li} (b).    

We expand the left-hand side of $\mathbf{w}_{j,r} \Mf$ starting from the rightmost column and ask if the two sides
of $\mathbf{w}_{j,r} \Mf = \lambda_{j,r} \mathbf{w}_{j,r}$ are equal:
\begin{align*}
\mathbf{w}_1 &\isit \lambda_{j,r} \mathbf{w}_0 = q^{r}(q^j+q^{-j}) \mathbf{w}_0 =  q^{r}\Lc_1(q^j+q^{-j}) \mathbf{w}_0, \\
\mathbf{w}_2 &\isit \lambda_{j,r}\mathbf{w}_1 - \mathbf{w}_0(2\Zr) = \mathbf{w}_0 (\lambda_{j,r}^2-2\Zr)  = \mathbf{w}_0 (\lambda_{j,r}^2-2q^{2r})\\ &= \mathbf{w}_0 q^{2r}\Big((q^j+q^{-j})^2 - 2\Big)= q^{2r}\Lc_2(q^j+q^{-j})\mathbf{w}_0,  \\
\mathbf{w}_k &\isit \lambda_{j,r} \mathbf{w}_{k-1} - \mathbf{w}_{k-2}\Zr = \lambda_{j,r} q^{(k-1)r}\Lc_{k-1}(q^j+q^{-j})\mathbf{w}_0 - q^{(k-2)r}\Lc_{k-2}(q^j+q^{-j})\mathbf{w}_0\Zr \\
&=q^{kr}\bigg((q^j+q^{-j})\Lc_{k-1}(q^j+q^{-j}) - \Lc_{k-2}(q^j+q^{-j})\bigg) \mathbf{w}_0 = q^{kr} \Lc_k(q^j+q^{-j}) \mathbf{w}_0, \; (3 \le k \le n-1).
\end{align*} 
The leftmost column involves comparing $\mathbf{w}_{n-2}\Zr + 2\mathbf{w}_0$ with $\lambda_{j,r}\mathbf{w}_{n-1}$, which we
do by showing that 
\begin{align*}
\mathbf{w}_{n-2}\Zr + 2\mathbf{w}_0 - \lambda_{j,r}\mathbf{w}_{n-1} &= 
q^{(n-2)r} \Lc_{n-2}(q^j+q^{-j}) \mathbf{w}_0\Zr + 2\mathbf{w}_0 - \lambda_{j,r} \, q^{(n-1)r} \Lc_{n-1}(q^j+q^{-j}) \mathbf{w}_0 \\
&= q^{nr} \Lc_{n-2}(q^j+q^{-j}) \mathbf{w}_0 + 2\mathbf{w}_0 - q^{nr}(q^j+q^{-j})\Lc_{n-1}(q^j+q^{-j}) \mathbf{w}_0 \\
&= \big(q^{j(n-2)} + q^{-j(n-2)}\big) \mathbf{w}_0 + 2\mathbf{w}_0 - \big(q^j+q^{-j})(q^{j(n-1)} + q^{-j(n-1)}\big)  \mathbf{w}_0 \\
&= (q^{j(n-2)} + q^{-j(n-2)} + 2 - q^{jn} - q^{-j(n-2)} - q^{j(n-2)} - q^{-jn}) \mathbf{w}_0 = \mathbf{0}. \quad \;
\;  \square
\end{align*} 
 
The next corollary relates the above eigenvector results  to the dimension vectors.
\begin{corollary} \label{dim vector}
The dimension vector $\mathbf{s} = [\dimm(\Sf_1)\;\,\dimm(\Sf_2)\;\,  \dots\;\,\dimm(\Sf_{n^2})]^{\tt T}$   of  the  simple $\Df_n$-modules  is a right eigenvector corresponding to eigenvalue $\lambda_{0,0} = 2= \dimm \left(\VV(2,0)\right)$. The dimension vector $\mathbf{p}^{\tt T} = [\dimm(\Pf_1)\;\,\dimm(\Pf_2)\;\,\dots\;\,\dimm(\Pf_{n^2})]$
of the projective indecomposable $\Df_n$-modules 
  is a left eigenvector corresponding to the eigenvalue $\lambda_{0,0}$.   
\end{corollary} 
\begin{proof}   Observe that $\mathbf{v}_0 = \left [1\;1\;\, \cdots \;\,1\;1\right]^{\tt T}$ and $\mathbf{w}_0 = 
\mathbf{v}_0^{\tt T}$ are right and left (resp.) eigenvectors of the matrix $\Zr$ corresponding to the eigenvalue 1,  and  $\Uc_k(2) = k+1$ and  $\Lc_k(2) = 2$ for all values of $k \ge 0$.   
Therefore,  by Propositions \ref{prop:rteigenvector} and
\ref{left eigenvector},   $\mathbf{v}_k = (k+1)\mathbf{v}_0$  and $\mathbf{w}_k = 2 \mathbf{w}_0$ for all 
$1\le k\le n-1$.   Consequently,  $\mathbf{v}_{0,0} = \mathbf{s}$ and $\mathbf{w}_{0,0} =\frac{1}{n}\mathbf{p}^{\tt T}$, and the corresponding
eigenvalue is   $\lambda_{0,0} = 2= \dimm(\VV(2,0))$.    
\end{proof}

\begin{remark}{\rm
Corollary~\ref{dim vector} confirms the result in \cite[Sec.~3]{GHR}  mentioned in the Introduction in the
specific case of the Drinfeld double $\Df_n$ of the Taft algebra and the McKay matrix for tensoring with $\VV(2,0)$.}
\end{remark}
 
 \begin{proposition}\label{prop:genleft} Let $r \in\ZZ_n$, and 
 assume $\mathbf{w}= \mathbf{w}_0$, where $\mathbf{w}_0\Zr = q^{2r} \mathbf{w}_0$ 
 as in Proposition \ref{left eigenvector}.
Fix a choice of $j\in \{1,\dots, \frac{n-1}{2}\}$ and set  $\Uc_k =\Uc_k(q^j + q^{-j})$ for all $k \ge 0$.
Assume  $\mathbf{y}_0 = \mathbf{w}$,\; $\mathbf{y}_1 = 	q^r \Uc_1 \mathbf{w} + \mathbf{w}$, 
 and   $\mathbf{y}_{k} 
= q^{kr}( \Uc_k - \Uc_{k-2}\big)\mathbf{w}+ kq^{(k-1)r}\Uc_{k-1}\mathbf{w}$  for $2 \le k \le n-1$,  Then  
$\mathbf{y}_{j,r}:=  \left [\mathbf{y}_{n-1}\; \mathbf{y}_{n-2} \;\,\ldots \;\,\mathbf{y}_1 \; \mathbf{y}_0\right]$
is a generalized left eigenvector for $\McV$  corresponding to the eigenvalue $\lambda_{j,r}=q^r(q^{j}+q^{-j})$,
and  $\mathbf{y}_{j,r}$ satisfies

\begin{align}\begin{split} \label{eq:genleft} \mathbf{y}_{j,r}\Mf &=
 \left[\mathbf{y}_{n-1}\; \mathbf{y}_{n-2} \;\,\ldots \;\,\mathbf{y}_1 \; \mathbf{y}_0\right]
\left( \begin{matrix} 0 & \mathrm{I} & 0  & & \cdots & 0 & 0 \\ 
 \mathrm{Z} & 0 & \mathrm{I}  & & \cdots & 0 & 0 \\
 0 &  \mathrm{Z} & 0 & \mathrm{I} & \cdots & 0 & 0 \\
\vdots  & \vdots  & \mathrm{Z}   & \ddots  & \ddots & 0 & 0 \\ 
 \vdots & \vdots & \vdots & &  \ddots &  \mathrm{I} & 0 \\
0 & 0 & 0 & 0  \cdots &  \mathrm{Z} & 0 &  \mathrm{I} \\ 
2\mathrm{I}  & 0 & 0 & 0 & \cdots &  2\mathrm{Z} & 0 
  \end{matrix} \right) \\
&= \lambda_{j,r}\left[\mathbf{y}_{n-1}\; \mathbf{y}_{n-2} \;\,\ldots \;\,\mathbf{y}_1 \; \mathbf{y}_0\right]
 + \left[q^{(n-1)r}\Lc_{n-1}\mathbf{w}\; q^{(n-2)r}\Lc_{n-2}\mathbf{w}\;\ldots \; q^{r}\Lc_1\mathbf{w} \; \mathbf{w}\right] = \lambda_{j,r}\mathbf{y}_{j,r} + \mathbf{w}_{j,r},\end{split} \end{align}
where  $\mathbf{w}_{j,r}$ is as in Proposition \ref{left eigenvector},  and $\Lc_k = \Lc_k(q^j + q^{-j})$
for $k\ge 1$.
\end{proposition}

\begin{proof}  The proof entails comparing the entries on both sides of \eqref{eq:genleft} starting with column 0 (the rightmost)
and proceeding to column $n-1$ (the leftmost).  This involves noting that
\begin{align*} \mathbf{y}_1 &= \lambda_{j,r}\mathbf{y}_0 + \mathbf{w} = q^r \Uc_1 \mathbf{w} + \mathbf{w}, \\  
\mathbf{y}_2 &= \lambda_{j,r}\mathbf{y}_1-2\mathbf{y}_0\Zr  + q^r \Lc_1 \mathbf{w} \\ 
&= q^r (q^j + q^{-j})\big( q^r \Uc_1 \mathbf{w} + \mathbf{w}\big) - 2q^{2r}\mathbf{w} + q^r\Lc_1 \mathbf{w} 
= q^{2r}\left(\Uc_2-\Uc_0\right) \mathbf{w} + 2 \Uc_1\mathbf{w},\end{align*} 
and then arguing inductively for $k \ge 3$ using the relation $\Lc_k = \Uc_k-\Uc_{k-2}$ to show that
\begin{align*} \mathbf{y}_{k+1} &= \lambda_{j,r}\mathbf{y}_k - \mathbf{y}_{k-1} \Zr + q^{kr}\Lc_k \mathbf{w}\\
&= q^r (q^j + q^{-j})\Big( q^{kr}\big( \Uc_k-\Uc_{k-2}\big) + kq^{(k-1)r} \Uc_{k-1}\Big)  \mathbf{w}   \\
& \qquad -  q^{(k+1)r}\big( \Uc_{k-1}-\Uc_{k-3}\big) \mathbf{w} - (k-1)q^{kr} \Uc_{k-2}  \mathbf{w} 
+ q^{kr}\big(\Uc_k - \Uc_{k-2}\big) \mathbf{w}\\ &= q^{(k+1)r}\big(\Uc_{k+1} - \Uc_{k-1}\big)\mathbf{w} + kq^{kr} (q^j+q^{-j})\Uc_{k-1} \mathbf{w} 
 - (k-1)q^{kr} \Uc_{k-2}  \mathbf{w} + q^{kr}\big(\Uc_k - \Uc_{k-2}\big) \mathbf{w}\\
& = q^{(k+1)r}\big(\Uc_{k+1} - \Uc_{k-1}\big)\mathbf{w} + kq^{kr}\big(\Uc_k + \Uc_{k-2}\big) \mathbf{w}- (k-1)q^{kr} \Uc_{k-2}  \mathbf{w} + q^{kr}\big(\Uc_k - \Uc_{k-2}\big) \mathbf{w}\\
& = q^{(k+1)r}\big(\Uc_{k+1} - \Uc_{k-1}\big) \mathbf{w} +  (k+1)q^{kr}\Uc_k \mathbf{w}.
\end{align*}
It remains to check the leftmost column,  which involves verifying that
\begin{align*} \mathbf{y}_{n-2}\Zr +2 \mathbf{y}_0 -\lambda_{j,r}\mathbf{y}_{n-1} - q^{(n-1)r}\Lc_{n-1}\mathbf{w} = \mathbf{0}.
\end{align*}
Now by the right-hand side of the above calculation for $\mathbf{y}_{k+1}$ with $k = n-1$,  we have
\begin{align*} \mathbf{y}_{n-2}\Zr +2 \mathbf{y}_0 - \lambda_{j,r}\mathbf{y}_{n-1} - q^{(n-1)r}\Lc_{n-1}\mathbf{w}\\
& \hspace{-2.1cm} = - q^{nr} (\Uc_n - \Uc_{n-2})\mathbf{w} -nq^{(n-1)r}\Uc_{n-1}\mathbf w  + 2 \mathbf{w}\\
&  \hspace{-2.1cm} = - \Lc_n \mathbf{w}- \mathbf{0} + 2 \mathbf{w} = -(q^{nj}+q^{-nj})\mathbf{w} + 2 \mathbf{w} = \mathbf{0},
\end{align*}
since $\Uc_{n-1} \mathbf{w} = \frac{q^{nj}-q^{-nj}}{q^j-q^{-j}}\mathbf{w} = \mathbf{0}$ by Proposition \ref{prop:qrels} (a). 
\end{proof}

\subsection{Left eigenvectors from grouplike elements of $\Df_n$}\label{S3.9}
Next we determine the vector $\Trp(g)$ explicitly for $g= b^ic^k$, a 
grouplike element of $\Df_n$.  In contrast to the situation for right eigenvectors, only the  $n$ vectors 
$\Trp(b^{i}c^{-i})$, $i \in \ZZ_n$, are nonzero.  

We assume an ordering $\Pf_1,\Pf_2, \dots, \Pf_{n^2}$ of the projective indecomposable $\Df_n$-modules
$\Pf(\ell,r)$, $1\le \ell \le n-1$, $\VV(n,r)$, $r\in \ZZ_n$,  
first by $\ell$ and then by $r$.   Since the dual of the simple module $\VV(\ell,r)$ is $\VV(\ell, 1-r-\ell)$
by \cite[Thm.~4.3]{Jedwab}, we see  $\big(\VV(\ell,r)^*\big)^*
\cong \VV(\ell,r)$,  so that  $\left(\VV^*\right)^* \cong \VV$ holds for any finite-dimensional $\Df_n$-module.  
Using that fact,  we have from \eqref{eq:pgrouplike} that for every grouplike element of $\Df_n$
and every $\Df_n$-module $\VV$,  $\Trp(g)\McV  =  \tr_{\VV^*}(g)  \Trp(g)$, where
 $\Trp(g)=\left [
\tr_{\Pf_1}(g)\, \; \tr_{\Pf_2}(g) \,\; \dots \; \,\tr_{\Pf_{n^2}}(g) 
\right ]$.
Hence for every grouplike element of $\Df_n$,  $\Trp(g)$ is a left eigenvector of $\McV$ of eigenvalue
$\tr_{\VV^*}(g)$.    

To compute the vector $\Trp(g)$ for $g= b^ic^k$, we use the explicit description of the projective modules $\Pf(\ell,r)$ given in \cite[Lem.~2.1]{Chen2}, showing 
for $1 \le \ell < n$ that $\Pf(\ell,r)$
has a basis $\mathbf{p}_1,  \mathbf{p}_2, \ldots, \mathbf{p}_{2n}$  such that the actions of $b$ and $c$ are diagonal:  
$$(b^ic^k).\mathbf{p}_t = \begin{cases}
q^{(r+t-1)i + (t-r-\ell)k}\mathbf{p}_t & \text{ for } 1 \leq t \leq n, \\
q^{(r+t-1+\ell)i + (t-r)k} \mathbf{p}_t& \text{ for } n+1 \leq t \leq 2n.
\end{cases}$$
Thus, for $\ell < n$,  
$$\tr_{\Pf(\ell, r)}(b^ic^k) =
 q^{(r-1)i + (-r-\ell)k} \sum_{t=1}^{n}q^{t(i+k)}+ \,q^{(r-1+\ell)i -rk} \sum_{t=n+1}^{2n}q^{t(i + k)}.$$ 

As shown in \eqref{eq:trace}, we have for $\VV(n,r)$, 
 $$\tr_{\VV(n, r)}(b^ic^k) = q^{(r-1)i-rk} \sum_{t=1}^{n} q^{t(i+k)}.$$
 Observe that when $i+k \neq 0 \; (\modd n)$, the trace in these expressions is 0;
 consequently,  $\Trp(b^ic^k) = 0$, when $i+k \neq 0\; (\modd n)$.   Hence, we may assume that $k = -i\ (\modd n)$ and obtain
\begin{align}
\begin{split} \label{eq:ellnotn} 
\tr_{\Pf(\ell, r)}(b^ic^{-i}) &=
 \left(q^{(r-1)i + (r+\ell)i} + q^{(r-1+\ell)i +ri}\right)n = 2n q^{(2r+\ell-1)i}  \quad \qquad \text{for $\ell < n$,} \\
 \tr_{\VV(n, r)}(b^ic^{-i}) &= q^{(r-1)i+ri} \sum_{t=1}^{n} q^{t(i+-i)} = n q^{(2r-1)i}.
 \end{split}
 \end{align} 
 
 Except for the extra factor of 2 that occurs when $\ell < n$, the second equation in  \eqref{eq:ellnotn}  is the same as the first one
  with $\ell = n$. This motivates us to define for a fixed value of $i$, a vector with the following $n$ components:  
\begin{equation}\label{eq:pblocks} \mathsf{t}_{\ell,i} = \begin{cases}  \left[2n q^{(\ell -1)i}\,\;2n ^{(\ell+1)i}\;2n ^{(\ell+3)i}\,\;\dots\,\;2nq^{(\ell+n-3)i}\right ]  & \qquad \ell \ne n, \\
 \left[nq^{(n-1)i} \;\;\; \;nq^{i} \; \qquad nq^{3i}\, \quad \,\; \dots \quad\, \;  nq^{(n-3)i} \right]  & \qquad \ell = n. \end{cases}\end{equation}
 From the calculations above, we can conclude that the following holds.
\begin{proposition}  Assume $i\in \ZZ_n$, and let  $\mathsf{t}_{\ell,i}$ be as in \eqref{eq:pblocks} for $1\le \ell\le n$.    Then 
$$\Trp(b^ic^{-i}) = \left[\begin{matrix}\mathsf{t}_{1,i}&\mathsf{t}_{2,i}&\dots& \mathsf{t}_{n,i}\end{matrix}\right]$$
is a left eigenvector of the McKay matrix $\Mf_\VV$ corresponding to the eigenvalue  $\tr_{\VV^*}(b^ic^{-i})$  for any finite-dimensional 
$\Df_n$-module $\VV$.
\end{proposition} 
\begin{remark} 
In comparison to the right eigenvector situation in Theorem \ref{thm:rightevs} (d),   no nonzero left eigenvectors corresponding to
the eigenvalues $q^r(q^j + q^{-j})$ for $j \ne 0$ arise from evaluating projective characters on grouplike 
elements of $\Df_n$. \end{remark}

\begin{example}
When $\VV = \VV(2,0)$, we have $\VV^*= \VV(2,-1)$,  and \eqref{eq:action} tells us that the matrix of $b^i c^{-i}$ on $\VV^*$
relative to the basis $\{v_1, v_2\}$  is
$$\begin{pmatrix} q^{-1} & 0 \\ 0 & 1 \end{pmatrix}^i \begin{pmatrix} 1 & 0 \\ 0 & q \end{pmatrix}^{-i}\; =\;
\begin{pmatrix} q^{-i} & 0 \\ 0 & q^{-i} \end{pmatrix}.$$
Hence, $\tr_{\VV^*}(b^i c^{-i}) = 2q^{-i} =\lambda_{0,(n-1)i}.$  
Since these eigenvalues are distinct for $i \in \ZZ_n$, the left eigenvectors $\Trp(b^ic^{-i})$ for $\Mf_\VV$ are linearly independent. 
When $n = 3$, the vectors $\Trp(b^ic^{-i})$ and their corresponding eigenvalues are
\begin{equation}  \label{eq:n=3} \begin{array}{ccc}
&\Trp(1) = [6\;\; 6\;\; 6\;\; 6\;\;6\;\;6\;\;3\;\; 3\;\; 3]  &\quad \lambda_{0,0} = 2, \\
&\Trp(bc^{-1}) = [6\;\, 6q\;\, 6q^2\;\,6q^2\;\,6\;\,6q\;\,3q\;\,3q^2\;\, 3] &\quad \ \lambda_{0,1} = \lambda_{0,-2} = 2q, \\
&\Trp(b^2c^{-2} )=[6\;\,6q^2\;\,6q\;\,6q\;\, 6\;\,6q^2\;\,3q^2\;\,3q\;\,3] &\quad\ \ \lambda_{0,2} = \lambda_{0,-1} = 2q^2. 
\end{array}\end{equation}
\end{example}

 \subsection{Tensoring with $\VV(\ell,s)$}\label{S3.10} 
 Let $g = [\VV(1,1)]$  and $x = [\VV(2,0)]$ in the Grothendieck ring $\GG_0(\Df_n)$.
 Then $g^n = 1 = [\VV(1,0)]$ and $x^k = [\VV(2,0)]^k = [\VV^{\ot k}] = [\VV]^k$ in $\GG_0(\Df_n)$.    The following results are consequences of
 Lemmas 3.1-3.3 of \cite{ZWLC} and the tensor rules in Section \ref{S3.2}:
\begin{equation}\label{eq:Ggen}  [\VV(\ell,0)] = x [\VV(\ell-1,0)] - g [\VV(\ell-2,0)] \quad \text{for all} \; 3 \le \ell \le n,  \end{equation}
and since $[\VV(\ell,s)] = g^s[\VV(\ell,0)]$ for all $s \in \ZZ_n$,  it follows that $\GG_0(\Df_n)$ is generated by $g$ and $x$.   Moreover, 
\begin{equation}\label{eq:Vell}  [\VV(\ell,0)] = \sum_{i=0}^{\lfloor \frac{\ell-1}{2}\rfloor}  (-1)^i {{\ell-1-i} \choose {i}} g^i x^{\ell-1-2i} \quad \text{for all} \; 1 \le \ell \le n. \end{equation}
Thus, this defines a sequence of elements of $\GG_0(\Df_n)$ given by
\begin{align}\begin{split}\label{eq:fkxg} & f_0(x,g) =1 =[\VV(1,0)], \quad f_1(x,g) =x = [\VV(2,0)], \\ & f_{\ell-1}(x,g)\, =\, [\VV(\ell,0)]\,=\, \sum_{i=0}^{\lfloor\frac{\ell-1}{2}\rfloor}
(-1)^i {{\ell-1-i} \choose {i}} g^i x^{\ell-1-2i}, \qquad  2 \le \ell \le n,\end{split}\end{align}
and satisfying 
$$f_\ell(x,g) = x f_{\ell-1}(x,g)-g f_{\ell-2}(x,g) = \sum_{i=0}^{\lfloor\frac{\ell}{2}\rfloor}
(-1)^i {{\ell-i} \choose {i}} g^i x^{\ell-2i}, \qquad  \,2 \le \ell \le n.$$
In addition, if $f(x,g): = f_n(x,g)-gf_{n-2}(x,g)-2$, then by Lemma 3.3 of \cite{ZWLC}  we have
\begin{equation}\label{eq:minpoly}  f(x,g) = \sum_{i=0}^{\lfloor \frac{n}{2}\rfloor}  (-1)^i \frac{n}{n-i} {{n-i} \choose {i}} g^i x^{n-2i} \;  -2  = 0. \end{equation} 

If $R$ is the group algebra $\ZZ\GG$, where $\GG$ is the cyclic group generated by $g$,   then 
$\GG_0(\Df_n)  \cong R[x]/\langle f(x,g) \rangle$ is a commutative $\ZZ$-algebra of dimension $n^2$ with a basis given by $\{g^i x^k \mid i\in \ZZ_n,  0 \le k \le n-1\}$.  
Under the correspondence $1 \leftrightarrow \Ir = \Ir_n, \; t\Ir \leftrightarrow x$, \;$\Dr \leftrightarrow g$, we have that $\Uc_k(t,\Dr) \leftrightarrow f_k(x,g)$ for $0 \le k \le n$,
since the recursion relations are the same.   Moreover, $\mathsf{p}_n(t,\Dr) \leftrightarrow f(x,g)$ implies that 
\begin{equation}\label{eq:pntD2} \mathsf{p}_n(t,\Dr) = \sum_{i=0}^{\lfloor \frac{n}{2}\rfloor}  (-1)^i \frac{n}{n-i} {{n-i} \choose {i}} \Dr^i t^{n-2i} \;  -2  = 0. 
\end{equation} 
These considerations give the following result (compare \cite[Thm.~3.4]{ZWLC}). 

\begin{proposition} The Grothendieck ring  $\GG_0(\Df_n) \cong \ZZ[g,x]/ \langle g^n-1, f(x,g)\rangle.$
\end{proposition}

\begin{example} When $n = 11$,   \eqref{eq:minpoly} says
\begin{align*} f(x,g) & = x^{11} - \frac{11}{10}{{10} \choose 1} g x^9 + \frac{11}{9}{{9} \choose 2} g^2 x^7 - \frac{11}{8}{{8} \choose 3} g^3 x^5
+\frac{11}{7}{{7} \choose 4} g^4 x^3 - \frac{11}{6}{{6} \choose 5} g^5 x -2 \\
& = x^{11}-11gx^9 + 44 g^2 x^7 -77 g^3 x^5 + 55 g^4 x^3 - 11g^5x -2, \; \; \text{and \eqref{eq:pntD2} says} \\
\mathsf{p}_n(t,\Dr) &=  t^{11}-11\Dr t^9+44\Dr^2 t^7 - 77\Dr^3 t^5+55\Dr^4 t^3 -11\Dr^5 t - 2\Ir = 0 \;\; \text{(compare  \eqref{eq:pntD})}. \end{align*} 
\end{example} 

Now suppose that $\Mf_{(\ell,s)}$ is the McKay matrix for tensoring with $\VV(\ell,s)$.  
In computing matrices here, we order the rows and columns as usual,  first by $\ell$ and then by $s$, so that the order is $(1,0),(1,1), \dots, \allowbreak (1,n-1), (2,1), \dots, (2,n-1), \dots, (n,0), \dots, (n,n-1)$ as before.  
Let $\mathrm{Z}^{(s)} = \diag\{\mathrm{Z}^s, \mathrm{Z}^s, \dots, \mathrm{Z}^s\}$ ($n$ copies), where $\mathrm Z$ is the cyclic permutation  
matrix in \eqref{eq:McZ},  and let $\Ir_{n^2}$ be the $n^2 \times n^2$ identity matrix.  Assume  $\Mf = \Mf_{(2,0)}$, the McKay matrix for
tensoring with $\VV = \VV(2,0)$.    Then \eqref{eq:Vell} implies that
\begin{equation}\label{eq:Mellr} \Mf_{(\ell,s)} =\sum_{i=0}^{\lfloor \frac{\ell-1}{2}\rfloor}  (-1)^i {{\ell-1-i} \choose {i}} \Zr^{(i+s)} \Mf^{\ell-1-2i} \quad \text{for all} \; 1 \le \ell \le n, s \in \ZZ_n. \end{equation}
Here are a few special cases:
\begin{gather*} \Mf_{(1,0)} = \Ir_{n^2}, \qquad \Mf_{(2,0)} = \Mf, \qquad 
\Mf_{(3,0)} = \Mf^2 - \mathrm{Z}^{(1)}, \\
\Mf_{(4,0)}   = \Mf^3 - 2\mathrm{Z}^{(1)}\Mf, \qquad 
\Mf_{(5,0)} 
= \Mf^4 - 3  \mathrm{Z}^{(1)}\Mf^2+\mathrm{Z}^{(2)}. \end{gather*} 
 
\begin{corollary}\label{cor:sameev}  The  left and right (generalized) eigenvectors are the same for all the matrices $\Mf_{(\ell,s)}$, $1 \le\ell \le  n$, $s\in \ZZ_n$.
\end{corollary}

\begin{proof} Since $\Mf$ and $\mathrm{Z}^{(1)}$ commute, we can simultaneously upper-triangularize them and find a basis of common
right  (generalized) eigenvectors for them.   Similarly, we can simultaneously lower-triangularize  $\Mf$ and $\mathrm{Z}^{(1)}$ and find a basis of common
left  (generalized) eigenvectors.  These vectors will be common  left and right  (generalized) eigenvectors 
for all powers of $\Mf$ and $\mathrm{Z}^{(1)}$,  hence for all the matrices $\Mf_{(\ell,s)} = \mathrm{Z}^{(s)}\Mf_{(\ell,0)}$.  \end{proof}

\subsection{Eigenvalues for the McKay matrix of any simple $\Df_n$-module.} \label{S3.11}  
In this section, we use the results above to determine
the eigenvalues of the McKay matrix $\Mf_{(\ell,s)}$ for all $1 \le \ell \le n$ and $s \in \ZZ_n$.   
Recall from Proposition \ref{prop:rteigenvector} that $\vb_{j,r}=[\mathbf{v}_0\;\mathbf{v}_1\;\, \dots\;\,\mathbf{v}_{n-1}]^{\tt T}$
is a right eigenvector for $\Mf = \Mf_{(2,0)}$ with eigenvalue $\lambda_{j,r} =q^r(q^j + q^{-j})$ when $\mathbf{v}_0$ is a right eigenvector for $\Zr$ with  
eigenvalue $q^{2r}$,  and $\mathbf{v}_k = q^{kr}\Uc_k(q^j + q^{-j})\mathbf{v}_0$ for $1\le k\le n-1$.  It follows that $\Zr^{(s)}\vb_{j,r} = q^{2sr}
\vb_{j,r}$ for all $s \in \ZZ_n$.    As a consequence, we have the next result by combining \eqref{eq:Mellr} with \eqref{eqQkt}.

\begin{theorem}\label{thm:Ml0evals}\begin{itemize} \item[{\rm (a)}]  Assume $\vb_{j,r}=\left[\begin{matrix} \mathbf{v}_0&\mathbf{v}_1&\dots &\mathbf{v}_{n-1}\end{matrix}\right]^{\tt T}$ 
is a right eigenvector for $\Mf = \Mf_{(2,0)}$ with eigenvalue $\lambda_{j,r} =q^r(q^j + q^{-j})$ as in Proposition \ref{prop:rteigenvector}. Then 
\begin{align}\begin{split}\label{Ml0evals}  \Mf_{(\ell,0)}\vb_{j,r} &= 
\sum_{i=0}^{\lfloor \frac{\ell-1}{2}\rfloor}  (-1)^i {{\ell-1-i} \choose {i}} q^{2ir} q^{(\ell-1-2i)r}(q^j+q^{-j})^{\ell-1-2i}\vb_{j,r} \\
&=q^{(\ell-1)r} \sum_{i=0}^{\lfloor \frac{\ell-1}{2}\rfloor}  (-1)^i {{\ell-1-i} \choose {i}}(q^j + q^{-j})^{\ell-1-2i}\vb_{j,r} = q^{(\ell-1)r} \Uc_{\ell-1}(q^j + q^{-j}) \vb_{j,r}.
\end{split}\end{align} 
Hence, for all $1 \le \ell \le n$ and for all $s \in \ZZ_n$,  $\vb_{j,r}$ is a right eigenvector for $\Mf_{(\ell,s)} = \Zr^{(s)}\Mf_{(\ell,0)}$ with eigenvalue 
\begin{equation*}\label{Mlrevals} q^{(\ell-1+2s)r} \Uc_{\ell-1}(q^j + q^{-j})
= q^{(\ell-1+2s)r}\frac{q^{j\ell}-q^{-j\ell}}{q^j - q^{-j}}\; \; \text{when $j \ne 0$,}\end{equation*}
and with eigenvalue $q^{(\ell-1+2s)r}\ell$ when $j = 0$.

 \item[{\rm (b)}]  Assume
$\mathbf{w}_{j,r} =\left[\mathbf{w}_{n-1}\;\, \mathbf{w}_{n-2}\;\,\ldots \;\,\mathbf{w}_1\;\,\mathbf{w}_0\right]$  is a left eigenvector for $\Mf$ with eigenvalue $q^r(q^j + q^{-j})$ as in Proposition \ref{left eigenvector}.    Then $\mathbf{w}_{j,r}$
 is a left eigenvector for $\Mf_{(\ell,s)} = \Zr^{(s)}\Mf_{(\ell,0)}$ with eigenvalue $q^{(\ell-1+2s)r}(q^j + q^{-j})$
 for all $1 \le \ell \le n$ and for all $s \in \ZZ_n$.
\end{itemize}
\end{theorem}  
 \begin{proof}  Relation \eqref{Mlrevals} follows directly from \eqref{Ml0evals} and the fact that $\vb_{j,r}$ is a right eigenvector for
 $\Zr^{(s)}$ with eigenvalue $q^{2sr}$.  The last equality in \eqref{Ml0evals} comes from  Proposition \ref{prop:qrels} (a).
 Part (b) is similar.
 \end{proof}

\subsection{Eigenvectors for the projective McKay matrices of $\Df_n$.} \label{S3.12}  

Recall from Section \ref{S2.2} that the projective McKay matrix for tensoring with
a finite-dimensional module $\VV$ is given by $\Qf_\VV = (\Qf_{ij})$, where $\Qf_{ij} =
[\Pf_i \ot \VV:\Pf_j]$.   We have shown in  Theorem \ref{thm:QV}  that $\Qf_\VV = \Mf_{\VV^*}^{\tt T}$, where  $\Mf_{\VV^*}$
is the McKay matrix for tensoring with the dual module $\VV^*$.   In particular, since for  the simple $\Df_n$-module $\VV(\ell,s)$, we have $\VV(\ell,s)^* = \VV(\ell, 1-\ell-s)$,
we can conclude the following about $\Qf_{(\ell,s)}$ using  Theorem \ref{thm:Ml0evals}. 

\begin{proposition}\label{prop:Ql0evals} Assume $\vb_{j,r}$ and $\mathbf{w}_{j,r}$ are 
as in  Theorem \ref{thm:Ml0evals}. Then for all $1 \le \ell \le n$ and for all $s \in \ZZ_n$,  \begin{itemize} \item[{\rm (a)}]  
 $\vb_{j,r}^{\tt T} = \left[\mathbf{v}_0\;\,\mathbf{v}_1\;\,\dots \;\,\mathbf{v}_{n-1}\right]$ is a left eigenvector for $\Qf_{(\ell,s)}$ with eigenvalue 
\begin{equation}\label{Qlrevals} q^{(1-\ell-2s)r} \Uc_{\ell-1}(q^j + q^{-j})= q^{(1-\ell-2s)r}\frac{q^{j\ell}-q^{-j\ell}}{q^j - q^{-j}}\; \;  \text{when $j \ne 0$,} \end{equation}
and with eigenvalue $q^{(1-\ell-2s)r}\ell$  when $j = 0$.
 \item[{\rm (b)}] 
$\mathbf{w}_{j,r}^{\tt T} =\left[\mathbf{w}_{n-1}\; \mathbf{w}_{n-2} \;\,\ldots\;\,\mathbf{w}_1\; \mathbf{w}_0\right]^{\tt T}$  is a right eigenvector for $\Qf_{(\ell,s)}$  with eigenvalue as in (a)
 for all $1 \le \ell \le n$ and for all $s \in \ZZ_n$.
\end{itemize}
\end{proposition}   
 
 \subsection{Multiplication operators in Grothendieck algebras and idempotents.} \label{S3.13}  
 In Section \ref{S2.7}, we have described a method for constructing a left eigenvector for the McKay matrix $\McV$ using an eigenvector for the right multiplication operator $\mathsf{R}_\VV$ in the Grothendieck algebra of an arbitrary Hopf algebra.   The next result has a similar
 flavor but is stated in greater generality for an arbitrary finite-dimensional algebra $\mathrm A$.
 
 \begin{proposition}\label{prop:eigenidem} 
 Suppose $\mathrm A$ is an algebra with basis $\{b_1,b_2,\dots, b_k\}$, $k\in\ZZ_{\ge 1}$,  such that $b_i b_j = \sum_{t=1}^k a_{i,t}^{(j)} b_t$ 
for all $1 \le i,j \le k$.    Let $\Mf_j = \big (a_{i,t}^{(j)}\big )$ be the McKay matrix for multiplying by $b_j$ (on the right).  Assume
$\mathbf u = [u_1\;\,u_2 \;\, \dots\;\,u_k]$ is a common left eigenvector for the $\Mf_j$ with  $\mathbf u \Mf_j = \beta_j \mathbf u$ 
for all $j$.    Then $e_{\mathbf u} = \sum_{i=1}^k u_i b_i \in \mathrm A$ satisfies $e_{\mathbf u} ^2 = c_{\mathbf u} e_{\mathbf u} $,  where $c_{\mathbf u} = \sum_{j=1}^k \beta_ju_j$. 
\end{proposition} 

\noindent{\it Proof.} It follows from the relations $\mathbf u \Mf_j = \beta_j \mathbf u$ that $\sum_{i=1}^k  u_i a_{i,t}^{(j)} = \beta_j u_t$ for all
$1 \le j,t \le k$.       Then for $e_{\mathbf u} = \sum_{i=1}^k u_i b_i$, we have  
\begin{align*}\hspace{1.1truein} e_{\mathbf u}^2 = \left(\sum_{i=1}^k u_i b_i \right)^2  &= \left(\sum_{i=1}^k u_i b_i \right)\left(\sum_{j=1}^k u_j b_j \right) = \sum_{i,j=1}^k u_i u_j b_i b_j \\
& = \sum_{i,j,t=1}^k  u_i u_j  a_{i,t}^{(j)} b_t   
= \sum_{j,t = 1}^k u_j \left (\sum_{i=1}^k  u_i a_{i,t}^{(j)} \right) b_t  = \sum_{j,t = 1}^k u_j \beta_j u_t b_t \\
& = \left( \sum_{j=1}^k \beta_j u_j \right)\left( \sum_{t=1}^k u_t b_t \right)  =  c_{\mathbf u} e_{\mathbf u}.  
\hspace{1.85 truein} \square
\end{align*}  
\begin{corollary} 
Assume  $\mathrm{A} = \GG^{\CC}_0(\Df_n) = \CC \ot_{\ZZ} \GG_0(\Df_n)$, and let $b_1,b_2, \dots, b_{n^2}$ be an ordering of the simple modules
for the Drinfeld double $\Df_n$, first by $\ell = 1,\dots,n$, and then for a given value of $\ell$ by $r=0,1,\dots,n-1$, so that  $b_1, b_2, \dots, b_{n^2}$  is a basis for $\mathrm{A}$, and $b_1 = \VV(1,0)$ is the unit element of $\mathrm{A}$.   Let $\Mf_j$ be the McKay matrix for tensoring with $b_j$, and let $\mathbf{u}
= [u_1\;\, u_2\;\,\dots\;\,u_{n^2}]$ be
a nonzero common left eigenvector for matrices $\Mf_j$  (such exist
by Corollary \ref{cor:sameev}). 
Assume  $\mathbf{u}$ has eigenvalue $\beta_j$ relative $\Mf_j$ for all $j$.   Then $e_{\mathbf u} = \sum_{i=1}^{n^2} u_i b_i \in \mathrm A$ satisfies $e_{\mathbf u}^2 = c_{\mathbf u} e_{\mathbf u}$,  where $c_{\mathbf u} = \sum_{j=1}^{n^2} \beta_ju_j$, and 
when $c_{\mathbf u} = \sum_{j=1}^{n^2} \beta_j u_j \ne 0$, then $e_{\mathbf u} = c_{\mathbf u}^{-1}e_{\mathbf u}$
is an idempotent in $\mathrm{A}$. 
\end{corollary}

\begin{remark} 
In Proposition \ref{left eigenvector}, we have described $\frac{n(n+1)}{2}$ left eigenvectors corresponding to distinct eigenvalues of the McKay matrix $\Mf$ 
for tensoring with $\VV(2,0)$.  They correspond to eigenvectors for the right multiplication operator $\mathsf{R}_x$, $x = [\VV(2,0)]$, of
$\GG_0^{\CC}(\Df_n)$ as in Section \ref{S3.10}. They are common left eigenvectors for the $n^2$ McKay
matrices that result from 
tensoring with \emph{any} simple module $\VV(\ell,s)$. Each such left eigenvector $\mathbf u = [u_1\;\, u_2\;\,\dots\;\,u_{n^2}]$ with $c_{\mathbf u} \ne 0$  gives an idempotent 
$e_{\mathbf u} = c_{\mathbf u}^{-1} \sum_{i=1}^{n^2} u_i b_i$  in  $\GG_0^{\CC}(\Df_n)$.  Moreover, such idempotents
are distinct, $\mathbf u \neq \mathbf v \implies e_{\mathbf u} \neq e_{\mathbf v}$, because the $b_i$ determine a basis for $\GG_0^{\CC}(\Df_n)$. 
This suggests that we should be able to locate $\frac{n(n+1)}{2}$ linearly independent idempotents in $\GG_0^{\CC}(\Df_n)$.  We accomplish this
in Theorem \ref{thm:Groth} (b) below.  Moreover, we show in part (a) of that theorem, that there are $\frac{n(n-1)}{2}$ linearly independent
elements that square to 0 and  form a basis for the Jacobson radical of  $\GG_0^{\CC}(\Df_n)$. 
\end{remark} 

Our aim here is to identify a
$\CC$-basis of $\GG_0^{\CC}(\Df_n)$ consisting of elements that square to 0 and idempotents.
Our calculations will be based on the following well-known result  (see for example, \cite[Sec.~2.2]{FH}).
if $\CC\GG$ is the group algebra of a finite group and $\Sf$  is a simple $\GG$-module, then  
$$\varepsilon_{\Sf}:=  
\frac{\dimm(\Sf)}{|\GG|} \sum_{g \in \GG} \chi_{_{\Sf}}(g^{-1})  \, g$$  is a central idempotent in $\CC \GG$, and $\varepsilon_{\Sf}$ projects
any finite-dimensional $\GG$-module onto the $\Sf$-isotypic component.    Moreover, $\varepsilon_{\Sf} \varepsilon_{\mathsf{T}} = 0$ whenever
$\Sf$ and $\mathsf{T}$ are simple, nonisomorphic modules.  

Recall that  $\GG_0^{\CC}(\Df_n) \cong \CC \GG[x]/\langle f(x,g)\rangle$,  where 
$\GG = \langle g \rangle$,  
$$f(x,g) =  \sum_{i=0}^{\lfloor \frac{n}{2}\rfloor}  (-1)^i \frac{n}{n-i} {{n-i} \choose {i}} g^i x^{n-2i} \;  -2,$$
and $g = [\VV(1,1)]$, $x = [\VV(2,0)]$.    Therefore, it follows for $u \in \ZZ_n$ that $\ec_u= \frac{1}{n}\sum_{v=0}^{n-1} q^{-uv} g^v$ is an idempotent in $\CC\GG \subset \GG_0^{\CC}(\Df_n)$ corresponding to the
character $\chi_u(g^v) = q^{uv}$ of $\GG = \langle g \rangle$, and these idempotents are orthogonal $\ec_u \ec_{u'} = \delta_{u,u'}\ec_u$.  
Note that 
\begin{equation}\label{eq:get} g \ec_u= \frac{1}{n}\sum_{v=0}^{n-1} q^{-uv} g^{v+1} =  \frac{q^u}{n}\sum_{v=0}^{n-1} q^{-u(v+1)} g^{v+1}  = q^u \ec_u,\end{equation}
so that $\CC \GG = \bigoplus_{u=0}^{n-1} \CC \ec_u$ is a decomposition of the group algebra  $\CC \GG$ into simple $\GG$-modules,
where $\CC \ec_u$ is the one-dimensional $\GG$-module with corresponding character $\chi_u$.   

As a consequence of \eqref{eq:get}, we know that
\begin{equation}\label{eq:fxqt}  f(x,g) \ec_u = f(x, q^u) \ec_u = \mathsf{p}_n(x,q^u)\ec_u \end{equation}
in the notation of Corollary \ref{cor:evals}.   We can write
$u = 2r$ for some $r$, since $n$ is odd and $2$ is invertible modulo $n$, and then
\begin{equation}\label{eq:f=p}f(x,q^{2r}) = \mathsf{p}_n(x, q^{2r}) = (x-\lambda_{0,r}) \prod_{j=1}^{\frac{n-1}{2}}  (x-\lambda_{j,r})^2.\end{equation}

\begin{definition}\label{def:FG} For $r \in \ZZ_n$, let
\begin{align}\begin{split} \label{eq:fjrdef} \fc_{j,r} &:= \frac{f(x,q^{2r})}{x - \lambda_{j,r}}\ec_{2r}, \; \; \, \quad 0 \le j \le \frac{n-1}{2}, \\
 \gc_{j,r} &:= \frac{f(x,q^{2r})}{(x - \lambda_{j,r})^2}\ec_{2r}, \; \quad 1 \le j \le \frac{n-1}{2}.\end{split} \end{align}
\end{definition}

\begin{proposition}\label{prop:evalxg} The elements defined in \eqref{eq:fjrdef}  satisfy the 
relations
\begin{itemize} \item[{\rm (a)}]  
$x \mathcal{F}_{j,r} = \lambda_{j,r} \mathcal{F}_{j,r} \quad \text{and}\quad  g \mathcal{F}_{j,r} = q^{2r} \mathcal{F}_{j,r};$
\item[{\rm (b)}] $x \gc_{j,r} = \lambda_{j,r} \gc_{j,r}+\fc_{j,r} \quad \text{and}\quad  g \mathcal{G}_{j,r} = q^{2r} \mathcal{G}_{j,r}.$
\end{itemize}
\end{proposition}

\noindent{\it Proof.}  The fact that $g\fc_{j,r} = q^{2r}\fc_{j,r}$ and $g\gc_{j,r} = q^{2r}\gc_{j,r}$ hold is a consequence of $g \ec_{2r} = q^{2r}\ec_{2r}$.  
The relations involving $x$ follow easily from  
\begin{align*}\hspace{1.1 truein}(x -\lambda_{j,r})\mathcal{F}_{j,r} &= (x-\lambda_{j,r})\frac{f(x,q^{2r})}{x - \lambda_{j,r}}\ec_{2r} = f(x,q^{2r})\ec_{2r} = f(x,g)\ec_{2r} = 0,\\
(x -\lambda_{j,r})\mathcal{G}_{j,r} &= (x-\lambda_{j,r})\frac{f(x,q^{2r})}{(x - \lambda_{j,r})^2}\ec_{2r} = \fc_{j,r}. 
 \hspace{2 truein} \square \end{align*}  
\begin{proposition}\label{prop:zero} For the elements defined in \eqref{eq:fjrdef},  the following hold:
 \begin{itemize}
 \item[{\rm(a)}] $\fc_{j,r} \fc_{k,s} = \fc_{j,r} \gc_{k,s} = \gc_{j,r}\gc_{k,s} = 0$  when $ r \ne s$;
 \item[{\rm(b)}] The following products are 0 whenever $j \ne k$,
$$\fc_{j,r}\fc_{k,r}, \qquad  \fc_{j,r}\gc_{k,r}, \qquad \gc_{j,r}\gc_{k,r};$$
 \item[{\rm(c)}]  $\fc_{j,r}^2 = 0$ when $1 \le j \le \frac{n-1}{2}$;  
 \item[{\rm(d)}] The ideal $\mathcal{J} = \mathsf{span}_{\mathbb k}\{\fc_{j,r} \mid 1 \le j \le \frac{n-1}{2}, \, r \in \ZZ_n\}$ satisfies $\mathcal J^2 = (0)$.  
 \end{itemize}
 \end{proposition}
 \begin{proof} 
(a) Since $\ec_{2r} \ec_{2s} = 0$ whenever $r \ne s$,   it is apparent that   
 $\fc_{j,r} \fc_{k,s} = \fc_{j,r} \gc_{k,s} = \gc_{j,r}\gc_{k,s} = 0$ for $r \ne s$ 
and any choice of $j$ and $k$. 
(b) Suppose $m_j = 1$ when $j = 0$,  and $m_j = 1$ or $2$ when $1 \le j \le \frac{n-1}{2}$.       Then when $j \ne k$, we have
\begin{align}\begin{split}\label{eq:prod}  \frac{f(x,q^{2r})}{(x-\lambda_{j,r})^{m_j}}\ec_{2r} \cdot  \frac{f(x,q^{2r})}{(x-\lambda_{k,r})^{m_k}} \ec_{2r} & = 
 \frac{f(x,q^{2r})}{(x-\lambda_{j,r})^{m_j}(x-\lambda_{k,r})^{m_k}}\cdot  f(x,q^{2r}) \ec_{2r},   \\
 &=  \frac{f(x,q^{2r})}{(x-\lambda_{j,r})^{m_j}(x-\lambda_{k,r})^{m_k}}\cdot  f(x,g) \ec_{2r} = 0. \end{split} \end{align}
 Part (b) is now clear from the calculation in \eqref{eq:prod}.  For part (c),  observe that equation \eqref{eq:prod} holds when $k = j \ne 0$,
 and $m_j = m_k = 1$.     Part (d) follows from (b) and (c).  \end{proof} 

 \begin{proposition} \label{prop:idems} The elements $\{\xi_r^{-1}\mathcal{F}_{0,r} =\xi_r^{-1} \frac{f(x,q^{2r})}{x-\lambda_{0,r}}\ec_{2r} \mid r \in \ZZ_n \}$ are (nonzero) orthogonal idempotents,
 where $\xi_r = \prod_{j=1}^{\frac{n-1}{2}}\bigg(2q^{r} - q^r(q^j + q^{-j})\bigg)^2 \ne 0$. \end{proposition}
 
 \begin{proof}  Orthogonality is a consequence of Proposition \ref{prop:zero} (a), and the remaining assertions follow from
 \eqref{eq:f=p} and the calculation
 \begin{align*} \fc_{0,r}^2 &= \frac{f(x,q^{2r})}{x-\lambda_{0,r}}\ec_{2r} \cdot \frac{f(x,q^{2r})}{x-\lambda_{0,r}}\,\ec_{2r}  
 = \frac{f(x,q^{2r})}{x-\lambda_{0,r}} \fc_{0,r} \\ 
 &= \Bigg(\prod_{ j=1}^{\frac{n-1}{2}} (x - \lambda_{j,r})^2 \Bigg)  \fc_{0,r}  
=\Bigg(\prod_{j=1}^{ \frac{n-1}{2}} (\lambda_{0,r} - \lambda_{j,r})^2 \Bigg)   \fc_{0,r} \\
 &=\Bigg(\prod_{ j=1}^{\frac{n-1}{2}} \big(2q^r - q^r(q^j + q^{-j})\big)^2 \Bigg) \fc_{0,r}  = \xi_r \,\fc_{0,r},  \end{align*}
 where $\xi_r = \bigg(\prod_{ j=1}^{\frac{n-1}{2}} \big(2q^r - q^r(q^j + q^{-j})\big)^2 \bigg) \ne 0$,
 which implies $\xi_r^{-1} \fc_{0,r}$ is an idempotent. 
 \end{proof}

So far we have identified $n$ idempotents $\xi_{r}^{-1}\fc_{0,r}, r\in \ZZ_n$,   in $\GG_0^\CC(\Df_n)$ and
$\frac{n(n-1)}{2}$ elements $\fc_{j,r}$, $1 \le j \le \frac{n-1}{2}$, $r\in \ZZ_n$,  that square to 0.     
Next we will find some more idempotents using the elements $\gc_{j,r}$.    In Sec. \ref{S:3.13.2}, 
we will examine the elements $\fc_{j,r}$ and $\gc_{j,r}$ explicitly for $n=3$.

\subsubsection{Idempotents from the elements $\gc_{j,r}$}

We begin by computing $\gc_{j,r} \fc_{j,r}$ and $\gc_{j,r}^2$ for $\le j \le\frac{n-1}{2}, \, r\in \ZZ_n$.   Now
\begin{align*} \gc_{j,r} \fc_{j,r} &= \Bigg( \frac{f(x,q^{2r})}{(x-\lambda_{j,r})^2}\Bigg)\fc_{j,r}   
= (x-\lambda_{0,r})\Bigg(\prod_{ k=1, k\ne j}^{\frac{n-1}{2}} (x - \lambda_{k,r})^2 \Bigg)  \fc_{j,r}  \\
& = (\lambda_{j,r}-\lambda_{0,r})\Bigg(\prod_{ k=1, k\ne j}^{\frac{n-1}{2}} (\lambda_{j,r} - \lambda_{k,r})^2 \Bigg)  \fc_{j,r} =
\vartheta_{j,r} \fc_{j,r},  \qquad \text{where}
\end{align*}
\begin{equation}\label{eq:vtheta} \vartheta_{j,r} =  (\lambda_{j,r}-\lambda_{0,r})\Bigg(\prod_{ k=1, k\ne j}^{\frac{n-1}{2}} (\lambda_{j,r} - \lambda_{k,r})^2 \Bigg).\end{equation}
This shows that  $\gc_{j,r} \fc_{j,r}$ is a nonzero multiple $\vartheta_{j,r}$ of $\fc_{j,r}$. 
We also have
\begin{align}\label{eq:gsq} \gc_{j,r}^2 & =\Bigg( \frac{f(x,q^{2r})}{(x-\lambda_{j,r})^2}\ec_{2r}\Bigg)^2 
=  \frac{f(x,q^{2r})}{(x-\lambda_{j,r})^2} \gc_{jr} =\Bigg((x-\lambda_{0,r}) \prod_{k =1, k\ne j}^{\frac{n-1}{2}}  (x-\lambda_{k,r})^2\Bigg)\gc_{j,r}.
\end{align} 
For $1\le k \le \frac{n-1}{2}$, $k \ne j$, 
\begin{align*} (x-\lambda_{k,r})\gc_{j,r}  &= (\lambda_{j,r}-\lambda_{k,r})\gc_{j,r} + \fc_{j,r} \\ 
(x-\lambda_{k,r})^2\gc_{j,r}  &=(\lambda_{j,r}-\lambda_{k,r})(x-\lambda_{k,r}) \gc_{j,r} + (x-\lambda_{k,r})\fc_{j,r} \\
 &=(\lambda_{j,r}-\lambda_{k,r})^2  \gc_{j,r} + 2(\lambda_{j,r}-\lambda_{k,r})\fc_{j,r} \\
 (x-\lambda_{0,r}) \gc_{j,r} & = (\lambda_{j,r}-\lambda_{0,r})\gc_{j,r} + \fc_{j,r}.
\end{align*}
These computations and \eqref{eq:gsq} imply that 
$$\gc_{j,r}^2 = \vartheta_{j,r} \gc_{j,r} + \nu_{j,r} \fc_{j,r},$$
for some scalar $\nu_{j,r}$, where $ \vartheta_{j,r}\ne 0$ is as in \eqref{eq:vtheta}.   
Then 
\begin{align*}\big(\gc_{j,r}-\frac{\nu_{j,r}}{\vartheta_{j,r}} \fc_{j,r}\big)^2 &= \gc_{j,r}^2 - 2 \frac{\nu_{j,r}}{\vartheta_{j,r}}\gc_{j,r} \fc_{j,r}\\
&= \vartheta_{j,r} \gc_{j,r} + \nu_{j,r} \fc_{j,r} - 2 \frac{\nu_{j,r}}{\vartheta_{j,r}} \vartheta_{j,r} \fc_{j,r}  \\
&= \vartheta_{j,r} \gc_{j,r} - \nu_{j,r} \fc_{j,r} = \vartheta_{j,r} \bigg( \gc_{j,r}-\frac{\nu_{j,r}}{\vartheta_{j,r}} \fc_{j,r}\bigg),\end{align*}
which shows that 
\begin{equation}\label{eq:gcprime}\gc_{j,r}' : = \vartheta_{j,r}^{-1}\bigg(\gc_{j,r}-\frac{\nu_{j,r}}{\vartheta_{j,r}} \fc_{j,r}\bigg)\end{equation}
is an idempotent in $\GG_0^\CC(\Df_n)$ for $ \vartheta_{j,r}$ as in \eqref{eq:vtheta} and some $\nu_{j,r} \in \CC$ for each $1 \le j \le \frac{n-1}{2}$, $r \in \ZZ_n$.  

In summary, we have the following  
\begin{theorem}\label{thm:Groth} \begin{itemize} \item[{\rm (a)}]  The $\frac{n(n-1)}{2}$  elements $\fc_{j,r}$, $1 \le j \le \frac{n-1}{2}$, $r \in \ZZ_n$,
determine a $\CC$-basis for the Jacobson radical $\mathcal{J}$ of the Grothendieck algebra $\GG_0^{\CC}(\Df_n)$  
and $\mathcal{J}^2 = (0)$.  
\item[{\rm(b)}] The elements   $\xi_{0,r}^{-1}\fc_{0,r}$ and $\gc_{j,r}'$  for $r \in \ZZ_n$ and $1 \le j \le \frac{n-1}{2}$ 
are orthogonal idempotents, and they form a basis for $\GG_0^{\CC}(\Df_n)$ modulo $\mathcal{J}$.  
 \item[{\rm(c)}]  Suppose $\Sf_1,\Sf_2, \dots, \Sf_{n^2}$ is an ordering of the nonisomorphic simple $\Df_n$-modules, first by $\ell=1,\dots,n$
 and then by $r=0,1,\dots, n-1$.    If $\fc_{j,r} = c^{j,r}_1 \Sf_1 + c^{j,r}_2 \Sf_2+ \cdots +  c^{j,r}_{n^2} \Sf_{n^2}$, then
 $\mathsf{f}_{j,r}:= [c^{j,r}_1\;\,c^{j,r}_2\;\, \ldots\;\,c^{j,r}_{n^2}]$ is a left eigenvector for the McKay matrix $\McV$, $\VV = \VV(2,0)$, 
 for all $0 \le j \le \frac{n-1}{2}$ and $r \in \ZZ_n$ corresponding to the eigenvalue $\lambda_{j,r} = q^r(q^j + q^{-j})$.
 \item[{\rm(d)}]  If $\gc_{j,r} = d^{j,r}_1 \Sf_1 +d^{j,r}_2 \Sf_2 + \cdots +  d^{j,r}_{n^2} \Sf_{n^2}$, then
 $\mathsf{g}_{j,r}:= [d^{j,r}_1\;\,d^{j,r}_2\;\, \ldots\;\,d^{j,r}_{n^2}]$ is a generalized left eigenvector for the McKay matrix $\Mf_{\VV}$
 such that  $\mathsf{g}_{j,r} \McV = \lambda_{j,r}\mathsf{g}_{j,r} + \mathsf{f}_{j,r}$ 
 for all $1 \le j \le \frac{n-1}{2}$ and $r \in \ZZ_n$.
 \item[{\rm(e)}] The vectors $\mathsf{f}_{j,r}$, $0 \le j \le \frac{n-1}{2}, r \in \ZZ_n$,  give 
 a complete set of left eigenvectors,  and the 
 vectors $\mathsf{g}_{j,r}, 1 \le j \le \frac{n-1}{2}, r \in \ZZ_n$,   give a complete set of generalized left eigenvectors
for the McKay matrix $\McV$, $\VV = \VV(2,0)$, hence, for any McKay matrix $\Mf_{(\ell,s)}$ by Corollary \ref{cor:sameev}. 
 \end{itemize}
  \end{theorem} 
  \begin{proof} From Proposition \ref{prop:evalxg} (a) we know that the elements $\fc_{j,r}$ are eigenvectors for the multiplication operator $\mathsf{R}_x$, $x = [\VV(2,0)]$,
  of $\GG_0^{\CC}(\Df_n)$, so the corresponding coordinate vectors $\mathsf{f}_{j,r}$ relative to the basis of nonisomorphic simple modules will be left eigenvectors
  for $\Mf = \Mf_\VV$ by Proposition \ref{prop:rightmult}.  Part (b) of Proposition \ref{prop:evalxg} shows that $x \gc_{j,r} = \lambda_{j,r} \gc_{j,r} + \fc_{j,r}$.  Therefore,
  the coordinate vector $\mathsf{g}_{j,r}$ of $\gc_{j,r}$  will be a generalized left eigenvector for $\McV$ corresponding to the eigenvalue $\lambda_{j,r}$ by Proposition \ref{prop:rightmult}\,(b).  \end{proof}

\subsubsection{Computations for ${\Df_3}$} \label{S:3.13.2} \;\;
When $n = 3$,  we have $f(x,q^{2r}) = x^3 - 3 q^{2r}x -2$, and  
\begin{align*} \frac{f(x,q^{2r})}{(x-2q^r)} &= x^2 + 2q^r x + q^{2r} = (x+q^{r})^2, \quad \text{and} \quad 
\xi_r =   \big(2q^{r} - q^r(q + q^{-1})\big)^2 = (3q^r)^2 =9q^{2r}. \end{align*}
Thus,  $\xi_r^{-1}\fc_{0,r} =(9q^{2r})^{-1} \big(x^2 +2q^r x + q^{2r}\big) \ec_{2r} = \frac{1}{9}\big(q^r x^2 + 2q^{2r}x + 1\big) \ec_{2r}$
is an idempotent for $r=0,1,2$.    Now
\begin{equation}\label{fc1}\fc_{1,r} = \frac{f(x,q^{2r})}{x-\lambda_{1,r}}\ec_{2r} = (x-\lambda_{0,r})(x-\lambda_{1,r})\ec_{2r},\end{equation}  so that $\fc_{1,r}^2 = (x-\lambda_{0.r})^2 (x-\lambda_{1,r})^2 \ec_{2r} = (x-\lambda_{0,r})f(x,q^{2r})\ec_{2r} =  (x-\lambda_{0,r})f(x,g)\ec_{2r}= 0$. Consequently, the elements $\fc_{1,r}$, $r \in \ZZ_3$,  square to 0 in agreement with Proposition \ref{prop:zero} (b).    

Finally, 
$$\gc_{1,r} = \frac{f(x,q^{2r})}{(x-\lambda_{1,r})^2}\ec_{2r} = (x-\lambda_{0,r})\ec_{2r} \quad \text{and} \quad
\gc_{1,r}^2 = (x-\lambda_{0,r})\gc_{1,r} = (\lambda_{1,r}-\lambda_{0,r}) \gc_{1,r} + \fc_{1,r}.$$
This tells us that by taking $\nu_{1,r} = 1$ and $\vartheta_{1,r} = \lambda_{1,r}-\lambda_{0,r} = q^r(q+q^{-1})-2q^r = -3q^r$,
$$\bigg(\gc_{1,r} + \frac{1}{3q^r} \fc_{1,r}\bigg)^2 = -3q^r \gc_{1,r} + \fc_{1,r} -2 \fc_{1,r} = -3q^r \bigg(\gc_{1,r} + \frac{1}{3q^r} \fc_{1,r}\bigg),$$
and therefore by setting $\gc_{1,r}' =   -\frac{1}{3q^r}\big(\gc_{1,r} + \frac{1}{3q^r} \fc_{1,r}\big)$, we get an idempotent for $r \in \ZZ_3$.

Writing $(\ell,r)$ for $\VV(\ell,r)$ and recalling that $x = [\VV(2,0)]$ and $g = [\VV(1,1)]$ gives
\begin{align*}\xi_r^{-1}\fc_{0,r} & =(9q^{2r})^{-1} \big(x^2 +2q^r x + q^{2r}\big) \ec_{2r}  
=  \frac{1}{27}\big(q^r x^2 + 2q^{2r}x + 1\big)\big(q^{-4r}g^2 + q^{-2r}g + 1\big)\\
&\hspace{-.9cm}=  \frac{1}{27}\Big(q^r (3,0) + q^r(1,1) + 2q^{2r}(2,0) + (1,0)\Big)\Big(q^{2r} (1,2) + q^r (1,1) + (1,0) \Big)\\
&\hspace{-.9cm}=  \frac{1}{27}\Big((3,2) + q^{2r}(3,1) + q^r(3,0) + 2q^{r}(2,2) + 2(2,1) + 2q^{2r}(2,0)+2q^{2r}(1,2) + 2q^r(1,1) + 2(1,0)\Big).
\end{align*}
Ordering the summands from $(1,0)$ to $(3,2)$, ignoring the factor of $\frac{1}{27}$,  and recording the coefficients,  we have
$$\mathsf{f}_{0,r} = [2\;\,2q^r\;\, 2q^{2r}\;\, 2q^{2r}\;\, 2\;\, 2q^{r}\;\, q^r\;\, q^{2r}\;\,1].$$ 
Multiplying $\mathsf{f}_{0,r}$ by $3$ and then setting $r=0,1,2$, we obtain the left eigenvectors of $\McV$ for $\VV = \VV(2,0)$  in  \eqref{eq:n=3}
exactly:
\begin{equation*} \begin{array}{ccc}
&\Trp(1)\; = \,\;[6\;\; 6\;\; 6\;\; 6\;\;6\;\;6\;\;3\;\; 3\;\; 3]  &\quad \lambda_{0,0} = 2, \\
&\Trp(bc^{-1}) \,=\, [6\;\, 6q\;\, 6q^2\;\,6q^2\;\,6\;\,6q\;\,3q\;\,3q^2\;\, 3] &\quad \ \lambda_{0,1}= 2q, \\
&\Trp(b^2c^{-2} )=[6\;\,6q^2\;\,6q\;\,6q\;\, 6\;\,6q^2\;\,3q^2\;\,3q\;\,3] &\quad\ \ \lambda_{0,2} = 2q^2. 
\end{array}\end{equation*}
Now 
\begin{align*}\gc_{1,r} &= \frac{f(x,q^{2r})}{(x-\lambda_{1,r})^2}\ec_{2r} = (x-\lambda_{0,r})\ec_{2r} = (x-2q^r)\big(\frac {1}{3}(q^{2r} +q^rg + 1)\big) \\
 &=\frac{1}{3}\big(-2q^r(1,0) -2q^{2r}(1,1) -2(1,2) + (2,0) + q^r (2,1) + q^{2r}(2,2)\big). 
 \end{align*}
 Therefore, $\mathsf{g}_{1,r} = \frac{1}{3}[-2q^r, -2q^{2r}, -2, 1, q^r, q^{2r}, 0,0,0]$,   and $\lambda_{1,r} = q^r(q+q^{-1}) = -q^r$.  From
 \eqref{fc1}, we can deduce that $\mathsf{f}_{1,r} = \frac{1}{3}[-q^{2r},-1,-q^{r},-q^{r},-q^{2r},-1,1,q^r,q^{2r}].$  Then
 \begin{align*}\mathsf{g}_{1,r}\McV &=  \frac{1}{3}[-2q^r, -2q^{2r}, -2, 1, q^r, q^{2r}, 0,0,0] \left(\begin{matrix} 0 & \Ir & 0 \\
 \Zr & 0 & \Ir \\ 2 \Ir & 2 \Zr & 0 \end{matrix}\right)\\
 &= \frac{1}{3}[q^{2r},1,q^r, -2q^r, -2q^{2r}, -2, 1, q^r, q^{2r}]\\ &= -q^r \frac{1}{3}[-2q^r, -2q^{2r}, -2, 1, q^r, q^{2r}, 0,0,0] + 
 \frac{1}{3}[-q^{2r},-1,-q^{r},-q^{r},-q^{2r},-1,1,q^r,q^{2r}] \\
&= \lambda_{1,r}\mathsf{g}_{1,r}+\mathsf{f}_{1,r},
  \end{align*}
  so that $\mathsf{g}_{1,r}$ is a generalized left eigenvector for $\McV$ with eigenvalue $\lambda_{1,r}$ for $r \in \ZZ_3$.

\subsection{Fusion rules for tensoring a maximal set of independent projective modules in $\GG_0(\Df_n)$ 
with $\VV$} \label{S3.14}  We have seen in Proposition \ref{prop:CartanDn} that the Cartan map $\mathsf{c}$ for $\Df_n$ has rank
$\frac{n(n+1)}{2}$, and that the modules $\Pf(\ell,r) - \Pf(n-\ell, \ell+r)$ lie in the kernel of $\mathsf{c}$ for $1 \le \ell \le \frac{n(n-1)}{2}$.   Following \cite{CW1}, we let
 $\Nf_\VV$ be the matrix that records tensoring a projective module $\Pf$ with $\VV = \VV(2,0)$ and writing
the answer $[\Pf \ot \VV]$  as a $\ZZ$-combination of isomorphism classes of projectives whose images form a $\ZZ$-basis for
 $\mathsf{c}\big(\Kf_0(\Df_n)\big) \subseteq \GG_0(\Df_n)$.  Since the Cartan map has rank $\frac{n(n+1)}{2}$, we use only the modules 
 $\VV(n,r), \Pf(1,r), \dots, \Pf(\frac{n-1}{2},r)$ in forming $\Nf_\VV$.  
We assume that ordering and take all values of $r$ for each type, first for $\VV(n,r)$, then for $\Pf(1,r)$ etc..  From the tensor rules
\eqref{eq:tensrules}, we have that the resulting matrix $\Nf_\VV$ is $\frac{n(n+1)}{2} \times \frac{n(n+1)}{2}$ for any $n=2h+1$ with $h\ge 1$, and  
\begin{equation} \label{eq:NcKay}  
\Nf_\VV =\left( \begin{matrix} 0 & \mathrm{I} & 0  & & \cdots & 0 & 0 \\ 
 2\mathrm{Z} & 0 & \mathrm{I}  & & \cdots & 0 & 0 \\
 0 &  \mathrm{Z} & 0 & \mathrm{I} & \cdots & 0 & 0 \\
\vdots  & \vdots  & \mathrm{Z}   & \ddots  & \ddots & 0 & 0 \\ 
 \vdots & \vdots & \vdots & &  \ddots &  \mathrm{I} & 0 \\
0 & 0 & 0 &   \cdots &  \mathrm{Z} & 0 &  \mathrm{I} \\ 
0  & 0 & 0 & 0 & \cdots &  \mathrm{Z} & \Zr^{h+1}
  \end{matrix} \right), 
\end{equation}  
  where $\Ir$ is the $n \times n$ identity matrix, and $\Zr$ is the $n\times n$ cyclic matrix in \eqref{eq:McZ}.  
In this section, we show that the matrix $\Nf_\VV$ has right eigenvectors whose components involve 
the modified Chebyshev polynomials $\Lc_k(t)$ of Section \ref{S3.8}, and left eigenvectors whose
components involve the Chebyshev polynomials $\Vc_k(t)$ of the third kind \cite[Sec.~1.2.3]{MH},  which are defined by 
\begin{equation}\label{eq:VCheby} \Vc_0(t) = 1, \; \Vc_1(t) = t-1,  \; \, \Vc_k(t) = t \Vc_{k-1}(t) + \Vc_{k-2}(t), \; \; k \ge 2.
\end{equation} 
They are expressible in terms of other Chebyshev polynomials via the relations
\begin{equation}\label{eq:LV} \Vc_k(t) = \Uc_k(t) - \Uc_{k-1}(t), \;\, \text{and}\;\,   \Lc_k(t) = \Vc_k(t) + \Vc_{k-1}(t) \; \text {for all} \; k \ge 1. \end{equation}
The first identity can be found in \cite[1.17]{MH}, and the second comes from
$\Lc_k(t) = \Uc_{k}(t) - \Uc_{k-2}(t)$ and the first.   More specifically, we show the following for all $n=2h+1 \ge 3$:

\begin{theorem}\begin{itemize} \item [{\rm (a)}] The matrix $\Nf_\VV$ in \eqref{eq:NcKay}  has eigenvalues $\lambda_{j,r}
= q^r(q^j + q^{-j})$ for $0 \le j \le \frac{n-1}{2}, r \in \ZZ_n$, (each with
multiplicity one), so the matrix $\Nf_\VV$ is diagonalizable (as expected from \cite{CW1}).
 \item [{\rm (b)}]  Let $\vb \ne \mathbf{0}$ satisfy 
$\Zr  \vb = q^{2r } \vb$, and assume  $\Lc_k$ stands for  $\Lc_{k}(q^j + q^{-j})$.
Then $$[\vb\; \;  q^r \Lc_{1} \vb\; \;  \ldots \; \;  q^{hr} \Lc_{h} \vb]^{\tt T}$$
is a right eigenvector for $\Nf_\VV$ of eigenvalue $\lambda_{j,r}$   for $0 \le j \le \frac{n-1}{2}, r \in \ZZ_n$.  
\item [{\rm (c)}]  Let $\mathbf{w}\ne \mathbf{0}$ satisfy $\mathbf{w}\Zr = q^{2r}\mathbf{w}$, and assume $\Vc_k$ stands for $\Vc_k(q^j + q^{-j})$.  Then
$$[q^{hr} \Vc_{h} \mathbf{w}\; \; \ldots \: \; q^{r} \Vc_{1} \mathbf{w} \; \;  \mathbf{w}]$$
is a left eigenvector for $\Nf_\VV$ with eigenvalue $\lambda_{j,r}$ for $0 \le j \le \frac{n-1}{2}, r \in \ZZ_n$.  
\end{itemize}
\end{theorem}
\begin{proof} We will argue that (b) and (c) hold,  and part (a) will follow.  

(b) We compare both sides of this equation and verify they are indeed equal:
 \begin{equation*}
\label{eq:NcK}  
\left( \begin{matrix} 0 & \mathrm{I} & 0  & & \cdots & 0 & 0 \\ 
 2\mathrm{Z} & 0 & \mathrm{I}  & & \cdots & 0 & 0 \\
 0 &  \mathrm{Z} & 0 & \mathrm{I} & \cdots & 0 & 0 \\
\vdots  & \vdots  & \mathrm{Z}   & \ddots  & \ddots & 0 & 0 \\ 
 \vdots & \vdots & \vdots & &  \ddots &  \mathrm{I} & 0 \\
0 & 0 & 0 &   \cdots &  \mathrm{Z} & 0 &  \mathrm{I} \\ 
0  & 0 & 0 & 0 & \cdots &  \mathrm{Z} & \Zr^{h+1}
  \end{matrix}\right) \left [\begin{matrix}\vb \\ q^r \Lc_{1}\vb \\
  \\ q^{2r} \Lc_{2} \vb \\   \vdots \\ q^{(h-1)r} \Lc_{h-1}  \vb \\ q^{hr} \Lc_{h}  \vb\end{matrix}
  \right ]
  = \lambda_{j,r} \left [\begin{matrix} \vb  \\ q^r \Lc_{1}  \vb \\
  \\ q^{2r} \Lc_{2}  \vb \\   \vdots \\ q^{(h-1)r} \Lc_{h-1}  \vb \\ q^{hr} \Lc_{h} \vb \end{matrix}\right ].
\end{equation*} 
Row 0 on the left is $q^r\Lc_1 \vb  = q^r(q^j+q^{-j})\vb$, which equals  $\lambda_{j,r} \vb$, the entry in  row 0 on the right.
Row 1 says \; 
$2q^{2r}\vb + q^{2r} \Lc_2 \vb  = q^{2r}\left(\Lc_0 + \Lc_2\right) \vb
= q^{2r}(q^j + q^{-j})\Lc_1 \vb = \lambda_{j,r} q^r \Lc_1 \vb.$ \newline
\noindent Now for rows $2 \le s \le h-1$, we have
\begin{align*} \Zr q^{(s-2)r} \Lc_{s-2} \vb + q^{sr} \Lc_s \vb & = q^{sr} \left( \Lc_{s-2} + \Lc_s\right)\vb  
= q^{sr} (q^j + q^{-j}) \Lc_{s-1} \vb = \lambda_{j,r} q^{(s-1)r} \Lc_{s-1} \vb.
\end{align*}
Finally, for the last row,  recall that $h = \frac{n-1}{2}$ so that $h+1 = \frac{n+1}{2}$.   Then on the left we have
\begin{align*} \Zr q^{(h-1)r} \Lc_{h-1} \vb + \Zr^{h+1} q^{hr}\Lc_h \vb & =\left( q^{(h+1)r} \Lc_{h-1} +q^{\big(2(h+1)+h\big)r} \Lc_h\right)\vb  =q^{(h+1)r}\left( \Lc_{h-1} + \Lc_h\right)\vb.  
\end{align*}
On the right, the last entry is 
$$\lambda_{j,r}q^{hr} \Lc_h \vb = q^{(h+1)r} (q^j + q^{-j}) \Lc_h = q^{(h+1)r}\big( \Lc_{h+1} + \Lc_{h-1}\big)\vb.$$
So comparing the left and right sides, we see that   the argument boils down to whether $\Lc_h \vb = \Lc_{h+1}\vb$.
But since
$$\Lc_h(q^j+q^{-j}) = q^{hj} + q^{-hj} = q^{-(h+1)j} + q^{(h+1)j} = \Lc_{h+1}(q^j+q^{-j}),$$ the left and right sides are indeed equal,
so (b) holds.

(c) The connection with the Chebyshev polynomials $\Uc_k(t)$ in \eqref{eq:LV}  is the one that will be most helpful in proving part (c).
Recall we know by Proposition \ref{prop:qrels}(a) that for $t = x+x^{-1}$ and all $k \ge 1$, 
$$\Uc_k(t)  = x^{k} + x^{k-2} + \dots + x^{-(k-2)} + x^{-k} =  \frac{x^{k+1}- x^{-(k+1)}}{x - x^{-1}}.$$
Therefore,  
$$\Vc_k(t) = \Uc_k(t) - \Uc_{k-1}(t) =  \frac{x^{k+1}- x^{-(k+1)}}{x - x^{-1}} -  \frac{x^{k}- x^{-k}}{x - x^{-1}}$$ 
for all $k \ge 2$. In particular, taking $k = h+1$ for $h = \frac{n-1}{2}$, and assuming $x^n = 1$, we obtain
\begin{align}\begin{split}\label{eq:Vch1} \Vc_{h+1}(t) &= \frac{x^{h+2}- x^{-(h+2)}}{x - x^{-1}} -  \frac{x^{h+1}- x^{-(h+1)}}{x - x^{-1}} \\ 
&= \frac{x^{-(h-1)}- x^{h-1}}{x - x^{-1}}- \frac{x^{-h}- x^{h}}{x - x^{-1}}   
 =  \frac{x^{h}- x^{-h}}{x - x^{-1}} -  \frac{x^{h-1}- x^{-(h-1)}}{x - x^{-1}} = \Vc_{h-1}(t). \end{split}\end{align}

 We will use \eqref{eq:Vch1} and identify $x$ with $q$ when we argue that the following equation holds: 
  \begin{align*}
 [q^{hr} \Vc_{h} \mathbf{w}\;\; \ldots \;\; q^{r} \Vc_{1} \mathbf{w}\;\,  \mathbf{w}]  \left( \begin{matrix} 0 & \mathrm{I} & 0  & & \cdots & 0 & 0 \\ 
 2\mathrm{Z} & 0 & \mathrm{I}  & & \cdots & 0 & 0 \\
 0 &  \mathrm{Z} & 0 & \mathrm{I} & \cdots & 0 & 0 \\
\vdots  & \vdots  & \mathrm{Z}   & \ddots  & \ddots & 0 & 0 \\ 
 \vdots & \vdots & \vdots & &  \ddots &  \mathrm{I} & 0 \\
0 & 0 & 0 &   \cdots &  \mathrm{Z} & 0 &  \mathrm{I} \\ 
0  & 0 & 0 & 0 & \cdots &  \mathrm{Z} & \Zr^{h+1}
  \end{matrix}\right) = \lambda_{j,r} [q^{hr} \Vc_{h} \mathbf{w}\;\; \ldots \;\; q^{r} \Vc_{1} \mathbf{w}\;\,  \mathbf{w}].
  \end{align*}
Consider column $h$ on both sides (numbering columns $h$ to $0$ from left to right).    On the left we have 
$2 q^{(h-1)r} \Vc_{h-1} \mathbf{w}\Zr = 2 q^{(h+1)r} \Vc_{h-1}  \mathbf{w}$.   On the right we have for column $h$,  
$$\lambda_{j,r} q^{hr}\Vc_h  \mathbf{w} =  q^{(h+1)r}(q^j + q^{-j})\Vc_{h}  \mathbf{w} = q^{(h+1)r}\left(\Vc_{h+1}+\Vc_{h-1}\right) \mathbf{w}
= 2q^{(h+1)r} \Vc_{h-1}\mathbf{w}$$ by \eqref{eq:Vch1},  so the two are equal. 

Now for $s=h,\dots,2$,  column $s-1$ on the left gives $q^{sr}\Vc_s  \mathbf{w} + q^{(s-2)r}\Vc_{s-2}  \mathbf{w} \Zr =
q^{sr}\left(\Vc_s + \Vc_{s-2}\right)\mathbf{w} $.   The corresponding column on the right has entry 
$$\lambda_{j,r} q^{(s-1)r}\Vc_{s-1}\mathbf{w}  = q^{sr} (q^j + q^{-j})\Vc_{s-1} = q^{sr}\left(\Vc_s + \Vc_{s-2}\right)\mathbf{w},$$
so the two are equal.   Finally, for column 0, we have on the left 
\begin{align*} q^r \Vc_1 \mathbf{w}  + \mathbf{w} \Zr^{h+1} & = q^r \Vc_1\mathbf{w} + q^{2(h+1)r} \mathbf{w}  
= q^r (\Vc_1 + 1)\mathbf{w}  \\ & = q^r (q^j+q^{-j})\mathbf{w} = \lambda_{j,r}\mathbf{w}, \end{align*} 
which is precisely the entry in column 0 on the right-hand side.  

We have produced $\frac{n(n+1)}{2}$ right (and left) eigenvectors with distinct eigenvalues $\lambda_{j,r}$, for $ 0\le j\le \frac{n-1}{2}, r\in \ZZ_n$,  for the $\frac{n(n+1)}{2} \times \frac{n(n+1)}{2}$ matrix $\Nf_\VV$, so the $\lambda_{j,r}$
are exactly the eigenvalues of $\Nf_\VV$.  \end{proof}

\subsection{Further Questions}  In this paper, we have proven results on McKay matrices of arbitrary finite-dimensional Hopf algebras
and illustrated them for the Drinfeld double $\Df_n$ of the Taft algebra, but there remain many interesting open questions, even for semisimple Hopf algebras. 
  \begin{itemize}
  \item When is the McKay matrix symmetric or normal, hence orthogonally diagonalizable?  It is shown in \cite{W} that the McKay matrix corresponding to any simple module is orthogonally diagonalizable when $\Af$ is semisimple and almost cocommutative, and we have shown in Corollary \ref{cor:Q} that if $\Af$ is semisimple and $\VV$ is self-dual, then $\Mf_{\VV}$ is symmetric.  The McKay matrix $\McV$, $\VV = \VV(2,0)$,  for the nonsemisimple  Drinfeld double $\Df_n$ is not symmetric,  and for  the algebra $\Af$ that is (14) in Kashina's classification  \cite{K} of 16-dimensional semisimple Hopf algebras there is a module $\VV \not \cong \VV^*$ such that $\Mf_\VV$ is not
  symmetric.   
  \item For which Hopf algebras do the (generalized) right eigenvectors of McKay matrices correspond to columns in something 
 that can be regarded as a character table?
  
\item When can all the (right) eigenvectors of the McKay matrix $\Mf_\VV$ be obtained from traces of grouplike elements?  
We have shown this is possible for $\Df_n$ in Sec. \ref{S3.6}.  It is possible for Radford's Hopf algebra $\Af(n,m)$ which is 
also not semisimple \cite[Exer. 10.5.9]{Radford}, but fails to be true for the Kac-Palyutkin algebra
which is semisimple \cite{KP}.  We have seen in Sec. \ref{S3.9} that for  $\Df_n$,  only $n$ of the $\frac{n(n+1)}{2}$ linearly independent left eigenvectors can be realized as
trace vectors of grouplike elements on projective covers.
  \item  Under what assumptions can the (generalized) eigenvectors of McKay matrices be related to central elements of  the Hopf algebra $\Af$ or cocommutative elements in $\Af^*$?  
  \item  When are the eigenvalues of the fusion matrix $\Nf_\VV$ obtained by  tensoring a maximal independent set of indecomposable
  projective modules  with $\VV$ the same as the eigenvalues for the McKay matrix $\McV$?
    They are for Drinfeld double $\Df_n$ and $\VV = \VV(2,0)$. 
  \item What can be said about the (generalized) eigenvectors of matrices that encode the fusion relations  in the more general context of tensor or fusion categories (see e.g.  \cite{GKP}, \cite{Sh})?  In \cite{ENO}, Etingof, Nikshych, and Ostrik introduced the Frobenius-Perron dimension of a fusion category as the spectral radius of a matrix representing the fusion relations.   
 \item When $n$ is even, do the (generalized) eigenvectors of the McKay matrices for tensoring 
$\Df_n$-modules have expressions in terms of Chebyshev polynomials, and what can be said about the multiplicities of the eigenvalues?
   \end{itemize}


\end{document}